\documentclass[11pt,pdftex]{amsart} 

\usepackage{settings}

\begin{document}
\title{Allard's interior \texorpdfstring{$\varepsilon$}{E}-Regularity Theorem in Alexandrov spaces}

\author{Marcos Agnoletto, Julio C. Correa Hoyos, Márcio Fabiano da Silva and Stefano Nardulli}

\begin{abstract} 
In this paper, we prove Allard's Interior $\varepsilon$-Regularity Theorem for $m$-dimensional varifolds with generalized mean curvature in $L^p_{loc}$, $p > m$, in non-collapsed Alexandrov spaces with curvature bounded both from above and below. We first develop an intrinsic proof of the theorem for varifolds in Riemannian manifolds with metric tensor of class $\mathcal{C}^2$, without appealing to Nash's Isometric Embedding Theorem. This yields explicitly computable constants depending only on $m$, $n$, the double sided sectional curvature bounds, and the harmonic radius (or, equivalently, the injectivity radius). We then extend the result to Alexandrov spaces via the Approximation Theorem of Berestovskij and Nikolaev, where the explicit control of the constants in terms of the geometric data is required for the approximation argument.
\end{abstract}
\maketitle
\onehalfspacing
\tableofcontents
\singlespacing
\singlespacing
\section{Introduction}
The regularity theory of varifolds lies at the heart of geometric measure theory, providing the analytical foundation for understanding the fine structure of generalized surfaces arising as solutions in variational problems, e.g., soap film models, minimal surface theory, or more general prescribed mean curvature surfaces. Varifolds were introduced by F. J. Almgren in the 1960s as measure-theoretic surfaces capable of describing a wide range of geometrically relevant objects, including differentiable manifolds with locally finite area, integral currents, and soap film surfaces (see \cites{Almgren64,Almgren66}). Less than a decade later, W. K. Allard published his seminal paper \cite{Allard}, establishing that varifolds lying in the $n$-dimension Euclidean space satisfying suitable assumptions on density, mass ratio, and generalized mean curvature can be written locally as the graph of a $C^{1,\alpha}$ function (for some $\alpha \in ]0,1[$), along with sharp \textit{a priori} estimates on its $C^{1,\alpha}$ norm. This result, now known as Allard's Interior $\varepsilon$-Regularity Theorem, remains until today one of the cornerstones of the field.

A natural and compelling question is whether Allard's theorem extends beyond the Euclidean setting to curved ambient spaces in order to satisfy the needs of the working geometric  analyst. In this paper, we give an affirmative answer in a broad and geometrically significant class: we prove Allard's Interior $\varepsilon$-Regularity Theorem in non-collapsed Alexandrov spaces with curvature bounded both from above and below. Our approach is intrinsic --- built directly on the geometry of the ambient space --- and crucially relies on the Approximation Theorem of Berestovskij and Nikolaev \cite{Nikolaev-Berestovskij}*{Section 15.1}. This allows us to obtain explicitly constants depending only on the dimension of the varifold, the dimension of the ambient space, the sectional curvature bounds, and the harmonic radius.

The choice for an intrinsic approach is not merely aesthetic: explicit constants depending only on the sectional curvature bounds and the harmonic
radius reveal precisely how the geometry of the ambient space governs the regularity of the varifold and at which scale, rendering the theorem directly applicable in geometric contexts where such quantitative control is indispensable, e.g., geometric variational problems with varying ambient spaces and metrics but with a uniform lower bound on the injectivity radius and double sided bound on the sectional curvature. An alternative approach would be to transfer the problem to $\mathbb{R}^n$ via coordinate charts, paying the price to have an anisotropic version of the Allard's Euclidean theorem \cite{Allard}, and pull the result back to the ambient manifold. Although feasible, because of the special simple anisotropy that corresponds to Riemannian area functional written in coordinates, at the best of our knowledge there is no written trace of a such approach in the literature.  Moreover, this approach cannot yield the explicit dependence of the constants on intrinsic geometric data --- such as sectional curvature bounds and lower bound of injectivity radius of the ambient space  --- that our approach provides. Another natural strategy would be to use an isometric embedding of the ambient manifold into some Euclidean space $\mathbb{R}^N$, as already suggested in the introduction of \cite{Allard}*{p. 418} and was historically the way the geometric analysts have used the Allard's Interior $\varepsilon$-Regularity Theorem to tackle geometric variational problems inside Riemannian ambient manifolds. However this procedure has a drawback because it relies on the existence of such robust isometric embbeding into an Euclidean space which existence is in general guaranteed by the celebrated Nash's isometric embedding theorem. It is well known (see \cite{Jacobowitz72}) that a $\mathcal{C}^{k,\alpha}$ metric with $k + \alpha > 2$ admits an isometric embedding of the same class, but the existence of a $\mathcal{C}^2$ isometric embedding for metrics of class exactly $\mathcal{C}^2$ remains an open problem, as surveyed in \cite{Gromov17}*{Outrageous $\mathcal{C}^2$-Immersion Conjecture, p.~177}. In our setting, Alexandrov spaces or $\mathcal{C}^2$ Riemannian metrics, the existence of a isometric embedding is not guaranteed and so this strategy is not available to us at the present moment. Thus, to overcome this issue, the intrinsic method developed here fills this gap in a natural and self-contained way, working only intrinsically inside the abstract Riemannian ambient manifold $(M^n,g)$ and avoiding the use of an isometric embedding in some finite dimensional Euclidean space.

Complete statements with explicit dependence of all constants are given in \Cref{secondthmcomplete} for non-collapsed Alexandrov spaces with curvature bounded both from above and below, and in \Cref{maintheocomplete} for Riemannian manifolds with metric tensor of class $\mathcal{C}^2$. We state below a simplified version of Allard's Interior $\varepsilon$-Regularity Theorem for non-collapsed Alexandrov spaces with curvature bounded both from above and below. 
\begin{theo}\label{secondthm}
Let $m,n \in \mathbb{N}$ and $p \in \mathbb{R}$ such that $1 \leq m < n$ and $m < p < \infty$; in case $m=1$, we require that $p \geq 2$, $(X,g)$ be an $n$-dimensional non-collapsed Alexandrov space with curvature bounded both from above and below. Then there exist $\delta_0  > 0$ and $\rho_0  > 0$ such that, if $V \in \mathbf{V}_m(X)$ satisfies 
\begin{equation}
\Theta^m(||V||,x) \geq d > 0,
\end{equation}
for $||V||$-a.e. $x \in B_{d_g}(\xi,\rho)$,
\begin{equation}
\dfrac{||V||(B_{d_g}(\xi,\rho))}{\omega_m\rho^m} \leq d(1 + \delta),
\end{equation}
and
\begin{equation}
||\overrightarrow{H_g}||_{L^p_{loc}(||V||)} \leq \dfrac{d^{\frac{1}{p}}\delta}{\rho^{1-\frac{m}{p}}},
\end{equation}
where $\overrightarrow{H_g}$ is the generalized mean curvature of $V$ and $||V||$ is the weight of $V$, for some $0 < \delta < \delta_0$, for some $0 < d < \infty$, and for some $0 < \rho < \min\{\rho_0,\sqrt{\delta}\}$, then $\supp(||V||) \cap B_{d}(\xi,\varepsilon\rho)$ is a $C^{1,1-\frac{m}{p}}$ $m$-dimensional submanifold of $X$ with scaling invariant $C^{1,1-\frac{m}{p}}$ estimates.
\end{theo}
Before stating \Cref{secondcorol}, a consequence of \Cref{secondthm}, let us define the regular and singular sets of a varifold.
\begin{defi}\label{defi:sing}
Let $(X,g)$ be an $n$-dimensional non-collapsed Alexandrov space with curvature bounded both from above and below, $V \in \mathbf{V}_m(X)$ and $\xi \in \supp(||V||)$. We say that $\xi$ is a regular point of $V$ if there exists $\rho > 0$ such that $B_{d_g}(\xi,\rho) \cap \supp(||V||)$ is an $m$-dimensional $\C^1$ embedded submanifold of $X$. Furthermore, 
\begin{displaymath}
\mathrm{reg}(V) \defeq \left\{\xi \in \supp(||V||) : \xi \text{ is a regular point of } V \right\} \hspace{1cm} \text{and} \hspace{1cm} \mathrm{sing}(V) \defeq \supp(||V||) \setminus \mathrm{reg}(V).
\end{displaymath}
\end{defi}
With this definition in hand, \Cref{secondcorol} establishes that 
$\mathrm{reg}(V)$ is a relatively open dense subset of $\supp(||V||)$. We remark that this is essentially the best result currently available: already in $\mathbb{R}^n$, the optimal size of $\mathrm{sing}(V)$ --- even its Hausdorff dimension, or whether 
$\mathcal{H}^m(\mathrm{sing}(V)) = 0$ --- remains unknown for general stationary varifolds, and constitutes one of the main open problems in the regularity theory of varifolds.
\begin{coro}\label{secondcorol}
Let $m,n \in \mathbb{N}$ and $p \in \mathbb{R}$ such that $1 \leq m < n$ and $m < p < \infty$; in case $m=1$, we require that $p \geq 2$, $(X,g)$ be an $n$-dimensional Alexandrov space with double-sided bounded intrinsic sectional curvature and $V \in \mathbf{V}_m(X)$ with generalized mean curvature vector $\overrightarrow{H_g} \in L^p_{\mathrm{loc}}(X,||V||)$ and $\Theta^m(||V||,x)\in\mathbb{N}\setminus\{0\}$, for $||V||$-a.e. $x \in \supp(||V||)$. Then $\mathrm{reg}(V)$ is a relatively open dense subset of $\supp(||V||)$, i.e., $\mathrm{sing}(V)$ is nowhere dense in $\supp(||V||)$ and $\supp(||V||) = \overline{\mathrm{reg}(V)}$.
\end{coro}
\begin{rema}
Another consequence of \Cref{maintheocomplete} is that the hypothesis of $\mathcal{C}^4$ bounded geometry in \cite{NarCalcVar} can be relaxed to $\mathcal{C}^2$ bounded geometry, at the cost of weaker control on the remainder terms of order strictly greater than two in the Puiseaux expansion of the isoperimetric profile for small volumes.
\end{rema}
We conclude the introduction with an overview of the organization of the paper, 
highlighting the main results of each section.
\subsection*{Structure of this article}
This paper is divided into three main parts. In \Cref{part:generalvarifolds}, we establish the basic framework for the theory of intrinsic varifolds on the ambient of abstract Riemannian manifolds, fix notation, and recall fundamental definitions. Specifically  we begin by defining varifolds intrinsically and obtaining the first variation of a varifold in a Riemannian manifold. The first place where the adaptations to the Riemannian setting become substantial is  estimating $\operatorname{div}_{S}(u_y \grad{g} u_y)$ (see \Cref{lemm:bounds:section:monotonicity}), which requires comparison geometry to be controlled in terms of sectional curvature bounds. Using these estimates, we prove the monotonicity formula (see \Cref{lemm:mono:section:monotonicity}) and derive its consequences, including the existence and upper semicontinuity of the density $\Theta^m(||V||,y)$ for varifolds with $L^p_{\mathrm{loc}}$-bounded generalized mean curvature (see \Cref{coro:densLp:section:monotonicity}). Furthermore we prove the Rectifiability Theorem (see \Cref{theo:rectifiable}).

In \Cref{part:allardreg}, we develop the intrinsic 
notion of excess for varifolds in Riemannian manifolds and carry out the proof of the regularity theorem. We begin with a discussion of Fermi coordinates and harmonic $m$-submanifolds, the key new geometric ingredient of our approach. Since  $\mathcal{C}^2$ metrics do not provide sufficient regularity for normal coordinates alone \cite{Jost84}*{Section 2.8}, harmonic submanifolds offer better-controlled coordinate systems. We prove the Caccioppoli inequality 
(\Cref{lemm:excess:section:DensityEstimatesApp}) with our intrinsic excess. The curvature of the harmonic submanifold contributes additional terms when estimating $\operatorname{div}_{T_{x}V}(u_{\Sigma} \nabla{g}(u_{\Sigma}))$ (see \Cref{equa4proof:lemm:excess:section:DensityEstimatesApp}), which are controlled via the principal curvatures of the harmonic submanifold. We then introduce the main hypotheses on varifolds carried through the rest of the paper (\Cref{defi:boundsratio:section:DensityEstimatesApp}) and establish the density estimate (\Cref{lemm:density:section:DensityEstimatesApp}) to obtain a first decay estimate for the excess (\Cref{coro:Remark212:section:DensityEstimatesApp}) from harmonic submanifolds and the parameter $\rho_0$. Finally we introduce the notion of normal graph and prove the Lipschitz Approximation Lemma (\Cref{lemm:lipschitz:section:DensityEstimatesApp}). 

We note that the proofs of many key results follow arguments parallel to those in 
\cite{Simon}, with two notable exceptions: the Caccioppoli inequality 
\cite{Simon}*{Lemma 2.5, p. 131} (see \Cref{lemm:excess:section:DensityEstimatesApp}), where the main difference arises from our intrinsic definition of excess (see \Cref{defi:excess}); and the Lipschitz approximation \cite{Simon}*{Lemma 2.11, p. 135} (see \Cref{lemm:lipschitz:section:DensityEstimatesApp}), where instead of the graph of a Lipschitz function one must work with normal graphs (see \Cref{defi:normalgraph:section:DensityEstimatesApp}).

In \Cref{sec:metric}, we conclude the paper by proving Allard's Interior $\varepsilon$-Regularity Theorem in non-collapsed Alexandrov spaces with curvature bounded both from above and below. 
The argument proceeds by approximating the metric of the Alexandrov space by smoother Riemannian metrics via the Berestovskij--Nikolaev theorem, applying \textit{a priori} estimates given by the intrinsic regularity theory developed in \Cref{part:allardreg} to each approximation, and then passing to the limit. The explicit dependence of all constants on the geometric data, established throughout \Cref{part:allardreg}, is what makes this limiting argument viable.
\subsection*{Acknowledgments}
M. Agnoletto is partially supported by Coordenação de Aperfeiçoamento de Pessoal de Nível Superior – Brasil (CAPES) – Finance Code 88887.667684/2022-00 - and by Conselho Nacional de Desenvolvimento Científico e Tecnológico - Brasil (CNPq) - Finance Code 201543/2024-9.\\
\indent S. Nardulli was supported by FAPESP Auxílio Jovem Pesquisador \# 2021/05256-0, CNPq  Bolsa de Produtividade em Pesquisa 1D \# 12327/2021-8, 23/08246-0, Geometric Variational Problems in Smooth and Nonsmooth Metric Spaces \# 441922/2023-6. \\
\indent J. Correa has received partial support from CNPq-Brazil under Grant No. 408169/2023-0. \\
\indent The authors would like to thank Reinaldo Resende for his valuable discussions, suggestions, and support throughout this work. The authors also would like to thank Gabriel Augusto Correia and Bruno Souza dos Santos de Almeida for their suggestions. This paper was written as part of the first author's Ph.D. thesis at the Federal University of ABC, under the supervision of Prof. Stefano Nardulli.

\section{Theory of general varifolds in abstract Riemannian manifolds}\label{part:generalvarifolds}
\subsection{Intrinsic definition of varifolds in Riemannian manifolds}\label{section:definitionvarifolds}
Let $m,n \in \mathbb{N}$ such that $1 \leq m < n$. We denote by $(M^n,g)$ an $n$-dimensional complete connected Riemannian manifold with differentiable structure of class $\C^3$ and  abstract metric tensor $g$ of class $\C^2$. We also denote by $\Gr_m(TM^n)$ the Grassmannian $m$-plane bundle of $M^n$ and 
\begin{displaymath}
\Gr_m(M^n) \defeq \left\{(x,S) \in M^n \times \Gr_m(TM^n) : S \in \tilde\pi^{-1}(x)\right\},
\end{displaymath}
where $\tilde\pi : \Gr_m(TM^n) \to M^n$ is the natural fiber bundle projection.
\begin{defi}
We say that $V$ is an $m$-dimensional varifold in $M^n$, if $V$ is a nonnegative real-extended valued Radon measure on $\Gr_m(M^n)$. We denote by $\mathbf{V}_m(M^n)$ the space of all $m$-dimensional varifolds in $M^n$.
\end{defi}
\begin{defi}
Let $V \in \mathbf{V}_m(M^n)$ and $\pi:\Gr_m(M^n) \to M^n$ the projection onto the first factor. The weight of $V$, denoted by $||V||$, is the measure given by the push-forward of $V$ by $\pi$.
\end{defi}
We fix $\injec{r0} \in \mathbb{R}$ such that 
\begin{injectivity}\label{r0}
0 < \injec{r0} < \frac{\inj(M^n,g)}{4},    
\end{injectivity}

where $\inj(M^n,g)$ is the injectivity radius of $(M^n,g)$.

Let $\Gamma \subset M^n$ be a countably $m$-rectifiable set and $\theta \in L^1_{loc}(\Gamma,]0,\infty[)$. For every $B \in  \mathcal{B}(\Gr_{m}(M^n))$,  we define
\begin{displaymath}
V(\Gamma,\theta,g)(B) \defeq \int_{\left\{x\in\Gamma:(x, T_x\Gamma)\in B\right\}}\theta d\mathcal{H}_{g}^m,
\end{displaymath}
where $\mathcal{H}_{g}$ is the Hausdorff measure in $M^n$ with respect to the metric $g$.
\begin{defi}
We say that $V \in \mathbf{V}_m(M^n)$ is a rectifiable $m$-varifold in $M^n$ when there exist a countably $m$-rectifiable set $\Gamma \subset M^n$ and $\theta \in L^1_{loc}(\Gamma,]0,\infty[)$ such that $V = V(\Gamma,\theta,g)$. Furthermore, when $\theta:\Gamma\to\mathbb{N}$ we say that $V$ is an integer $m$-varifold in $M^n$. We denote by $\mathbf{RV}_m(M^n)$ the space of $m$-rectifiable varifolds in $M^n$ and by $\mathbf{IV}_m(M^n)$ the set of all integer $m$-varifolds in $M^n$.
\end{defi}
\subsection{First variation of a varifold}\label{section:firstvariation}
Let $k \in \mathbb{N}$. We denote by $\mathfrak{X}^k_{c}(M^n)$ the set of $k$-differentiable vector fields on $M^n$ with compact support, by $\mathfrak{X}^0_{c}(M^n)$ the set of continuous vector fields on $M^n$ with compact support, and by $\mathfrak{X}^{\infty}_{c}(M^n)$ the set of smooth vector fields on $M^n$ with compact support.

Furthermore, we denote by $\nabla^g$ the Levi-Civita connection on $M^n$, and, for $x \in M^n$ and $S \in \Gr(m,T_xM^n)$, we denote by $P_S: T_xM^n \to S$ the orthogonal projection of $T_xM^n$ onto $S$.
\begin{defi}
Let $V\in\mathbf{V}_m(M^n)$. We define the first variation of the varifold $V$ as the linear functional $\delta_gV : \mathfrak{X}^1_{c}(M^n) \to \mathbb{R}$ given by
\begin{displaymath}
\delta_gV(X) \defeq \int_{\Gr_m(M^n)} \diver_S(X)(x) dV(x,S),
\end{displaymath}

where $\diver_S(X)$ denotes the tangential divergence of $X$ along the $m$-plane $S$.
\end{defi}
\begin{defi}
Let $d,d_1,d_2 \in \mathbb{N}$ such that $d < \min\{d_1,d_2\}$, $(N_1,g_1)$ and $(N_2,g_2)$ be $d_1$ and $d_2$-dimensional Riemannian manifolds, respectively, $V \in \mathbf{V}_d(N_1)$, and $f: N_1 \to N_2$ a proper $\C^1$ map. We define the push-forward of $V$ by $f$, for every $B \in \mathcal{B}(\Gr_d(N_2))$, by
\begin{displaymath}
f^{\#}V(B) \defeq \int_{\{(x,S) \in \Gr_d(N_1) : (f(x),df_x(S)) \in B\}} \jac_S(f)(x) dV(x,S),
\end{displaymath}

where $\jac_S(f)$ denotes the tangential Jacobian of $f$ along the $m$-plane $S$.
\end{defi}
Given $X \in \X^1_c(M^n)$, the one-parameter family of diffeomorphisms generated by $X$ is the map $\Phi : \mathbb{R} \times M^n \to M^n$, which is, for all $(t,x) \in \mathbb{R} \times M^n$, the unique solution of
\begin{equation}\label{onepatameterdiff}
\begin{cases}
\Phi(0,x) &= x \\
\left.\dfrac{\partial \Phi}{\partial t}\right|_{t=0} &= X(x).
\end{cases}
\end{equation}

Note that the flow $\Phi(\cdot,x)$ is defined on all $\mathbb{R}$ and any $x\in M^n$ because $X$ is assumed to have compact support in $M^n$.
\begin{prop}\label{prop:firstvariationone}
Let $V \in \mathbf{V}_m(M^n)$ and $X \in \X^1_c(M^n)$. Then
\begin{equation}\label{equa1:firstvariation}
\delta_g V(X)=\left.\dfrac{d}{dt}\right|_{t=0}\left(||{\Phi_t}^{\#}V||(M^n)\right),
\end{equation}

where $\Phi_t(\cdot) \defeq \Phi(t,\cdot)$ and $\Phi$ is the one-parameter family of diffeomorphisms generated by $X$.
\end{prop}
\begin{proof}
Let $x \in M^n$ and $(U,\varphi)$ be a normal coordinate chart centered at $x$. By \cite{Lee}*{Proposition 5.24, p. 123} we have that 
\begin{enumerate}
    \item $\phi(x) = (\phi^1(x),\ldots,\phi^n(x)) = (0,\ldots,0)$, where $\phi^i:M^n \to \mathbb{R}$ for all $i = 1,\ldots,n$;
    \item The components of the metric at $x$ are $g_{ij}(x) = \delta_{ij}$ for all $i,j = 1,\ldots,n$;
    \item $\Gamma_{ij}^k(x) = 0$ for all $i,j,k = 1, \ldots, n$;
    \item $\frac{\partial g_{ij}}{\partial \phi^k}(x) = 0$ for all $i,j,k = 1, \ldots, n$.
\end{enumerate}
Let $\Psi:\mathbb{R} \times M^n \to M^n$ be given by $\Psi(t,x) \defeq \exp_x^g(tX(x))$ and note that
\begin{displaymath}
\Psi(0,x) = \exp_x(0X(x)) = x
\end{displaymath}

and
\begin{displaymath}
\left.\dfrac{\partial \Psi}{\partial t}\right|_{t=0} = \left.\dfrac{\partial}{\partial t}\right|_{t=0} \exp_x(tX(x)) = X(x).
\end{displaymath}

So, $\Psi$ is a solution to \Cref{onepatameterdiff}, and by the uniqueness of solutions for ODEs, it is the only one. In this case, we denote it by $\Phi$.

Since $g$ is of class $\C^2$ we have, by a standard argument of mollification of the Riemannian metric $g$ (see \cite{Fukuoka2006}) for all $t>0$ sufficiently close to $0$, for all $(x,S) \in \Gr_m(M^n)$, and for all $i,j = 1,\ldots,m$, that
\begin{equation}\label{equa0:firstariation}
\left[\mathrm{Jac}_S(\Phi_t)(x)\right]_{ij} = \delta_{ij} + t\left(\left\langle \nabla_{e_i}X,e_j\right\rangle_{g_x} + \left\langle e_i , \nabla_{e_j}X\right\rangle_{g_x} + E_{ij}(\Phi_t(x))\right),
\end{equation}

where $\{e_1,\ldots,e_m\}$ is an orthonormal basis for $S$ and $\sup_{x \in \supp(X)} E_{ij}(\Phi_t(x)) \to 0$ when $t \downarrow 0$.

By \cite{FedererGMT}*{Section 1.4.5, p. 21} and \Cref{equa0:firstariation} we have, for all $t>0$ sufficiently close to $0$, and for all $(x,S) \in \Gr_m(M^n)$, that
\begin{align*}
\jac_S(\Phi_t)(x) = \sqrt{1 + t\left(2\diver_S(X)(x) + E_{ii}(\Phi_t(x))\right) + O(t^2)}.
\end{align*}

Which implies, for all $(x,S) \in \Gr_m(M^n)$, that
\begin{align}\label{eq3:section:first-variation}
\left.\dfrac{d}{dt}\right|_{t=0} \jac_S(\Phi_t)(x) = \diver_S(X)(x).
\end{align}

So, by \Cref{eq3:section:first-variation} and \cite{Folland}*{Theorem 2.27, p. 56}, we have that \Cref{equa1:firstvariation} holds.
\end{proof}
\begin{rema}
We initially defined the first variation of a varifold as a linear functional on $\mathfrak{X}^1_c(M^n)$. We now extend this definition to the space $\mathfrak{X}^0_c(M^n)$ for varifolds that has $L^p_{\mathrm{loc}}(||V||)$-bounded first variation in $M^n$ (as a linear functional), for every $1 \leq p \leq \infty$.

We define $\widetilde{\delta_gV}:\mathfrak{X}^0_c(M^n) \to \mathbb{R}$ by
\begin{displaymath}
\widetilde{\delta_gV}(X) \defeq \lim\limits_{\varepsilon \downarrow 0} \delta_gV(X_{\varepsilon}),
\end{displaymath}

where $X_{\varepsilon}$ is a smoothing mollifier of the vector field $X \in \mathfrak{X}^0_c(M^n)$. 

Note that $\widetilde{\delta_gV}$ is well-defined and is equal to $\delta_g V$ on $\mathfrak{X}^1_c(M^n)$. From now on we identify $\widetilde{\delta_gV}$ with $\delta_gV$, although technically they may differ.
\end{rema}
\begin{defi}\label{defi:stationaryvariation:section:firstvariation}
We say that $V \in \mathbf{V}_m(M^n)$ is stationary when $\delta_g V(X)=0$ for any $X \in\mathfrak{X}^0_c(M^n)$.
\end{defi}
\begin{defi}\label{defi:locallyboundedfirstvariation:section:firstvariation}
We say that $V \in \mathbf{V}_m(M^n)$ has locally bounded first variation when, for every open subset $W \subset \subset M^n$, there exists a real constant $C_W > 0$ such that 
\begin{displaymath}
|\delta_gV(X)| \leq C_W ||X||_{L^{\infty}(W,||V||)},
\end{displaymath}

holds for every $X \in \mathfrak{X}^0_c(W)$.
\end{defi}
\begin{prop}\label{prop:generalizedmeancurvaturaL1bounded:section:firstvariation}
Let $V\in\mathbf{V}_m(M^n)$ with locally bounded first variation. Then the total variation of $\delta_gV$, denoted by $||\delta_gV||_{\mathrm{TV}}$, is a Radon measure. Furthermore, there exists a $||V||$-measurable function $\overrightarrow{H_g}: M^n \to TM^n$, a real constant $C_W > 0$ and $Z \subset M^n$ with $||V||(Z) = 0$, such that, for every open subset $W \subset \subset M^n$ and $X\in\mathfrak{X}^0_c(M^n)$, we have that
\begin{equation}\label{equa1:L1}
\delta_g V(X)=-\int_{M^n}\left\langle \overrightarrow{H_g},X\right\rangle_gd||V||+\int_{M^n}\left\langle\nu,X\right\rangle_g d||\delta_g V||_{\sing}
\end{equation}

and $||\overrightarrow{H_g}||_{L^1(W,||V||)} \leq C_W$, where $\nu$ is a $||\delta_gV||_{\mathrm{TV}}$-measurable function with $||\nu(x)||_g = 1$, for all $x \in M^n$, and $||\delta_gV||_{\sing} \defeq ||\delta_gV||_{\mathrm{TV}} \mres Z$.
\end{prop}
\begin{proof}
Since $V$ has locally bounded first variation, for all open set $W \subset \subset M^n$, there exists a real constant $C_W > 0$ such that $|\delta_g V(X)|\leq C_W||X||_{L^{\infty}(W,||V||)}$. Then, the total variation of $\delta_gV$,
\begin{displaymath}
||\delta_gV||_{\mathrm{TV}}(W) \defeq \sup\{ |\delta_gV(X)| : X \in \mathfrak{X}^0_c(W), ||X||_{L^{\infty}(W,||V||)} \leq 1\} \leq C_W < \infty.
\end{displaymath}

By Riesz representation theorem \cite{FedererGMT}*{Theorem 2.5.13, p. 106} $||\delta_gV||_{\mathrm{TV}}$ is a Radon measure on $M^n$ and there exists a $||\delta_gV||_{\mathrm{TV}}$-mensurable function $\nu:M^n \to TM^n$ with $||\nu(x)||_{g_x}=1$, for $||\delta_gV||_{\mathrm{TV}}$-a.e. $x \in M^n$. Furthermore, for all $X \in \mathfrak{X}^0_c(M^n)$, 
\begin{equation}\label{equa0:L1}
\delta_gV(X) = \int_{M^n} \langle \nu, X\rangle_gd||\delta_gV||_{\mathrm{TV}}.
\end{equation}

By Radon-Nikodym Theorem (\cite{FedererGMT}*{Theorem 2.9.2, p. 153} and \cite{Folland}*{Theorem 2.9.7, p. 155}) we have, for all $B \in \mathcal{B}(M^n)$,
\begin{equation}\label{equa0.1:L1}
||\delta_gV||_{\mathrm{TV}}(B) = \int_B \Theta^{||V||}(||\delta_gV||_{\mathrm{TV}},x) d||V||(x) + \left(||\delta_gV||_{\mathrm{TV}} \mres Z\right)(B)
\end{equation}

with $(\Theta^{||V||}||\delta_gV||_{\mathrm{TV}},x) \in L^1(B,||V||)$ and $Z = \{x \in M^n : \Theta^{||V||}\left(||\delta_gV||_{\mathrm{TV}},x\right) =  \infty\}$.

Defining $-\overrightarrow{H_g}(x) \defeq \Theta^{||V||}(||\delta_gV||_{\mathrm{TV}},x)\nu(x)$, for all $x \in M^n$, and $||\delta_gV||_{\sing} \defeq ||\delta_gV||_{\mathrm{TV}} \mres Z$, by \Cref{equa0:L1,equa0.1:L1} we conclude that \Cref{equa1:L1} holds.

Furthermore, for $W \subset \subset M^n$ open, we have that by Hölder inequality that
\begin{align*}
|| \overrightarrow{H_g} ||_{L^1(W,||V||)} = \left|\left| -\Theta^{||V||}\left(||\delta_gV||_{\mathrm{TV}},x\right)\nu \right|\right|_{L^1(W,||V||)} \leq C_W,
\end{align*}

where $C_W \geq 0$ exists because $\Theta^{||V||}(||\delta_gV||_{\mathrm{TV}},x) \in L^1(W,||V||)$ and $||\nu(x)||_{g_x}=1$, for $||\delta_gV||_{\mathrm{TV}}$-a.e. $x \in M^n$.
\end{proof}
Let $1 < p < \infty$. We say that $q \in \mathbb{R}$ is the conjugate exponent of $p$ when $q$ is such that
\begin{displaymath}
\dfrac{1}{p} + \dfrac{1}{q} = 1.
\end{displaymath}

When $p = 1$ we say that $q = \infty$ is its conjugate exponent and when $p = \infty$ we say that $q = 1$ is its conjugate exponent. By symmetry of this definition we have that, for $1 \leq p,q \leq \infty$, $p$ is the conjugate exponent of $q$ if, and only if, $q$ is the conjugate exponent of $p$.
\begin{prop}\label{prop:generalizedmeancurvaturaLpbounded:section:firstvariation}
Let $1 \leq  p < \infty$, $q$ the conjugate exponent of $p$, and $V\in\mathbf{V}_m(M^n)$. Then the following statements are equivalent:
\begin{enumerate}
    \item\label{statement1:prop:generalizedmeancurvaturaLpbounded:section:firstvariation} For every open subset $W \subset \subset M^n$, there exists a real constant $C_W > 0$ (which depends of $W$) such that
    \begin{displaymath}
    |\delta_g V(X)|\leq C_W||X||_{L^p(W,||V||)}
    \end{displaymath}
    
    holds for every $X \in \mathfrak{X}^0_c(W)$.
    \item\label{statement2:prop:generalizedmeancurvaturaLpbounded:section:firstvariation} The total variation of $\delta_gV$, denoted by $||\delta_gV||_{\mathrm{TV}}$, is a Radon measure. Furthermore, there exists a $||V||$-measurable function $\overrightarrow{H_g}: M^n \to TM^n$ such that
    \begin{equation}\label{equa1:Lp}
    \delta_g V(X)=-\int_{M^n}\left\langle \overrightarrow{H_g},X\right\rangle_gd||V||
    \end{equation}
    
    and $||\overrightarrow{H_g}||_{L^q(W,||V||)} \leq C_W$, for every open subset $W \subset \subset M^n$ and $X\in\mathfrak{X}^0_c(W)$.
\end{enumerate}
\end{prop}
\begin{proof}
$\left(\Cref{statement1:prop:generalizedmeancurvaturaLpbounded:section:firstvariation} \implies \Cref{statement2:prop:generalizedmeancurvaturaLpbounded:section:firstvariation}\right)$: Suppose that, for every open subset $W \subset \subset M^n$, there exists a real constant $C_W > 0$ such that $|\delta_g V(X)|\leq C_W||X||_{L^p(W,||V||)}$ holds for every $X \in \mathfrak{X}^0_c(W)$. Then, the total variation of $\delta_gV$,
\begin{align*}
||\delta_gV||_{\mathrm{TV}}(W) &\defeq \sup\{ |\delta_gV(X)| : X \in \mathfrak{X}^0_c(W), ||X||_{L^{\infty}(W,||V||)} \leq 1\} < \infty.
\end{align*}

By Riesz representation theorem \cite{FedererGMT}*{Theorem 2.5.13, p. 106} $||\delta_gV||_{\mathrm{TV}}$ is a Radon measure on $M^n$ and there exists a $||\delta_gV||_{\mathrm{TV}}$-mensurable function $\nu:M^n \to TM^n$ with $||\nu(x)||_{g_x}=1$, for $||\delta_gV||_{\mathrm{TV}}$-a.e. $x \in M^n$. Furthermore, for all $X \in \mathfrak{X}^0_c(M^n)$, 
\begin{equation}\label{equa0:Lp}
\delta_gV(X) = \int_{M^n} \langle \nu, X\rangle_gd||\delta_gV||_{\mathrm{TV}}.
\end{equation}

Furthermore, note that $||\delta_gV||_{\mathrm{TV}}$ is absolutely continuous with respect to $||V||$. In fact, let $B \in \mathcal{B}(M^n)$ such that $||V||(B) = 0$ then
\begin{align*}
||\delta_gV||_{\mathrm{TV}}(B) &\leq \sup\left\{ C_B||X||_{L^p(B,||V||)} : X \in \mathfrak{X}^0_c(B), ||X||_{L^{\infty}(B,||V||)}\leq 1\right\}\\
&\leq \sup\left\{ C_B||X||_{L^{\infty}(B,||V||)}\left(||V||(B)\right)^{\frac{1}{p}} : X \in \mathfrak{X}^0_c(B), ||X||_{L^{\infty}(B,||V||)}\leq 1\right\} = 0.
\end{align*}

By Radon-Nikodym Theorem (\cite{FedererGMT}*{Theorem 2.9.2, p. 153} and \cite{Folland}*{Theorem 2.9.7, p. 155}) we have, for all $B \in \mathcal{B}(M^n)$,
\begin{equation}\label{equa0.1:Lp}
||\delta_gV||_{\mathrm{TV}}(B) = \int_B \Theta^{||V||}(||\delta_gV||_{\mathrm{TV}},x) d||V||(x)
\end{equation}

with $\Theta^{||V||}||(\delta_gV||_{\mathrm{TV}},x) \in L^1(B,||V||)$.

Defining $-\overrightarrow{H_g}(x) \defeq \Theta^{||V||}(||\delta_gV||_{\mathrm{TV}},x)\nu(x)$, for all $x \in M^n$, by \Cref{equa0:Lp,equa0.1:Lp} we conclude that \Cref{equa1:Lp} holds.

Furthermore, since for every open subset $W \subset \subset M^n$ and for every $X \in \mathfrak{X}^0_c(W)$ there exists a real constant $C_W > 0$ such that $|\delta_g V(X)|\leq C_W||X||_{L^p(W,||V||)}$, we have that $\delta_gV$ is a bounded linear functional on $L^p(W,||V||)$. By \cite{RudinAnalysis}*{Theorem 6.16, p. 127} there exists a unique $F \in L^q(W,||V||)$ such that
\begin{equation}\label{equa2proof:prop:generalizedmeancurvaturaLpbounded:section:firstvariation}
\delta_gV(X) = \int_{M^n} \left\langle F,X\right\rangle_gd||V||.
\end{equation}

Combining \Cref{equa1:Lp} and \Cref{equa2proof:prop:generalizedmeancurvaturaLpbounded:section:firstvariation} we obtain that $\overrightarrow{H_g} = -F$ and, therefore, there exists a real constant $C_W > 0$ such that $||\overrightarrow{H_g}||_{L^q(W,||V||)} \leq C_W$, for every open subset $W \subset \subset M^n$.

$\left(\Cref{statement2:prop:generalizedmeancurvaturaLpbounded:section:firstvariation} \implies \Cref{statement1:prop:generalizedmeancurvaturaLpbounded:section:firstvariation}\right)$: Now, suppose that $||\delta_gV||_{\mathrm{TV}}$ is a Radon measure, there exists a $||V||$-measurable function $\overrightarrow{H_g}: M^n \to TM^n$ such that \Cref{equa1:Lp} holds for all $X\in\mathfrak{X}^0_c(M^n)$, and $||\overrightarrow{H_g}||_{L^q(W,||V||)} \leq C_W$, for every open subset $W \subset \subset M^n$. Let $W \subset \subset M^n$ and $X\in\mathfrak{X}^0_c(W)$ be arbitrary. Thus by Hölder inequality
\begin{align*}
|\delta_gV(X)| &= \left| -\int_M\left\langle \overrightarrow{H_g},X\right\rangle_gd||V|| \right| \leq \int_M ||\overrightarrow{H_g}||_g||X||_g d||V|| \leq C_W||X||_{L^p(W,||V||)}.
\end{align*}
\end{proof}
\begin{defi}
Using the same notation as in \Cref{prop:generalizedmeancurvaturaL1bounded:section:firstvariation} and \Cref{prop:generalizedmeancurvaturaLpbounded:section:firstvariation}, we denote by $\overrightarrow{H_g}$ the generalized mean curvature vector of $V$, by $Z$ the generalized boundary of $V$, by $||\delta_g V||_{\sing}$ the generalized boundary measure of $V$, and by $\nu{\restriction_Z}$ the generalized unit conormal of $V$.
\end{defi}
\begin{defi}
Let $1 \leq  p < \infty$, $q$ the conjugate exponent of $p$, and $V\in\mathbf{V}_m(M^n)$. We say that $V$ has $L^q_{\mathrm{loc}}$-bounded generalized mean curvature vector in $M^n$ when $V$ satisfies \Cref{statement1:prop:generalizedmeancurvaturaLpbounded:section:firstvariation} or \Cref{statement2:prop:generalizedmeancurvaturaLpbounded:section:firstvariation} in the \Cref{prop:generalizedmeancurvaturaLpbounded:section:firstvariation}.   
\end{defi}
\subsection{Monotonicity formula}\label{section:monotonicity}
Let $0 < \varepsilon < 1$, and  $\varphi_{\varepsilon} \in C^1_c\left(]-1,1[\right)$  an even function such that
\begin{displaymath}
\varphi_{\varepsilon}(x) \defeq
\begin{cases}
1 \text{ if } x \leq \varepsilon,\\
0 \text{ if } x \geq 1,
\end{cases}
\hspace{1cm}
\text{ and, }  
\hspace{1cm}
\varphi_{\varepsilon}' < 0 \text{, if } 0<\varepsilon < x < 1.
\end{displaymath}
\begin{defi}\label{defi:radial-pert:section:monotonicity}
Let $\xi \in M^n$, $0 < \delta < 1$, $0 < \rho < \injec{r0}$, $y \in B_{d_g}(\xi,\delta\rho)$, and $\sigma \in ]0,(1-\delta)\rho]$. We define the $\varepsilon,\sigma$-radial perturbation of the gradient of $u_y$ by
\begin{align*}
X_{\varepsilon,\sigma}: M^n &\to TM^n \nonumber \\
x &\mapsto X_{\varepsilon,\sigma}(x) \defeq \varphi_{\varepsilon}\left(\dfrac{u_{y}(x)}{\sigma}\right)u_y(x)\grad{g_x}(u_y)(x),
\end{align*}

where $u_y(x) \defeq d_g(y,x)$, for all $x \in M^n$, and $\grad{g_x}(u_y)(x)$ is the gradient of $u_y$ at $x$, for all $x \in M^n$. 
\end{defi}
\begin{rema}
Note that the $\varepsilon,\sigma$-radial perturbation of the gradient of $u_y$ is a vector field with compact support. Moreover, $\supp\left(X_{\varepsilon,\sigma}\right) \subset B_{d_g}(y,\sigma)$.
\end{rema}
\begin{lemm}\label{lemm1:div-radial-pert:section:monotonicity}
Let $\xi \in M^n$, $0 < \delta < 1$, $0 < \rho < \injec{r0}$. Then, for all $y \in B_{d_g}(\xi,\delta\rho)$, for all $\sigma \in ]0,(1-\delta)\rho]$, for all $x \in B_{d_g}(y,\sigma)$, and for all $S(x) \in Gr(m,T_xM)$, we have that  
\footnotesize
\begin{align}\label{equa1statement:lemm1:div-radial-pert:section:monotonicity}
\diver_{S(x)}\left(X_{\varepsilon,\sigma}\right)(x) &= \varphi_{\varepsilon}\left(\dfrac{u_y(x)}{\sigma}\right) \diver_{S(x)}\left(u_y\grad{g}u_y\right)(x) + \varphi_{\varepsilon}'\left(\dfrac{u_y(x)}{\sigma}\right) \dfrac{u_y(x)}{\sigma} \nonumber \\
&- \varphi_{\varepsilon}'\left(\dfrac{u_y(x)}{\sigma}\right) \dfrac{u_y(x)}{\sigma} \left|\left| P_{S(x)^{\perp}}(\grad{g_x}(u_y)(x)) \right|\right|_{g_x}^2,
\end{align}
\normalsize

where $P_{S(x)^{\perp}}: T_xM^n \to S(x)^{\perp}$ is the orthogonal projection onto $S(x)^{\perp}$ with respect to the metric $g_x$.
\end{lemm}
\begin{proof}
Let  $y \in B_{d_g}(\xi,\delta\rho)$, $\sigma \in ]0,(1-\delta)\rho]$, $x \in B_{d_g}(y,\sigma)$, and $S(x) \in Gr(m,T_xM)$ arbitrary. By a standard computation we have that

\footnotesize
\begin{align}\label{equa1proof:lemm1:div-radial-pert:section:monotonicity}
\diver_{S(x)}\left(X_{\varepsilon,\sigma}\right)(x) = \varphi_{\varepsilon}'\left(\dfrac{u_y(x)}{\sigma}\right) \dfrac{u_y(x)}{\sigma}  \left|\left| P_{S(x)}\left(\grad{g_x}(u_y)(x)\right) \right|\right|_{g_x}^2 + \varphi_{\varepsilon}\left(\dfrac{u_y(x)}{\sigma}\right) \diver_S\left(u_y\grad{g}u_y\right)(x).
\end{align}
\normalsize

where $P_{S(x)}: T_xM \to S(x)$ is the orthogonal projection onto $S(x)$ with respect to the metric $g_x$ and $P_{S(x)^{\perp}}: T_xM \to S(x)^{\perp}$ is the orthogonal projection onto $S(x)^{\perp}$ with respect to the metric $g_x$. Since $u_y$ is the distance function we have that
\begin{equation}\label{equa2proof:lemm1:div-radial-pert:section:monotonicity}
\left|\left| P_{S(x)}(\grad{g_x}(u_y)(x)) \right|\right|_{g_x}^2 + \left|\left| P_{S(x)^{\perp}}(\grad{g_x}(u_y)(x)) \right|\right|_{g_x}^2 = \left|\left| \grad{g_x}(u_y)(x) \right|\right|_{g_x}^2 = 1.
\end{equation}

By \Cref{equa1proof:lemm1:div-radial-pert:section:monotonicity} and \Cref{equa2proof:lemm1:div-radial-pert:section:monotonicity}, we obtain \Cref{equa1statement:lemm1:div-radial-pert:section:monotonicity}.
\end{proof}
Although this choice is enough to get many useful information, let us consider a general case, which will be used in the sequel when we will prove a fundamental height estimates.
\begin{defi}\label{defi:h-radial-pert:section:monotonicity}
Let $h \in \C^1(M^n,[0,+\infty[)$, $\xi \in M^n$, $0 < \delta < 1$, $0 < \rho < \injec{r0}$, $y \in B_{d_g}(\xi,\delta\rho)$, and $\sigma \in ]0,(1-\delta)\rho]$. We define the $h,\varepsilon,\sigma$-radial perturbation of the gradient of $u_y$ by
\begin{align*}
X_{h,\varepsilon,\sigma} : M^n &\to TM^n \nonumber \\
x &\mapsto X_{h,\varepsilon,\sigma}(x) \defeq h(x)X_{\varepsilon,\sigma}(x) = h(x)\varphi_{\varepsilon}\left(\dfrac{u_{y}(x)}{\sigma}\right)u_y(x)\grad{g_x}(u_y)(x),  
\end{align*}

where $u_y(x) \defeq d_g(y,x)$, for all $x \in M^n$, and $\grad{g_x}(u_y)(x)$ is the gradient of $u_y$ at $x$, for all $x \in M^n$. 
\end{defi}
\begin{lemm}\label{lemm1:div-h-radial-pert:section:monotonicity}
Let $h \in \C^1(M,[0,+\infty[)$, $\xi \in M$, $0 < \delta < 1$, $0 < \rho < \injec{r0}$, and $V \in \mathbf{V}_m(M^n)$. Then, for all $y \in B_{d_g}(\xi,\delta\rho)$ and for all $\sigma \in ]0,(1-\delta)\rho]$, we have that
\footnotesize
\begin{align}\label{equa1statement:lemm1:div-h-radial-pert:section:monotonicity}
\delta_gV\left(X_{h,\varepsilon,\sigma}\right) &= \int_{\Gr_m(M^n)} \varphi_{\varepsilon}\left(\dfrac{u_y(x)}{\sigma}\right) \left\langle u_y(x)\grad{g_x}(u_y)(x) , P_{S(x)}(\grad{g_x}(h)(x)) \right\rangle_{g_x} dV(x,S(x)) \nonumber \\
&+ \int_{\Gr_m(M^n)} h(x)\varphi_{\varepsilon}\left(\dfrac{u_y(x)}{\sigma}\right) \diver_{S}\left(u_y\grad{g}u_y\right)(x) dV(x,S(x)) \nonumber\\
&+ \int_{\Gr_m(M^n)} h(x)\varphi_{\varepsilon}'\left(\dfrac{u_y(x)}{\sigma}\right) \dfrac{u_y(x)}{\sigma} dV(x,S(x)) \nonumber \\
&- \int_{\Gr_m(M^n)} h(x)\varphi_{\varepsilon}'\left(\dfrac{u_y(x)}{\sigma}\right) \dfrac{u_y(x)}{\sigma} \left|\left| P_{S(x)^{\perp}}(\grad{g}(u_y)(x)) \right|\right|_{g_x}^2 dV(x,S(x)).
\end{align}
\normalsize
\end{lemm}
\begin{proof}
Let  $y \in B_{d_g}(\xi,\delta\rho)$, $\sigma \in ]0,(1-\delta)\rho]$, $x \in B_{d_g}(y,\sigma)$, and $S(x) \in Gr(m,T_xM)$ arbitrary. By \Cref{lemm1:div-radial-pert:section:monotonicity} we have that
\begin{align}\label{equa2proof:lemm1:div-h-radial-pert:section:monotonicity}
\diver_S\left( X_{h,\varepsilon,\sigma} \right)(x) &= \varphi_{\varepsilon}\left(\dfrac{u_{y}(x)}{\sigma}\right) \left\langle  u_y(x)\grad{g_x}(u_y)(x) , P_{S(x)}\left(\grad{g_x}(h)(x)\right) \right\rangle_{g_x} \nonumber \\
&+ h(x) \varphi_{\varepsilon}\left(\dfrac{u_y(x)}{\sigma}\right) \diver_S\left(u_y\grad{g}u_y\right)(x) \nonumber \\
&+ h(x) \varphi_{\varepsilon}'\left(\dfrac{u_y(x)}{\sigma}\right) \dfrac{u_y(x)}{\sigma} \nonumber \\
&- h(x) \varphi_{\varepsilon}'\left(\dfrac{u_y(x)}{\sigma}\right) \dfrac{u_y(x)}{\sigma} \left|\left| P_{S(x)^{\perp}}(\grad{g_x}(u_y)(x)) \right|\right|_{g_x}^2.
\end{align}

Then, by \Cref{equa2proof:lemm1:div-h-radial-pert:section:monotonicity} we have \Cref{equa1statement:lemm1:div-h-radial-pert:section:monotonicity}.
\end{proof}
\begin{rema}
When comparing \Cref{lemm1:div-h-radial-pert:section:monotonicity} to the Euclidean case \cite{Simon}*{Chapter 8, Section 3, p. 240}, we find a term that depends of the metric in addition to the dimensional Euclidean term. Therefore, in order to obtain our desired result, we need to make a comparison with the Euclidean case and we do it by introducing a curvature term.
\end{rema}
\begin{defi}\label{defi:cotb}
We define the function $c: \mathbb{R} \times \mathbb{R} \to \mathbb{R}$ by $c(t,b) \defeq t\cot_b(t)$, where $\cot_b(t) \defeq \frac{s'_b(t)}{s_b(t)}$ and 
\footnotesize
\begin{displaymath}
s_b(t) = \begin{cases}
\dfrac{\sinh(\sqrt{-b}t)}{\sqrt{-b}}, &\text{if } b<0,\\
t, &\text{if } b=0,\\
\dfrac{\sin\left(\sqrt{b}t\right)}{\sqrt{b}}, &\text{if } b>0.
\end{cases}   
\end{displaymath}
\normalsize
\end{defi}
From now on, we assume that there exists $K > 0$ such that $|\sec_g| \leq K$, where $\sec_g$ is the sectional curvature of $(M^n,g)$.
\begin{lemm}\label{lemm:bounds:section:monotonicity}
Let $\xi \in M^n$ and $0 < \delta < 1$. There exists $\rho_0 \defeq \rho_0(\injec{r0},K) > 0$ such that 
for all $0 < \rho < \rho_0$, for all $y \in B_{d_g}(\xi,\delta\rho)$, for all $\sigma \in ]0,(1-\delta)\rho]$, for all $x \in B_{d_g}(y,\sigma) \setminus \{y\}$, and for all $S(x) \in Gr(m,T_xM^n)$, we have that
\begin{equation}\label{equa2statement:lemm:bounds:section:monotonicity}
mc(u_y(x),K) \leq \diver_{S(x)}\left(u_y\grad{g}u_y\right)(x) \leq mc(u_y(x),-K).
\end{equation}
\end{lemm}
\begin{proof}
Let us fix $\rho_0 \defeq \rho_0(\injec{r0}) < \injec{r0}$. Consider $0 < \rho < \rho_0$, $y \in B_{d_g}(\xi,\delta\rho)$, $\sigma \in ]0,(1-\delta)\rho]$, $x \in B_{d_g}(y,\sigma) \setminus \{y\}$, and $S(x) \in Gr(m,T_xM)$ arbitrary. By a standard computation we have that
\begin{align}\label{equa1proof:lemm:bounds:section:monotonicity}
\diver_{S(x)}(u_y\grad{g}u_y)(x) = \left|\left| P_{S(x)}\left(\grad{g_x}(u_y)(x) \right) \right|\right|_{g_x}^2 + u_y(x) \sum\limits_{i=1}^m \hess_{g_x}(u_y)(x)\left(e_{j_i}(x),e_{j_i}(x)\right).
\end{align}

The existence of $\rho_0$ (with dependence only on $\injec{r0}$ and $K$) follows by applying the Hessian comparison theorem \cite{Lee}*{Theorem 11.7, p.~327} to the distance function $u_y$ in \Cref{equa1proof:lemm:bounds:section:monotonicity}. Indeed, for $\rho_0$ sufficiently small, the sectional curvature bounds imply corresponding upper and lower estimates for $\hess_{g_x}(u_y)$, which substituted into \Cref{equa1proof:lemm:bounds:section:monotonicity} yield \Cref{equa2statement:lemm:bounds:section:monotonicity}.
\end{proof}
\begin{lemm}\label{lemm:mono:section:monotonicity}
Let $h \in \C^1(M^n,[0,+\infty[)$, $\xi \in M^n$, $0 < \delta < 1$, and $V \in \mathbf{V}_m(M^n)$. There exists $\rho_0 \defeq \rho_0(\injec{r0},K) > 0$ such that for all $0 < \rho < \rho_0$, for all $y \in B_{d_g}(\xi,\delta\rho)$, and for all $\sigma \in ]0,(1-\delta)\rho]$, we have that
\footnotesize
\begin{align}\label{equa2statement:lemm:mono:section:monotonicity}
&\dfrac{d}{d\sigma}\left( \dfrac{1}{\sigma^m} \int_{B_{d_g}(y,\sigma)}  h(x)\varphi_{\varepsilon}\left(\dfrac{u_y(x)}{\sigma}\right)d||V||(x) \right) \geq \nonumber \\
&\geq \dfrac{1}{\sigma^{m+1}} \int_{\Gr_m(B_{d_g}(y,\sigma))} \varphi_{\varepsilon}\left(\dfrac{u_y(x)}{\sigma}\right) \left\langle u_y(x)\grad{g_x}(u_y)(x) , P_{S(x)}\left(\grad{g_x}(h)(x)\right)  \right\rangle_{g_x} dV(x,S(x)) \nonumber\\
&+ \dfrac{c(\sigma,K) - 1}{\sigma}\dfrac{m}{\sigma^m} \int_{B_{d_g}(y,\sigma)}  h(x)\varphi_{\varepsilon}\left(\dfrac{u_y(x)}{\sigma}\right)d||V||(x) \nonumber\\
&+ \int_{\Gr_m(B_{d_g}(y,\sigma))} \dfrac{h(x)}{\sigma^m}\dfrac{d}{d \sigma}\left(\varphi_{\varepsilon}\left(\dfrac{u_y}{\sigma}\right)\right)(x) \left|\left| P_{S(x)^{\perp}}(\grad{g_x}(u_y)(x)) \right|\right|_{g_x}^2 dV(x,S(x)) \nonumber\\
&-\dfrac{\delta_gV\left(X_{h,\varepsilon,\sigma}\right)}{\sigma^{m+1}},
\end{align}
\normalsize

and
\footnotesize
\begin{align}\label{equa3statement:lemm:mono:section:monotonicity}
&\dfrac{d}{d\sigma}\left( \dfrac{1}{\sigma^m} \int_{B_{d_g}(y,\sigma)}  h(x)\varphi_{\varepsilon}\left(\dfrac{u_y(x)}{\sigma}\right)d||V||(x) \right) \leq \nonumber \\
&\leq \dfrac{1}{\sigma^{m+1}} \int_{\Gr_m(B_{d_g}(y,\sigma))} \varphi_{\varepsilon}\left(\dfrac{u_y(x)}{\sigma}\right) \left\langle u_y(x)\grad{g_x}(u_y)(x) , P_{S(x)}\left(\grad{g_x}(h)(x)\right)  \right\rangle_{g_x} dV(x,S(x)) \nonumber\\
&+ \dfrac{c(\sigma,-K) - 1}{\sigma}\dfrac{m}{\sigma^m} \int_{B_{d_g}(y,\sigma)}  h(x)\varphi_{\varepsilon}\left(\dfrac{u_y(x)}{\sigma}\right)d||V||(x) \nonumber\\
&+ \int_{\Gr_m(B_{d_g}(y,\sigma))} \dfrac{h(x)}{\sigma^m}\dfrac{d}{d \sigma}\left(\varphi_{\varepsilon}\left(\dfrac{u_y}{\sigma}\right)\right)(x) \left|\left| P_{S(x)^{\perp}}(\grad{g_x}(u_y)(x)) \right|\right|_{g_x}^2 dV(x,S(x)) \nonumber\\
&-\dfrac{\delta_gV\left(X_{h,\varepsilon,\sigma}\right)}{\sigma^{m+1}},
\end{align}
\normalsize
\end{lemm}
\begin{proof}
Let us fix $\rho_0 \defeq \rho_0(\injec{r0},K) > 0$ given by \Cref{lemm:bounds:section:monotonicity}. Consider $0 < \rho < \rho_0$, $y \in B_{d_g}(\xi,\delta\rho)$, and $\sigma \in ]0,(1-\delta)\rho]$ arbitrary. 

Note that if $x \in M^n \setminus B_{d_g}(y,\sigma)$ then $u_y(x) \geq \sigma$, that is, $\frac{u_y(x)}{\sigma} \geq 1$. Thus,
\begin{equation}\label{equa1proof:lemm:mono:section:monotonicity}
x \in M \setminus B_{d_g}(y,\sigma) \implies \varphi_{\varepsilon}\left( \dfrac{u_y(x)}{\sigma} \right) = 0 \text{ and } \varphi_{\varepsilon}'\left( \dfrac{u_y(x)}{\sigma} \right) = 0.
\end{equation}

Since $M = B_{d_g}(y,\sigma) \sqcup \left(M \setminus B_{d_g}(x,\sigma)\right)$, by \Cref{equa1proof:lemm:mono:section:monotonicity} and \Cref{lemm1:div-h-radial-pert:section:monotonicity}, we have that
\footnotesize
\begin{align}\label{equa2proof:lemm:mono:section:monotonicity}
\delta_gV\left(X_{h,\varepsilon,\sigma}\right) &= \int_{\Gr_m\left(B_{d_g}(y,\sigma)\right)} \varphi_{\varepsilon}\left(\dfrac{u_y(x)}{\sigma}\right) \left\langle u_y(x)\grad{g_x}(u_y)(x) , P_{S(x)}(\grad{g_x}(h)(x)) \right\rangle_{g_x} dV(x,S(x)) \nonumber \\
&+ \int_{\Gr_m\left(B_{d_g}(y,\sigma)\right)} h(x)\varphi_{\varepsilon}\left(\dfrac{u_y(x)}{\sigma}\right) \diver_{S}\left(u_y\grad{g}u_y\right)(x) dV(x,S(x)) \nonumber\\
&+ \int_{\Gr_m\left(B_{d_g}(y,\sigma)\right)} h(x)\varphi_{\varepsilon}'\left(\dfrac{u_y(x)}{\sigma}\right) \dfrac{u_y(x)}{\sigma} dV(x,S(x)) \nonumber \\
&- \int_{\Gr_m\left(B_{d_g}(y,\sigma)\right)} h(x)\varphi_{\varepsilon}'\left(\dfrac{u_y(x)}{\sigma}\right) \dfrac{u_y(x)}{\sigma} \left|\left| P_{S(x)^{\perp}}(\grad{g}(u_y)(x)) \right|\right|_{g_x}^2 dV(x,S(x)).
\end{align}
\normalsize

We will divide the proof in the following two cases. Firstly, the case when $\sec_g \leq K$. By \Cref{equa2proof:lemm:mono:section:monotonicity} and \Cref{lemm:bounds:section:monotonicity}, we obtain
\footnotesize
\begin{align}\label{equa3proof:lemm:mono:section:monotonicity}
\delta_gV\left(X_{h,\varepsilon,\sigma}\right) &\geq \int_{\Gr_m\left(B_{d_g}(y,\sigma)\right)} \varphi_{\varepsilon}\left(\dfrac{u_y(x)}{\sigma}\right) \left\langle u_y(x)\grad{g_x}(u_y)(x) , P_{S(x)}(\grad{g_x}(h)(x)) \right\rangle_{g_x} dV(x,S(x)) \nonumber \\
&+ \int_{\Gr_m\left(B_{d_g}(y,\sigma)\right)} h(x)\varphi_{\varepsilon}\left(\dfrac{u_y(x)}{\sigma}\right)mc(u_y(x),K) dV(x,S(x)) \nonumber\\
&+ \int_{\Gr_m\left(B_{d_g}(y,\sigma)\right)} h(x)\varphi_{\varepsilon}'\left(\dfrac{u_y(x)}{\sigma}\right) \dfrac{u_y(x)}{\sigma} dV(x,S(x)) \nonumber \\
&- \int_{\Gr_m\left(B_{d_g}(y,\sigma)\right)} h(x)\varphi_{\varepsilon}'\left(\dfrac{u_y(x)}{\sigma}\right) \dfrac{u_y(x)}{\sigma} \left|\left| P_{S(x)^{\perp}}(\grad{g}(u_y)(x)) \right|\right|_{g_x}^2 dV(x,S(x)).
\end{align}
\normalsize

Note that $c(t,b)$ is non-increasing for $t>0$ and $b>0$. Thus, since $u_y(x) < \sigma$, for every $x \in B_{d_g}(y,\sigma)$, we have that $c(u_y(x),K) \geq c(\sigma,K)$. Therefore, by \Cref{equa3proof:lemm:mono:section:monotonicity}, we have
\footnotesize
\begin{align}\label{equa5proof:lemm:mono:section:monotonicity}
\dfrac{\delta_gV\left(X_{h,\varepsilon,\sigma}\right)}{\sigma^{m+1}} &\geq \dfrac{1}{\sigma^{m+1}} \int_{\Gr_m\left(B_{d_g}(y,\sigma)\right)} \varphi_{\varepsilon}\left(\dfrac{u_y(x)}{\sigma}\right) \left\langle u_y(x)\grad{g_x}(u_y)(x) , P_{S(x)}(\grad{g_x}(h)(x)) \right\rangle_{g_x} dV(x,S(x)) \nonumber \\
&+ \dfrac{c(\sigma,K) - 1}{\sigma}\dfrac{m}{\sigma^m} \int_{\Gr_m\left(B_{d_g}(y,\sigma)\right)} h(x)\varphi_{\varepsilon}\left(\dfrac{u_y(x)}{\sigma}\right) dV(x,S(x)) \nonumber\\
&+ \int_{\Gr_m\left(B_{d_g}(y,\sigma)\right)} \left(h(x)\varphi_{\varepsilon}'\left(\dfrac{u_y(x)}{\sigma}\right) \dfrac{u_y(x)}{\sigma^{m+2}} + \dfrac{m}{\sigma^{m+1}}h(x)\varphi_{\varepsilon}\left(\dfrac{u_y(x)}{\sigma}\right) \right)dV(x,S(x)) \nonumber \\
&- \int_{\Gr_m\left(B_{d_g}(y,\sigma)\right)} \dfrac{h(x)}{\sigma^{m+1}}\varphi_{\varepsilon}'\left(\dfrac{u_y(x)}{\sigma}\right) \dfrac{u_y(x)}{\sigma} \left|\left| P_{S(x)^{\perp}}(\grad{g}(u_y)(x)) \right|\right|_{g_x}^2 dV(x,S(x)).
\end{align}
\normalsize

By \cite{Folland}*{Theorem 2.27, p. 56}, we have that
\footnotesize
\begin{align}\label{equa6proof:lemm:mono:section:monotonicity}
&\dfrac{d}{d \sigma}\left(\dfrac{1}{\sigma^m} \int_{\Gr_m\left(B_{d_g}(y,\sigma)\right)} h(x) \varphi_{\varepsilon}\left(\dfrac{u_y(x)}{\sigma}\right) dV(x,S(x))\right) = \nonumber \\
&= -\int_{\Gr_m\left(B_{d_g}(y,\sigma)\right)} \left(h(x)\varphi_{\varepsilon}'\left(\dfrac{u_y(x)}{\sigma}\right) \dfrac{u_y(x)}{\sigma^{m+2}} + \dfrac{m}{\sigma^{m+1}}h(x)\varphi_{\varepsilon}\left(\dfrac{u_y(x)}{\sigma}\right) \right)dV(x,S(x)).
\end{align}
\normalsize

By \Cref{equa5proof:lemm:mono:section:monotonicity,equa6proof:lemm:mono:section:monotonicity}, we obtain
\footnotesize
\begin{align}\label{equa7proof:lemm:mono:section:monotonicity}
\dfrac{\delta_gV\left(X_{h,\varepsilon,\sigma}\right)}{\sigma^{m+1}} &\geq \dfrac{1}{\sigma^{m+1}} \int_{\Gr_m\left(B_{d_g}(y,\sigma)\right)} \varphi_{\varepsilon}\left(\dfrac{u_y(x)}{\sigma}\right) \left\langle u_y(x)\grad{g_x}(u_y)(x) , P_{S(x)}(\grad{g_x}(h)(x)) \right\rangle_{g_x} dV(x,S(x)) \nonumber \\
&+ \dfrac{c(\sigma,K) - 1}{\sigma}\dfrac{m}{\sigma^m} \int_{\Gr_m\left(B_{d_g}(y,\sigma)\right)} h(x)\varphi_{\varepsilon}\left(\dfrac{u_y(x)}{\sigma}\right) dV(x,S(x)) \nonumber\\
&- \dfrac{d}{d \sigma}\left(\dfrac{1}{\sigma^m} \int_{\Gr_m\left(B_{d_g}(y,\sigma)\right)} h(x) \varphi_{\varepsilon}\left(\dfrac{u_y(x)}{\sigma}\right) dV(x,S(x))\right) \nonumber \\
&- \int_{\Gr_m\left(B_{d_g}(y,\sigma)\right)} \dfrac{h(x)}{\sigma^{m+1}}\varphi_{\varepsilon}'\left(\dfrac{u_y(x)}{\sigma}\right) \dfrac{u_y(x)}{\sigma} \left|\left| P_{S(x)^{\perp}}(\grad{g}(u_y)(x)) \right|\right|_{g_x}^2 dV(x,S(x)).
\end{align}
\normalsize

By Disintegration Lemma there exists a Radon measure $V^{(x)}$ on $\Gr(m,T_xM^n)$ for $||V||$-a.e. $x \in M^n$ such that $V^{(x)}(\Gr(m,T_xM^n)) = 1$ and 
\begin{equation}\label{equa8proof:lemm:mono:section:monotonicity}
\int_{\Gr_m\left(B_{d_g}(y,\sigma)\right)} h(x)\varphi_{\varepsilon}\left(\dfrac{u_y(x)}{\sigma}\right) dV(x,S(x)) = \int_{B_{d_g}(y,\sigma)} h(x)\varphi_{\varepsilon}\left(\dfrac{u_y(x)}{\sigma}\right) d||V||(x). 
\end{equation}

By \Cref{equa7proof:lemm:mono:section:monotonicity,equa8proof:lemm:mono:section:monotonicity}, we obtain
\footnotesize
\begin{align*}
&\dfrac{d}{d \sigma}\left(\dfrac{1}{\sigma^m} \int_{B_{d_g}(y,\sigma)} h(x)\varphi_{\varepsilon}\left(\dfrac{u_y(x)}{\sigma}\right)  d||V||(x)\right) \geq \nonumber \\
&\geq \dfrac{1}{\sigma^{m+1}} \int_{\Gr_m\left(B_{d_g}(y,\sigma)\right)} \varphi_{\varepsilon}\left(\dfrac{u_y(x)}{\sigma}\right) \left\langle u_y(x)\grad{g_x}(u_y)(x) , P_{S(x)}(\grad{g_x}(h)(x)) \right\rangle_{g_x} dV(x,S(x)) \nonumber \\
&+ \dfrac{c(\sigma,K) - 1}{\sigma}\dfrac{m}{\sigma^m} \int_{B_{d_g}(y,\sigma)} h(x)\varphi_{\varepsilon}\left(\dfrac{u_y(x)}{\sigma}\right)  d||V||(x) \nonumber\\
&+ \int_{\Gr_m\left(B_{d_g}(y,\sigma)\right)} \dfrac{h(x)}{\sigma^m} \dfrac{d}{d \sigma} \left(\varphi_{\varepsilon}\left(\dfrac{u_y}{\sigma}\right)\right)(x) \left|\left| P_{S(x)^{\perp}}(\grad{g}(u_y)(x)) \right|\right|_{g_x}^2 dV(x,S(x)) \nonumber \\
&- \dfrac{\delta_gV\left(X_{h,\varepsilon,\sigma}\right)}{\sigma^{m+1}}.
\end{align*}
\normalsize

The case $\sec_g \geq -K$ is handled analogously to $\sec_g \leq K$, noting that $c(t,b)$ is non-decreasing for $t > 0$ and $b < 0$.
\end{proof}
Now we apply \Cref{lemm:mono:section:monotonicity} to varifolds that satisfy certain hypotheses on their first variation.
\begin{theo}\label{theo:mono-Lp:section:monotonicity}
Let $1 < p \leq \infty$, $h \in \C^1(M^n,[0,+\infty[)$, $\xi \in M^n$, $0 < \delta < 1$, and $V \in \mathbf{V}_m(M^n)$. There exists $\rho_0 \defeq \rho_0(\injec{r0},K) > 0$ such that if $V$ has $L^p_{\mathrm{loc}}$-bounded generalized mean curvature vector in $B_{d_g}(\xi,\rho)$, for some $0 < \rho < \rho_0$, then, for all $y \in B_{d_g}(\xi,\delta\rho)$ and for all $\sigma \in ]0,(1-\delta)\rho]$, we have that
\footnotesize
\begin{align}\label{equa2statement:theo:mono-Lp:section:monotonicity}
&\dfrac{d}{d\sigma}\left( \dfrac{1}{\sigma^m} \int_{B_{d_g}(y,\sigma)}  h(x)d||V||(x) \right) \geq \nonumber \\
&\geq \dfrac{1}{\sigma^{m+1}} \int_{\Gr_m(B_{d_g}(y,\sigma))} \left\langle u_y(x)\grad{g_x}(u_y)(x) , P_{S(x)}\left(\grad{g_x}(h)(x)\right)  \right\rangle_{g_x} dV(x,S(x)) \nonumber\\
&+ \dfrac{c(\sigma,K) - 1}{\sigma}\dfrac{m}{\sigma^m} \int_{B_{d_g}(y,\sigma)}  h(x)d||V||(x) \nonumber\\
&+ \dfrac{d}{d \sigma}\left(\int_{\Gr_m(B_{d_g}(y,\sigma))} \dfrac{h(x)}{u_y(x)^m} \left|\left| P_{S(x)^{\perp}}(\grad{g_x}(u_y)(x)) \right|\right|_{g_x}^2 dV(x,S(x))\right) \nonumber\\
&+ \dfrac{1}{\sigma^{m+1}} \int_{B_{d_g}(y,\sigma)} \left\langle \overrightarrow{H_g}(x), h(x)u_y(x)\grad{g_x}(u_y)(x) \right\rangle_{g_x} d||V||(x),
\end{align}
\normalsize

and
\footnotesize
\begin{align}\label{equa3statement:theo:mono-Lp:section:monotonicity}
&\dfrac{d}{d\sigma}\left( \dfrac{1}{\sigma^m} \int_{B_{d_g}(y,\sigma)}  h(x)d||V||(x) \right) \leq \nonumber \\
&\leq \dfrac{1}{\sigma^{m+1}} \int_{\Gr_m(B_{d_g}(y,\sigma))} \left\langle u_y(x)\grad{g_x}(u_y)(x) , P_{S(x)}\left(\grad{g_x}(h)(x)\right)  \right\rangle_{g_x} dV(x,S(x)) \nonumber\\
&+ \dfrac{c(\sigma,-K) - 1}{\sigma}\dfrac{m}{\sigma^m} \int_{B_{d_g}(y,\sigma)}  h(x)d||V||(x) \nonumber\\
&+ \dfrac{d}{d \sigma} \left(\int_{\Gr(B_{d_g}(y,\sigma))} \dfrac{h(x)}{u_y(x)^m} \left|\left| P_{S(x)^{\perp}}(\grad{g_x}(u_y)(x)) \right|\right|_{g_x}^2 dV(x,S(x))\right) \nonumber\\
&+ \dfrac{1}{\sigma^{m+1}} \int_{B_{d_g}(y,\sigma)} \left\langle \overrightarrow{H_g}(x), h(x)u_y(x)\grad{g_x}(u_y)(x) \right\rangle_{g_x} d||V||(x).
\end{align}
\normalsize
\end{theo}
\begin{proof}
Let us fix $\rho_0 \defeq \rho_0(\injec{r0},K) > 0$ given by \Cref{lemm:mono:section:monotonicity}. Suppose that $V$ has $L^p_{\mathrm{loc}}$-bounded generalized mean curvature vector in $B_{d_g}(\xi,\rho)$, for some $0 < \rho < \rho_0$. Consider $y \in B_{d_g}(\xi,\delta\rho)$ and $\sigma \in ]0,(1-\delta)\rho]$ arbitrary.

Let $\{\varepsilon_j\}_{j \in \mathbb{N}}$ a increasing sequence such that $\varepsilon_j > 0$, for all $j \in \mathbb{N}$, and $\varepsilon_j \uparrow 1$, when $j \uparrow \infty$. Note that, for all $t \in \mathbb{R}$, $\varphi_{\varepsilon_j}(t) \leq \mathcal{X}_{]-1,1[}(t)$, for all $j \in \mathbb{N}$, and $\lim_{j \uparrow \infty} \varphi_{\varepsilon_j}(t) = \mathcal{X}_{]-1,1[}(t)$. Thus, for all $x \in B_{d_g}(y,\sigma)$, we have that
\begin{equation*}
\lim\limits_{j \uparrow \infty} \varphi_{\varepsilon_j}\left(\dfrac{u_y(x)}{\sigma}\right) = \mathcal{X}_{]-1,1[}\left(\dfrac{u_y(x)}{\sigma}\right).
\end{equation*}

However, note that $\mathcal{X}_{]-1,1[}\left(\frac{u_y(x)}{\sigma}\right) = 1$ if, and only if, $-1 < \dfrac{u_y(x)}{\sigma} < 1$ which occurs if, and only if, $|u_y(x)| < \sigma$, that is, $x \in B_{d_g}(y,\sigma)$. Therefore, 
\begin{equation}\label{equa1proof:theo:mono-Lp:section:monotonicity}
\lim\limits_{j \uparrow \infty} \varphi_{\varepsilon_j}\left(\dfrac{u_y(x)}{\sigma}\right) = \mathcal{X}_{]-1,1[}\left(\dfrac{u_y(x)}{\sigma}\right) = \mathcal{X}_{B_{d_g}(y,\sigma)}(x).
\end{equation}

Applying this to \Cref{equa2statement:lemm:mono:section:monotonicity,equa3statement:lemm:mono:section:monotonicity} and taking the limit as $j \uparrow \infty$, we obtain, respectively, 
\footnotesize
\begin{align}\label{equa7proof:theo:mono-Lp:section:monotonicity}
&\lim\limits_{j \uparrow \infty} \left(\dfrac{d}{d\sigma}\left( \dfrac{1}{\sigma^m} \int_{B_{d_g}(y,\sigma)}  h(x)\varphi_{\varepsilon_j}\left(\dfrac{u_y(x)}{\sigma}\right)d||V||(x) \right)\right) \geq \nonumber \\
&\geq \lim\limits_{j \uparrow \infty} \left(\dfrac{1}{\sigma^{m+1}} \int_{\Gr_m(B_{d_g}(y,\sigma))} \varphi_{\varepsilon_j}\left(\dfrac{u_y(x)}{\sigma}\right) \left\langle u_y(x)\grad{g_x}(u_y)(x) , P_{S(x)}\left(\grad{g_x}(h)(x)\right)  \right\rangle_{g_x} dV(x,S(x))\right) \nonumber\\
&+ \lim\limits_{j \uparrow \infty} \left(\dfrac{c(\sigma,K) - 1}{\sigma}\dfrac{m}{\sigma^m} \int_{B_{d_g}(y,\sigma)}  h(x)\varphi_{\varepsilon_j}\left(\dfrac{u_y(x)}{\sigma}\right)d||V||(x)\right) \nonumber\\
&+ \lim\limits_{j \uparrow \infty} \int_{\Gr_m(B_{d_g}(y,\sigma))} \dfrac{h(x)}{\sigma^m}\dfrac{d}{d \sigma}\left(\varphi_{\varepsilon_j}\left(\dfrac{u_y}{\sigma}\right)\right)(x) \left|\left| P_{S(x)^{\perp}}(\grad{g_x}(u_y)(x)) \right|\right|_{g_x}^2 dV(x,S(x)) \nonumber\\
&- \lim\limits_{j \uparrow \infty}\dfrac{\delta_gV\left(X_{h,\varepsilon_j,\sigma}\right)}{\sigma^{m+1}},
\end{align}
\normalsize

and
\footnotesize
\begin{align}\label{equa8proof:theo:mono-Lp:section:monotonicity}
&\lim\limits_{j \uparrow \infty} \left(\dfrac{d}{d\sigma}\left( \dfrac{1}{\sigma^m} \int_{B_{d_g}(y,\sigma)}  h(x)\varphi_{\varepsilon_j}\left(\dfrac{u_y(x)}{\sigma}\right)d||V||(x) \right)\right) \leq \nonumber \\
&\leq \lim\limits_{j \uparrow \infty} \left(\dfrac{1}{\sigma^{m+1}} \int_{\Gr_m(B_{d_g}(y,\sigma))} \varphi_{\varepsilon_j}\left(\dfrac{u_y(x)}{\sigma}\right) \left\langle u_y(x)\grad{g_x}(u_y)(x) , P_{S(x)}\left(\grad{g_x}(h)(x)\right)  \right\rangle_{g_x} dV(x,S(x))\right) \nonumber\\
&+ \lim\limits_{j \uparrow \infty} \left(\dfrac{c(\sigma,-K) - 1}{\sigma}\dfrac{m}{\sigma^m} \int_{B_{d_g}(y,\sigma)}  h(x)\varphi_{\varepsilon_j}\left(\dfrac{u_y(x)}{\sigma}\right)d||V||(x)\right) \nonumber\\
&+ \lim\limits_{j \uparrow \infty} \left(\int_{\Gr_m(B_{d_g}(y,\sigma))} \dfrac{h(x)}{\sigma^m}\dfrac{d}{d \sigma}\left(\varphi_{\varepsilon_j}\left(\dfrac{u_y}{\sigma}\right)\right)(x) \left|\left| P_{S(x)^{\perp}}(\grad{g_x}(u_y)(x)) \right|\right|_{g_x}^2 dV(x,S(x))\right) \nonumber\\
&- \lim\limits_{j \uparrow \infty}\dfrac{\delta_gV\left(X_{h,\varepsilon_j,\sigma}\right)}{\sigma^{m+1}},
\end{align}
\normalsize

By a standard application of \cite{Folland}*{Theorem 2.24, p. 54} and \Cref{equa1proof:theo:mono-Lp:section:monotonicity} in each term of \Cref{equa7proof:theo:mono-Lp:section:monotonicity,equa8proof:theo:mono-Lp:section:monotonicity}, we obtain the desired result
\end{proof}
\begin{rema}\label{rema::mono-Lp:section:monotonicity}
In the context of \Cref{theo:mono-Lp:section:monotonicity}. Firstly note that
\begin{displaymath}
\lim\limits_{K \downarrow 0} \dfrac{c(\sigma,K) - 1}{\sigma} = 0 \hspace{1cm} \text{ and } \hspace{1cm} \lim\limits_{K \downarrow 0} \dfrac{c(\sigma,-K)-1}{\sigma} = 0.
\end{displaymath}

Thus taking the limit $K \downarrow 0$ in \Cref{equa2statement:theo:mono-Lp:section:monotonicity,equa3statement:theo:mono-Lp:section:monotonicity} we obtain (compare with \cite{Simon}*{Equation 4.1, Chapter 4, Section 4, p. 93}) that
\footnotesize
\begin{align*}
\dfrac{d}{d\sigma}\left( \dfrac{1}{\sigma^m} \int_{B_{d_g}(y,\sigma)}  h(x)d||V||(x) \right) &= \dfrac{1}{\sigma^{m+1}} \int_{\Gr_m(B_{d_g}(y,\sigma))} \left\langle u_y(x)\grad{g_x}(u_y)(x) , P_{S(x)}\left(\grad{g_x}(h)(x)\right)  \right\rangle_{g_x} dV(x,S(x)) \nonumber\\
&+ \dfrac{d}{d \sigma} \left(\int_{\Gr_m(B_{d_g}(y,\sigma))} \dfrac{h(x)}{u_y(x)^m} \left|\left| P_{S(x)^{\perp}}(\grad{g_x}(u_y)(x)) \right|\right|_{g_x}^2 dV(x,S(x))\right) \nonumber\\
&+ \dfrac{1}{\sigma^{m+1}} \int_{B_{d_g}(y,\sigma)} \left\langle \overrightarrow{H_g}(x), h(x)u_y(x)\grad{g_x}(u_y)(x) \right\rangle_{g_x} d||V||(x),
\end{align*}
\normalsize
\end{rema}
\begin{theo}\label{theo:mono-Linfty:section:monotonicity}
Let $\xi \in M$, $0 < \delta < 1$, and $V \in \mathbf{V}_m(M^n)$. There exists $\rho_0 \defeq \rho_0(\injec{r0},K) > 0$ such that if $V$ has $L^{\infty}_{\mathrm{loc}}$-bounded generalized mean curvature vector in $B_{d_g}(\xi,\rho)$, for some $0 < \rho < \rho_0$, then, for all $y \in B_{d_g}(\xi,\delta\rho)$ and for all $\sigma \in ]0,(1-\delta)\rho]$, we have that the function defined by
\begin{align}\label{equa2statement:theo:mono-Linfty:section:monotonicity}
f : ]0,\sigma[ &\to [0,+\infty[ \nonumber \\
t &\mapsto e^{\mathbf{C_1}t}\dfrac{||V||(B_{d_g}(y,t))}{t^m},
\end{align}

is non-decreasing, where $\mathbf{C_1} = \mathbf{C_1}(C(\rho),m,\rho,K) > 0$ and $C(\rho)>0$ is given by $||\overrightarrow{H_g}||_{L^{\infty}(B_{d_g}(\xi,\rho),||V||)} \leq C(\rho)$. Furthermore, for $0 < \sigma_1 < \sigma_2 < \sigma$ we have that
\footnotesize
\begin{align}\label{equa3statement:theo:mono-Linfty:section:monotonicity}
&e^{\mathbf{C_1}\sigma_2}\dfrac{||V||(B_{d_g}(y,\sigma_2))}{\sigma_2^m} - e^{\mathbf{C_1}\sigma_1}\dfrac{||V||(B_{d_g}(y,\sigma_1))}{\sigma_1^m} \geq \nonumber \\
&\geq \int_{\Gr_m\left(B_{d_g}(y,\sigma_2)\setminus B_{d_g}(y,\sigma_1)\right)} \dfrac{\left|\left| P_{S(x)^{\perp}}(\grad{g_x}(u_y)(x)) \right|\right|_{g_x}^2}{u_y(x)^m}  dV(x,S(x)),
\end{align}
\normalsize

and
\footnotesize
\begin{align}\label{equa4statement:theo:mono-Linfty:section:monotonicity}
&e^{-\mathbf{C_2}\sigma_2}\dfrac{||V||(B_{d_g}(y,\sigma_2))}{\sigma_2^m} - e^{-\mathbf{C_2}\sigma_1}\dfrac{||V||(B_{d_g}(y,\sigma_1))}{\sigma_1^m} \leq \nonumber \\
&\leq \int_{\Gr_m\left(B_{d_g}(y,\sigma_2)\setminus B_{d_g}(y,\sigma_1)\right)} \dfrac{\left|\left| P_{S(x)^{\perp}}(\grad{g_x}(u_y)(x)) \right|\right|_{g_x}^2}{u_y(x)^m}  dV(x,S(x)),
\end{align}
\normalsize

where $\mathbf{C_2} = \mathbf{C_2}(C(\rho),m,\rho,K) > 0$.
\end{theo}
\begin{proof}
Let us fix $\rho_0 \defeq \rho_0(\injec{r0},K) > 0$ given by \Cref{theo:mono-Lp:section:monotonicity}. Suppose that $V$ has $L^{\infty}_{\mathrm{loc}}$-bounded generalized mean curvature vector in $B_{d_g}(\xi,\rho)$, for some $0 < \rho < \rho_0$. Consider $y \in B_{d_g}(\xi,\delta\rho)$ and $\sigma \in ]0,(1-\delta)\rho]$ arbitrary.

Let $t \in ]0,\sigma[$ and $h: \supp(||V||) \to \mathbb{R}$ given by $h \equiv 1$. Applying \Cref{equa2statement:theo:mono-Lp:section:monotonicity} for $t$ we have that 
\footnotesize
\begin{align}\label{equa1proof:theo:mono-Linfty:section:monotonicity}
\dfrac{d}{d t}\left( \dfrac{||V||\left(B_{d_g}(y,t)\right)}{t^m} \right) &\geq  \dfrac{c(t,K) - 1}{t}m \dfrac{||V||\left(B_{d_g}(y,t)\right)}{t^m} + \dfrac{d}{d t}\left(\int_{\Gr_m(B_{d_g}(y,t))} \dfrac{\left|\left| P_{S(x)^{\perp}}(\grad{g_x}(u_y)(x)) \right|\right|_{g_x}^2}{u_y(x)^m}  dV(x,S(x))\right) \nonumber\\
&+ \dfrac{1}{t^{m+1}} \int_{B_{d_g}(y,t)} \left\langle \overrightarrow{H_g}(x), u_y(x)\grad{g_x}(u_y)(x) \right\rangle_{g_x} d||V||(x),
\end{align}
\normalsize

By Cauchy-Schwarz inequality and standard computations, we obtain that
\footnotesize
\begin{equation}\label{equa4proof:theo:mono-Linfty:section:monotonicity}
\dfrac{1}{t^{m+1}} \int_{B_{d_g}(y,t)} \left\langle \overrightarrow{H_g}(x) , u_y(x)\grad{g_x}(u_y)(x) \right\rangle_{g_x} d||V||(x) \geq - \dfrac{1}{t^m}\int_{B_{d_g}(y,t)} \left|\left|\overrightarrow{H_g}(x)\right|\right|_{g_x} d||V||(x).
\end{equation}
\normalsize

By \Cref{equa4proof:theo:mono-Linfty:section:monotonicity,equa1proof:theo:mono-Linfty:section:monotonicity}, and Hölder inequality we obtain that
\footnotesize
\begin{align}\label{equa5proof:theo:mono-Linfty:section:monotonicity}
\dfrac{d}{d t}\left( \dfrac{||V||\left(B_{d_g}(y,t)\right)}{t^m} \right) &\geq  -\left( \left|\left|\overrightarrow{H_g}\right|\right|_{L^{\infty}(B_{d_g}(y,t),||V||)} - m\dfrac{c(t,K) - 1}{t}\right) \dfrac{||V||\left(B_{d_g}(y,t)\right)}{t^m} \nonumber \\
&+ \dfrac{d}{d t}\left(\int_{\Gr_m(B_{d_g}(y,t))} \dfrac{\left|\left| P_{S(x)^{\perp}}(\grad{g_x}(u_y)(x)) \right|\right|_{g_x}^2}{u_y(x)^m}  dV(x,S(x))\right).
\end{align}
\normalsize

Note that for $t>0$ and $b>0$ the function $\frac{c(t,b)-1}{t}$ is non-increasing and negative. Thus, since $0 < t < \sigma < \rho$, we have that $\frac{c(t,K)-1}{t} \geq \frac{c(\rho,K)-1}{\rho}$. Moreover, since we are assuming that $V\in\mathbf{V}_m(M)$ has $L^{\infty}_{\mathrm{loc}}$-bounded generalized mean curvature vector in $B_{d_g}(\xi,\rho)$, there exists $C(\rho) > 0$ such that $\left|\left|\overrightarrow{H_g}\right|\right|_{L^{\infty}(B_{d_g}(\xi,\rho),||V||)} \leq C(\rho)$. Thus, by \Cref{equa5proof:theo:mono-Linfty:section:monotonicity}, we obtain
\footnotesize
\begin{align}\label{equa8proof:theo:mono-Linfty:section:monotonicity}
\dfrac{d}{d t}\left( \dfrac{||V||\left(B_{d_g}(y,t)\right)}{t^m} \right) &\geq - \mathbf{C_1} \dfrac{||V||\left(B_{d_g}(y,t)\right)}{t^m} + \dfrac{d}{d t}\left(\int_{\Gr_m(B_{d_g}(y,t))} \dfrac{\left|\left| P_{S(x)^{\perp}}(\grad{g_x}(u_y)(x)) \right|\right|_{g_x}^2}{u_y(x)^m}  dV(x,S(x))\right).
\end{align}
\normalsize

Since $e^{\mathbf{C_1}t} > 1$, for all $t > 0$, \Cref{equa8proof:theo:mono-Linfty:section:monotonicity} implies
\begin{equation}\label{equa13proof:theo:mono-Linfty:section:monotonicity}
f'(t) \geq \dfrac{d}{d t}\left(\int_{\Gr_m(B_{d_g}(y,t))} \dfrac{\left|\left| P_{S(x)^{\perp}}(\grad{g_x}(u_y)(x)) \right|\right|_{g_x}^2}{u_y(x)^m}  dV(x,S(x))\right) \geq 0.
\end{equation}

By \Cref{equa13proof:theo:mono-Linfty:section:monotonicity}, $f$ is non-decreasing. Furthermore, for $0 < \sigma_1 < \sigma_2 < \sigma$, integrating \Cref{equa13proof:theo:mono-Linfty:section:monotonicity} over $[\sigma_1,\sigma_2]$ with respect to $dt$, by Fundamental Theorem of Calculus, we obtain \Cref{equa3statement:theo:mono-Linfty:section:monotonicity}.

By a similar argument, we obtain \Cref{equa4statement:theo:mono-Linfty:section:monotonicity}.
\end{proof}
\begin{theo}\label{theo:mono-Lpm:section:monotonicity}
Let $m < p < \infty$, $\xi \in M^n$, $0 < \delta < 1$, and $V \in \mathbf{V}_m(M^n)$. There exists $\rho_0 \defeq \rho_0(\injec{r0},K) > 0$ such that if $V$ has $L^p_{\mathrm{loc}}$-bounded generalized mean curvature vector in $B_{d_g}(\xi,\rho)$, for some $0 < \rho < \rho_0$, then for all $y \in B_{d_g}(\xi,\delta\rho)$ and for all $\sigma \in ]0,(1-\delta)\rho]$, we have that the function defined by
\begin{align}\label{equa2statement:theo:mono-Lpm:section:monotonicity}
f : ]0,\sigma[ &\to [0,+\infty[ \nonumber \\
t &\mapsto \left(\dfrac{||V||\left(B_{d_g}(y,t)\right)}{t^m}\right)^{\frac{1}{p}} + \mathbf{C}t^{1-\frac{m}{p}},
\end{align}

is non-decreasing, where $\mathbf{C} = \mathbf{C}(C(\rho),m,\rho,K) > 0$ and $C(\rho)>0$ is given by $||\overrightarrow{H_g}||_{L^p(B_{d_g}(\xi,\rho),||V||)} \leq C(\rho)$. Furthermore, for $0 < \sigma_1 < \sigma_2 < \sigma$ we have that
\footnotesize
\begin{equation}\label{equa3statement:theo:mono-Lpm:section:monotonicity}
\left(\dfrac{||V||\left(B_{d_g}(y,\sigma_2)\right)}{\sigma_2^m}\right)^{\frac{1}{p}} - \left(\dfrac{||V||\left(B_{d_g}(y,\sigma_1)\right)}{\sigma_1^m}\right)^{\frac{1}{p}} \geq \mathbf{C}\left(\sigma_1^{1-\frac{m}{p}} - \sigma_2^{1-\frac{m}{p}}\right).
\end{equation}
\normalsize
\end{theo}
\begin{proof}
Let us fix $\rho_0 \defeq \rho_0(\injec{r0},K) > 0$ given by \Cref{theo:mono-Lp:section:monotonicity}. Suppose that $V$ has $L^p_{\mathrm{loc}}$-bounded generalized mean curvature vector in $B_{d_g}(\xi,\rho)$, for some $0 < \rho < \rho_0$. Consider $y \in B_{d_g}(\xi,\delta\rho)$ and $\sigma \in ]0,(1-\delta)\rho]$ arbitrary.

Let $t \in ]0,\sigma[$ and $h: \supp(||V||) \to \mathbb{R}$ given by $h \equiv 1$. Applying \Cref{equa2statement:theo:mono-Lp:section:monotonicity} for $t$ we have that 
\footnotesize
\begin{align}\label{equa1proof:theo:mono-Lpm:section:monotonicity}
\dfrac{d}{d t}\left( \dfrac{||V||\left(B_{d_g}(y,t)\right)}{t^m} \right) &\geq  \dfrac{c(t,K) - 1}{t}m \dfrac{||V||\left(B_{d_g}(y,t)\right)}{t^m} \nonumber\\
&+ \dfrac{d}{d t}\left(\int_{\Gr_m(B_{d_g}(y,t))} \dfrac{\left|\left| P_{S(x)^{\perp}}(\grad{g_x}(u_y)(x)) \right|\right|_{g_x}^2}{u_y(x)^m}  dV(x,S(x))\right) \nonumber\\
&+ \dfrac{1}{t^{m+1}} \int_{B_{d_g}(y,t)} \left\langle \overrightarrow{H_g}(x), u_y(x)\grad{g_x}(u_y)(x) \right\rangle_{g_x} d||V||(x),
\end{align}
\normalsize

By Cauchy-Schwarz inequality and standard computations, we obtain that
\footnotesize
\begin{equation}\label{equa4proof:theo:mono-Lpm:section:monotonicity}
\dfrac{1}{t^{m+1}} \int_{B_{d_g}(y,t)} \left\langle \overrightarrow{H_g}(x) , u_y(x)\grad{g_x}(u_y)(x) \right\rangle_{g_x} d||V||(x) \geq - \dfrac{1}{t^m}\int_{B_{d_g}(y,t)} \left|\left|\overrightarrow{H_g}(x)\right|\right|_{g_x} d||V||(x).
\end{equation}
\normalsize

By \Cref{equa1proof:theo:mono-Lpm:section:monotonicity,equa4proof:theo:mono-Lpm:section:monotonicity}, and Hölder inequality we obtain that
\tiny
\begin{align}\label{equa5proof:theo:mono-Lpm:section:monotonicity}
\dfrac{d}{d t}\left( \dfrac{||V||\left(B_{d_g}(y,t)\right)}{t^m} \right) &\geq - \left(\left|\left|\overrightarrow{H_g}\right|\right|_{L^p(B_{d_g}(y,t),||V||)} - \dfrac{c(t,K) - 1}{t}m\left(||V||\left(B_{d_g}(y,t)\right)\right)^{\frac{1}{p}}\right)\dfrac{\left(||V||\left(B_{d_g}(y,t)\right)\right)^{\frac{p-1}{p}}}{t^m} \nonumber\\
&+ \dfrac{d}{d t}\left(\int_{\Gr_m(B_{d_g}(y,t))} \dfrac{\left|\left| P_{S(x)^{\perp}}(\grad{g_x}(u_y)(x)) \right|\right|_{g_x}^2}{u_y(x)^m}  dV(x,S(x))\right).
\end{align}
\normalsize

Note that for $t>0$ and $b>0$ the function $\frac{c(t,b)-1}{t}$ is non-increasing and negative. Moreover, since we are assuming that $V\in\mathbf{V}_m(M)$ has $L^{p}_{\mathrm{loc}}$-bounded generalized mean curvature vector in $B_{d_g}(\xi,\rho)$, there exists $C(\rho) > 0$ such that $\left|\left|\overrightarrow{H_g}\right|\right|_{L^p(B_{d_g}(\xi,\rho),||V||)} \leq C(\rho)$. Thus, by \Cref{equa5proof:theo:mono-Lpm:section:monotonicity}, we obtain
\begin{equation}\label{equa11proof:theo:mono-Lpm:section:monotonicity}
f'(t) = \dfrac{d}{dt} \left( \left( \dfrac{||V||(B_{d_g}(y,t))}{t^m} \right)^{\frac{1}{p}} \right) + \mathbf{C} \dfrac{d}{dt}\left(t^{1-\frac{m}{p}}\right) \geq 0.
\end{equation}

By \Cref{equa11proof:theo:mono-Lpm:section:monotonicity} we have that $f'(t) \geq 0$ for all $0 < t < \sigma$, that is, $f$ is non-decreasing. Furthermore, for $0 < \sigma_1 < \sigma_2 < \sigma$, integrating \Cref{equa11proof:theo:mono-Lpm:section:monotonicity} over $[\sigma_1,\sigma_2]$ with respect to $dt$, by Fundamental Theorem of Calculus, we obtain \Cref{equa3statement:theo:mono-Lpm:section:monotonicity}.
\end{proof}
\begin{coro}\label{coro:densLp:section:monotonicity}
Let $m < p \leq \infty$, $\xi \in M^n$, $0 < \delta < 1$, and $V \in \mathbf{V}_m(M^n)$. There exists $\rho_0 \defeq \rho_0(\injec{r0},K) > 0$ such that if $V$ has $L^p_{\mathrm{loc}}$-bounded generalized mean curvature vector in $B_{d_g}(\xi,\rho)$, for some $0 < \rho < \rho_0$, then, for all $y \in B_{d_g}(\xi,\delta\rho)$, there exists the density
\begin{equation}\label{equa2statement:coro:densLp:section:monotonicity}
\Theta^m\left(||V||,y\right) = \lim\limits_{t \downarrow 0} \dfrac{||V||\left(B_{d_g}(y,t)\right)}{\omega_mt^m}.
\end{equation}

Furthermore, $\Theta^m\left(||V||, \cdot\right)$ is an upper-semicontinuous function in $B_{d_g}(\xi,\delta\rho)$, that is, for every $y \in B_{d_g}(\xi,\delta\rho)$, we have that
\begin{equation}\label{equa3statement:coro:densLp:section:monotonicity}
\limsup\limits_{x \to y}\Theta^m\left(||V||,x\right) \leq \Theta^m\left(||V||,y\right).
\end{equation}
\end{coro}
\begin{proof}
The proof we be divide in the cases $p = \infty$ and $m < p < \infty$.
\begin{itemize}
    \item $p = \infty$

    Let us fix $\rho_0 \defeq \rho_0(\injec{r0},K) > 0$ given by \Cref{theo:mono-Linfty:section:monotonicity}. Suppose that $V$ has $L^{\infty}_{\mathrm{loc}}$-bounded generalized mean curvature vector in $B_{d_g}(\xi,\rho)$, for some $0 < \rho < \rho_0$. Consider $y \in B_{d_g}(\xi,\delta\rho)$ and $\sigma \in ]0,(1-\delta)\rho]$ arbitrary.

    Define the function $f:]0,\sigma[ \to [0,+\infty[$ by
    \begin{displaymath}
    f(t) = e^{\mathbf{C_1}t}\dfrac{||V||(B_{d_g}(y,t))}{t^m},
    \end{displaymath}
    where $\mathbf{C_1} = \mathbf{C_1}(C(\rho),m,\rho,K) > 0$ and $C(\rho)>0$ is given by $||\overrightarrow{H_g}||_{L^{\infty}(B_{d_g}(\xi,\rho),||V||)} \leq C(\rho)$.
    
    Note that by \Cref{theo:mono-Linfty:section:monotonicity} we have that $f$ is non-decreasing. Furthermore, $f$ is bounded from below by $0$. Thus, the limit $\lim_{t \downarrow 0} f(t)$ exists. Let $M \geq 0$ be such limit.
    
    Let $\{t_i\}_{i \in \mathbb{N}}$ be a sequence such that $0 < t_i < \sigma$, for all $i \in \mathbb{N}$, and $t_i \downarrow 0$ when $i \uparrow \infty$. Then,
    \begin{align*}
    \Theta^{m}(||V||,y) = \lim\limits_{i \uparrow \infty} \dfrac{1}{\omega_m}\dfrac{f(t_i)}{e^{\mathbf{C_1}t_i}} < \infty.
    \end{align*}

    Which implies that the density $\Theta^{m}(||V||,y)$ exists. 

    Let $\varepsilon \in ]0,1[$ arbitrary and $\{y_i\}_{i \in \mathbb{N}} \subset B_{d_g}(\xi,\delta\rho)$ any sequence such that $y_i \to y$ when $i \uparrow \infty$.
    
    Note that for $i$ sufficiently large we have that $B_{d_g}(y_i,(1-\varepsilon)\sigma) \subset B_{d_g}(y,\sigma)$. 
    
    Take $\sigma_1 < (1-\varepsilon)\sigma$. Note that $y_i \in B_{d_g}(\xi,\delta\rho)$, for all $i \in \mathbb{N}$, and $0 < \sigma_1 < (1-\varepsilon)\sigma < \sigma$. By \Cref{theo:mono-Linfty:section:monotonicity} we have, for each $i \in \mathbb{N}$, that
    \footnotesize
    \begin{align*}
    &e^{\mathbf{C_1}(1-\varepsilon)\sigma}\dfrac{||V||(B_{d_g}(y_i,(1-\varepsilon)\sigma))}{\left((1-\varepsilon)\sigma\right)^m} - e^{\mathbf{C_1}\sigma_1}\dfrac{||V||(B_{d_g}(y_i,\sigma_1))}{\sigma_1^m} \geq 0. 
    \end{align*}
    \normalsize

    Thus, by standard computations, we obtain that $\Theta^m(||V||,\cdot)$ is an upper-semicontinuous function in $B_{d_g}(\xi,\delta\rho)$.
    
    \item $m < p < \infty$

    Let us fix $\rho_0 \defeq \rho_0(\injec{r0},K) > 0$ given by \Cref{theo:mono-Lpm:section:monotonicity}. Suppose that $V$ has $L^{p}_{\mathrm{loc}}$-bounded generalized mean curvature vector in $B_{d_g}(\xi,\rho)$, for some $0 < \rho < \rho_0$. Consider $y \in B_{d_g}(\xi,\delta\rho)$ and $\sigma \in ]0,(1-\delta)\rho]$ arbitrary.
    
    Define the function $f:]0,\injec{r0}[ \to [0,+\infty[$ by
    \begin{displaymath}
    f(t) = \left( \dfrac{||V||(B_{d_g}(y,t))}{t^m} \right)^{\frac{1}{p}}  + \mathbf{C_2} t^{1-\frac{m}{p}}
    \end{displaymath}
    
    Note that by \Cref{theo:mono-Lpm:section:monotonicity} we have that $f$ is non-decreasing. Furthermore, $f$ is bounded from below by $0$. Thus, the limit $\lim_{t \downarrow 0} f(t)$ exists. Let $M \geq 0$ be such limit.

    Let $\{t_i\}_{i \in \mathbb{N}}$ be a sequence such that $0 < t_i < \sigma$, for all $i \in \mathbb{N}$, and $t_i \downarrow 0$ when $i \uparrow \infty$. Then,
    \begin{align*}
    \Theta^{m}(||V||,y) = \lim\limits_{i \uparrow \infty} \dfrac{1}{\omega_m}\left( f(t_i) - \mathbf{C_2} t_i^{1-\frac{m}{p}} \right)^p < \infty,
    \end{align*}
    where $\mathbf{C_2} = \mathbf{C_2}(C(\rho),m,\rho,K) > 0$ and $C(\rho)>0$ is given by $||\overrightarrow{H_g}||_{L^p(B_{d_g}(\xi,\rho),||V||)} \leq C(\rho)$, which implies that the density $\Theta^{m}(||V||,y)$ exists.

    Let $\varepsilon \in ]0,1[$ arbitrary and $\{y_i\}_{i \in \mathbb{N}} \subset B_{d_g}(\xi,\delta\rho)$ any sequence such that $y_i \to y$ when $i \uparrow \infty$.
    
    Note that for $i$ sufficiently large we have that $B_{d_g}(y_i,(1-\varepsilon)\sigma) \subset B_{d_g}(y,\rho)$.
    
    Take $\sigma_1 < (1-\varepsilon)\sigma$. Note that $y_i \in B_{d_g}(\xi,\delta\rho)$, for all $i \in \mathbb{N}$, and $0 < \sigma_1 < (1-\varepsilon)\sigma < \sigma$. By \Cref{theo:mono-Lpm:section:monotonicity} we have, for each $i \in \mathbb{N}$, that
    \footnotesize
    \begin{displaymath}
    \left(\dfrac{||V||\left(B_{d_g}(y_i,(1-\varepsilon)\sigma)\right)}{\left((1-\varepsilon)\sigma\right)^m}\right)^{\frac{1}{p}} - \left(\dfrac{||V||\left(B_{d_g}(y_i,\sigma_1)\right)}{\sigma_1^m}\right)^{\frac{1}{p}} \geq \mathbf{C_2}\left(\sigma_1^{1-\frac{m}{p}} - \left((1-\varepsilon)\sigma\right)^{1-\frac{m}{p}}\right).
    \end{displaymath}
    \normalsize

    Thus, by standard computations, we obtain that $\Theta^m(||V||,\cdot)$ is an upper-semicontinuous function in $B_{d_g}(\xi,\delta\rho)$. 
\end{itemize}

To conclude the proof, we choose $\rho_0 \defeq \rho_0(\injec{r0},K) > 0$ as the minimum value ensuring that both cases above hold.
\end{proof}
\begin{rema}
In the context of \Cref{coro:densLp:section:monotonicity}, since $\Theta^m(||V||,\cdot)$ is upper-semicontinuous in $B_{d_g}(\xi,\delta\rho)$ then we have that $\Theta^m(||V||,\cdot)$ is a $||V||$-measurable function in $B_{d_g}(\xi,\delta\rho)$.
\end{rema}
\begin{theo}\label{theo:densLpa:section:monotonicity}
Let $1 < p \leq \infty$, $\xi \in M^n$, $0 < \delta < 1$, and $V \in \mathbf{V}_m(M^n)$. There exists $\rho_0 \defeq \rho_0(\injec{r0},K) > 0$ such that if $V$ has $L^p_{\mathrm{loc}}$-bounded generalized mean curvature vector in $B_{d_g}(\xi,\rho)$, for some $0 < \rho < \rho_0$, and
\begin{equation}\label{equa2statement:theo:densLpa:section:monotonicity}
\left|\left|\overrightarrow{H_g}\right|\right|_{L^1(B_{d_g}(y,t),||V||)} \leq \alpha \Lambda \left(\dfrac{t}{\rho}\right)^{\alpha - 1} ||V||(B_{d_g}(y,t)),
\end{equation}

for all $y \in B_{d_g}(\xi,\rho)$, for all $t > 0$ such that $B_{d_g}(y,t) \subset B_{d_g}(\xi,\rho)$, for some $\alpha > 0$, and for some $\Lambda > 0$, then, for all $y \in B_{d_g}(\xi,\delta\rho)$ and for all $\sigma \in ]0,(1-\delta)\rho]$, we have that the function defined by
\begin{align}\label{equa3statement:theo:densLpa:section:monotonicity}
f : ]0,\sigma[ &\to [0,+\infty[ \nonumber \\
t &\mapsto e^{\lambda(t)}\dfrac{||V||(B_{d_g}(y,t))}{t^m}
\end{align}

is non-decreasing, where 
\begin{align}\label{equa4statement:theo:densLpa:section:monotonicity}
\lambda : ]0,\sigma[ &\to ]0,+\infty[ \nonumber \\
t &\mapsto \left(\Lambda\left(\dfrac{t}{\rho}\right)^{\alpha-1} - \mathbf{C_1}m\right)t,
\end{align}

and $\mathbf{C_1} = \mathbf{C_1}(\rho,K) > 0$.

Furthermore, for $0 < \sigma_1 < \sigma_2 < \sigma$, we have that 
\begin{align}\label{equa5statement:theo:densLpa:section:monotonicity}
&e^{\lambda(\sigma_2)}\dfrac{||V||(B_{d_g}(y,\sigma_2))}{\sigma_2^m} - e^{\lambda(\sigma_1)}\dfrac{||V||(B_{d_g}(y,\sigma_1))}{\sigma_1^m} \geq \nonumber \\
&\geq \int_{\Gr_m\left(B_{d_g}(y,\sigma_2) \setminus B_{d_g}(y,\sigma_1)\right)} \dfrac{\left|\left| P_{S(x)^{\perp}}(\grad{g_x}(u_y)(x)) \right|\right|_{g_x}^2}{u_y(x)^m}  dV(x,S(x)),
\end{align}

and
\begin{align}\label{equa6statement:theo:densLpa:section:monotonicity}
&e^{-\ell(\sigma_2)}\dfrac{||V||(B_{d_g}(y,\sigma_2))}{\sigma_2^m} - e^{-\ell(\sigma_1)}\dfrac{||V||(B_{d_g}(y,\sigma_1))}{\sigma_1^m} \leq \nonumber \\
&\leq \int_{\Gr_m\left(B_{d_g}(y,\sigma_2) \setminus B_{d_g}(y,\sigma_1)\right)} \dfrac{\left|\left| P_{S(x)^{\perp}}(\grad{g_x}(u_y)(x)) \right|\right|_{g_x}^2}{u_y(x)^m}  dV(x,S(x)),
\end{align}

where
\begin{align}\label{equa7statement:theo:densLpa:section:monotonicity}
\ell : ]0,\sigma[ &\to ]0,+\infty[ \nonumber \\
t &\mapsto \left(\Lambda\left(\dfrac{t}{\rho}\right)^{\alpha-1} + \mathbf{C_2}m\right)t,
\end{align}

and $\mathbf{C_2} = \mathbf{C_2}(\rho,K) > 0$.
\end{theo}
\begin{proof}
Let us fix $\rho_0 \defeq \rho_0(\injec{r0},K) > 0$ given by \Cref{theo:mono-Lp:section:monotonicity}. Suppose that $V$ has $L^p_{\mathrm{loc}}$-bounded generalized mean curvature vector in $B_{d_g}(\xi,\rho)$, for some $0 < \rho < \rho_0$, and \Cref{equa2statement:theo:densLpa:section:monotonicity} holds, for all $0 < t < \rho$, for some $\alpha > 0$, and for some $\Lambda > 0$. Consider $y \in B_{d_g}(\xi,\delta\rho)$ and $\sigma \in ]0,(1-\delta)\rho]$ arbitrary.

Let $t \in ]0,\sigma[$ and $h: \supp(||V||) \to \mathbb{R}$ given by $h \equiv 1$. Applying \Cref{equa2statement:theo:mono-Lp:section:monotonicity} for $t$, we have that 
\footnotesize
\begin{align}\label{equa1proof:theo:densLpa:section:monotonicity}
\dfrac{d}{dt}\left( \dfrac{||V||\left(B_{d_g}(y,t)\right)}{t^m} \right) &\geq \dfrac{c(t,K) - 1}{t}m \dfrac{||V||\left(B_{d_g}(y,t)\right)}{t^m} \nonumber\\
&+ \dfrac{d}{d t}\left(\int_{\Gr_m(B_{d_g}(y,t))} \dfrac{\left|\left| P_{S(x)^{\perp}}(\grad{g_x}(u_y)(x)) \right|\right|_{g_x}^2}{u_y(x)^m}  dV(x,S(x))\right) \nonumber\\
&+ \dfrac{1}{t^{m+1}} \int_{B_{d_g}(y,t)} \left\langle \overrightarrow{H_g}(x), u_y(x)\grad{g_x}(u_y)(x) \right\rangle_{g_x} d||V||(x).
\end{align}
\normalsize

By Cauchy-Schwarz inequality and standard computations, we obtain that
\footnotesize
\begin{equation}\label{equa4proof:theo:densLpa:section:monotonicity}
\dfrac{1}{t^{m+1}} \int_{B_{d_g}(y,t)} \left\langle \overrightarrow{H_g}(x) , u_y(x)\grad{g_x}(u_y)(x) \right\rangle_{g_x} d||V||(x) \geq - \dfrac{1}{t^m}\int_{B_{d_g}(y,t)} \left|\left|\overrightarrow{H_g}(x)\right|\right|_{g_x} d||V||(x).
\end{equation}
\normalsize

Since we are assuming that \Cref{equa2statement:theo:densLpa:section:monotonicity} holds, for all $y \in B_{d_g}(\xi,\rho)$, for all $t > 0$ such that $B_{d_g}(y,t) \subset B_{d_g}(\xi,\rho)$, for some $\alpha > 0$, and for some $\Lambda > 0$. By \Cref{equa1proof:theo:densLpa:section:monotonicity,equa4proof:theo:densLpa:section:monotonicity}, and Hölder inequality we obtain that
\footnotesize
\begin{align}\label{equa5proof:theo:densLpa:section:monotonicity}
\dfrac{d}{d t}\left( \dfrac{||V||\left(B_{d_g}(y,t)\right)}{t^m} \right) &\geq -\left(\alpha \Lambda \left(\dfrac{t}{\rho}\right)^{\alpha - 1} - \dfrac{c(t,K) - 1}{t}m\right)\dfrac{||V||\left(B_{d_g}(y,t)\right)}{t^m} \nonumber \\
&+ \dfrac{d}{d t}\left(\int_{\Gr_m(B_{d_g}(y,t))} \dfrac{\left|\left| P_{S(x)^{\perp}}(\grad{g_x}(u_y)(x)) \right|\right|_{g_x}^2}{u_y(x)^m}  dV(x,S(x))\right).
\end{align}
\normalsize

Note that for $t>0$ and $b>0$ the function $\frac{c(t,b)-1}{t}$ is non-increasing and negative. Thus, since $0 < t < \sigma < \rho$, we have that $\frac{c(t,K)-1}{t} \geq \frac{c(\rho,K)-1}{\rho}$. Moreover, $e^{\lambda(t)}>0$, for all $t \in \mathbb{R}$, and $\lambda(t) > 0$, for all $t \in ]0,\sigma[$. Thus, by \Cref{equa5proof:theo:densLpa:section:monotonicity}, we obtain that
\begin{equation}\label{equa10proof:theo:densLpa:section:monotonicity}
f'(t) \geq \dfrac{d}{d t}\left(\int_{\Gr_m(B_{d_g}(y,t))} \dfrac{\left|\left| P_{S(x)^{\perp}}(\grad{g_x}(u_y)(x)) \right|\right|_{g_x}^2}{u_y(x)^m}  dV(x,S(x))\right) \geq 0.
\end{equation}

By \Cref{equa10proof:theo:densLpa:section:monotonicity}, $f$ is non-decreasing. Furthermore, for $0 < \sigma_1 < \sigma_2 < \sigma$, integrating \Cref{equa10proof:theo:densLpa:section:monotonicity} over $[\sigma_1,\sigma_2]$ with respect to $dt$, by Fundamental Theorem of Calculus, we obtain \Cref{equa5statement:theo:densLpa:section:monotonicity}.

By a similar argument, we obtain \Cref{equa6statement:theo:densLpa:section:monotonicity}.
\end{proof}
\begin{theo}\label{theo:densloc}
Let $\xi \in M^n$, $0 < \delta < 1$, and $V \in \mathbf{V}_m(M^n)$. There exists $\rho_0 \defeq \rho_0(\injec{r0},K) > 0$ such that if $V$ has locally bounded first variation in $B_{d_g}(\xi,\rho)$, for some $0 < \rho < \rho_0$, then, for $||V||$-a.e. $y \in B_{d_g}(\xi,\delta\rho)$, there exists the density
\begin{equation}
\Theta^m(||V||,y) = \lim\limits_{t \downarrow 0} \dfrac{||V||(B_{d_g}(y,t))}{\omega_mt^m}.
\end{equation}

Furthermore, $\Theta^m(||V||,y)$ is an upper-semicontinuous function for $||V||$-a.e. $y \in B_{d_g}(\xi,\delta\rho)$, that is, for $||V||$-a.e. $y \in B_{d_g}(\xi,\delta\rho)$, we have that
\begin{equation}
\limsup\limits_{x \to y} \Theta^m(||V||,x) \leq \Theta^m(||V||,y).
\end{equation}
\end{theo}
\begin{proof}
Let us fix $\rho_0 \defeq \rho_0(\injec{r0},K) > 0$ given by \Cref{lemm:mono:section:monotonicity}. Suppose that $V$ has locally bounded first variation in $B_{d_g}(\xi,\rho)$, for some $0 < \rho < \rho_0$. Consider $y \in B_{d_g}(\xi,\delta\rho)$ and $\sigma \in ]0,(1-\delta)\rho]$ arbitrary.

The first part of the proof is to establish a monotonicity formula for varifolds with locally bounded first variation, following exactly the same lines as \Cref{theo:mono-Lp:section:monotonicity}, except for the differences addressed below.

Since we are assuming that $V$ has locally bounded first variation in $B_{d_g}(\xi,\rho)$, by \Cref{prop:generalizedmeancurvaturaL1bounded:section:firstvariation}, \cite{Folland}*{Theorem 2.24, p. 54} and \Cref{equa1proof:theo:mono-Lp:section:monotonicity}, we have that
\footnotesize
\begin{align}\label{equa1proof:theo:densloc}
-\lim\limits_{j \uparrow \infty} \dfrac{\delta_gV\left(X_{h,\varepsilon_j,\sigma}\right)}{\sigma^{m+1}} = -\dfrac{\delta_gV(\mathbf{X}^{h,\sigma})}{\sigma^m},
\end{align}
\normalsize

where $\mathbf{X}^{h,\sigma} \defeq h(x)\frac{u_y(x)}{\sigma}\grad{g_x}(u_y)(x)$. Therefore, 
\footnotesize
\begin{align}\label{equa2proof:theo:densloc}
&\dfrac{d}{d\sigma}\left( \dfrac{1}{\sigma^m} \int_{B_{d_g}(y,\sigma)}  h(x)d||V||(x) \right) \geq \nonumber \\
&\geq \dfrac{1}{\sigma^{m+1}} \int_{\Gr_m(B_{d_g}(y,\sigma))} \left\langle u_y(x)\grad{g_x}(u_y)(x) , P_{S(x)}\left(\grad{g_x}(h)(x)\right)  \right\rangle_{g_x} dV(x,S(x)) \nonumber\\
&+ \dfrac{c(\sigma,K) - 1}{\sigma}\dfrac{m}{\sigma^m} \int_{B_{d_g}(y,\sigma)}  h(x)d||V||(x) \nonumber\\
&+ \dfrac{d}{d \sigma}\left(\int_{\Gr_m(B_{d_g}(y,\sigma))} \dfrac{h(x)}{u_y(x)^m} \left|\left| P_{S(x)^{\perp}}(\grad{g_x}(u_y)(x)) \right|\right|_{g_x}^2 dV(x,S(x))\right) \nonumber\\
&-\dfrac{\delta_gV(\mathbf{X}^{h,\sigma})}{\sigma^m},
\end{align}
\normalsize

Let $t \in ]0,\sigma[$ and $h:\supp(||V||) \to \mathbb{R}$ given by a mollification of $\mathcal{X}_{B_{d_g}(y,t)}$. Applying \Cref{equa2proof:theo:densloc} for $t$ and taking the appropriates limits for the function $h$, we have that 
\begin{align}\label{equa3proof:theo:densloc}
\dfrac{d}{dt}\left( \dfrac{||V||(B_{d_g}(y,t))}{t^m} \right) &\geq \dfrac{c(t,K) - 1}{t}m\dfrac{||V||(B_{d_g}(y,t))}{t^m}\nonumber\\
&+ \dfrac{d}{d t}\left(\int_{\Gr_m(B_{d_g}(y,t))} \dfrac{\left|\left| P_{S(x)^{\perp}}(\grad{g_x}(u_y)(x)) \right|\right|_{g_x}^2}{u_y(x)^m} dV(x,S(x))\right) \nonumber\\
&-\dfrac{\delta_gV(\mathbf{X}^{t})}{t^m},
\end{align}

where $\mathbf{X}^{t} \defeq \mathcal{X}_{B_{d_g}(y,t)}(x)\frac{u_y(x)}{t}\grad{g_x}(u_y)(x)$.

Since we are assuming that $V$ has locally bounded first variation in $B_{d_g}(\xi,\rho)$, by \Cref{prop:generalizedmeancurvaturaL1bounded:section:firstvariation} we have that $||\delta_gV||_{\mathrm{TV}}$ is a Radon measure. Thus, by \cite{FedererGMT}*{Theorem 295, p. 154}, we have that $\Theta^{||V||}(||\delta_gV||_{TV},x)$  there exists for $||V||$-a.e. $x \in B_{d_g}(\xi,\rho)$. Therefore, by definition of density, we have that there exists $0 < C(x) < \infty$ such that 
\begin{equation}\label{equa4proof:theo:densloc}
||\delta_gV||_{\mathrm{TV}}(B_{d_g}(x,t)) \leq C(x)||V||(B_{d_g}(x,t)),
\end{equation}

for  $||V||$-a.e. $x \in B_{d_g}(\xi,\rho)$ and for all $t > 0$ such that $B_{d_g}(x,t) \subset B_{d_g}(\xi,\rho)$. On the other hand, note that $\mathbf{X}^{t} \in \mathfrak{X}^0_c(B_{d_g}(y,t))$ and  $\left|\left|\mathbf{X}^{t}\right|\right|_{L^{\infty}(B_{d_g}(y,t),||V||)}  = 1$. Let $y \in B_{d_g}(\xi, \delta\rho)$, which was chosen arbitrarily at the start of this proof, now be fixed such that \Cref{equa4proof:theo:densloc} holds for it. By \Cref{equa4proof:theo:densloc}, we have that
\begin{align}\label{equa6proof:theo:densloc}
-\dfrac{\delta_gV(\mathbf{X}^{t})}{t^m} \geq -C(y)\dfrac{||V||(B_{d_g}(y,t))}{t^m}.
\end{align}

By \Cref{equa3proof:theo:densloc,equa6proof:theo:densloc} we obtain that
\begin{align}\label{equa7proof:theo:densloc}
\dfrac{d}{dt}\left( \dfrac{||V||(B_{d_g}(y,t))}{t^m} \right) & -\left(C(y) - \dfrac{c(t,K) - 1}{t}m\right) \dfrac{||V||(B_{d_g}(y,t))}{t^m} \geq \nonumber \\
&\geq \dfrac{d}{d t}\left(\int_{\Gr_m(B_{d_g}(y,t))} \dfrac{\left|\left| P_{S(x)^{\perp}}(\grad{g_x}(u_y)(x)) \right|\right|_{g_x}^2}{u_y(x)^m} dV(x,S(x))\right).
\end{align}

Since we have now established a monotonicity formula for varifolds with locally bounded first variation, the proof follows the same lines as the previous cases.
\end{proof}
\begin{rema}\label{rema:densloc}
In the context of \Cref{theo:densloc}, since $\Theta^m(||V||,\cdot)$ is upper-semicontinuous for $||V||$-a.e. $y \in B_{d_g}(\xi,\delta\rho)$,  we have that $\Theta^m(||V||,\cdot)$ is a $||V||$-measurable function in $B_{d_g}(\xi,\delta\rho)$.
\end{rema}
\subsection{Rectifiability theorem}\label{section:vartan}
Before proving the rectifiability theorem, we note that the argument will be carried out using local charts on the manifold. The reason for this is that rectifiability is fundamentally a local measure-theoretic and geometric property, rather than a genuinely metric one. In particular, the proof of rectifiability relies on Euclidean tools such as density estimates, tangent planes, and covering arguments, and is stable under bi-Lipschitz changes of coordinates. Consequently, once we work in sufficiently small coordinate neighborhoods, the problem reduces to the classical Euclidean setting. Moreover, the validity of the argument is independent of the specific geometric constants of the ambient Riemannian manifold, provided the coordinate maps are uniformly bi-Lipschitz on small balls.
\begin{theo}\label{theo:rectifiable}
Let $\xi \in M^n$, $0 < \delta \leq 1$, and $V \in \mathbf{V}_m(M^n)$. There exists $\rho_0 \defeq \rho_0(\injec{r0}) > 0$ such that if $V$ has locally bounded first variation in $B_{d_g}(\xi,\rho)$, for some $0 < \rho < \rho_0$, and $\Theta^m(||V||,x) > 0$ for $||V||$-a.e. $x \in B_{d_g}(\xi,\rho)$, then $V \in \mathbf{RV}_m(B_{d_g}(\xi,\delta\rho))$.
\end{theo}
\begin{proof}
Let us fix $\rho_0 \defeq \rho_0(\injec{r0}) < \injec{r0}$. There exists a smooth coordinate chart
\begin{displaymath}
    \varphi : B_{d_g}(\xi,\rho) \to \mathbb{R}^n,
\end{displaymath}
such that $\varphi$ is bi-Lipschitz on $B_{d_g}(\xi,\rho)$.

Define the pushforward varifold $\widetilde V \defeq \varphi^{\#}V \in \mathbf{V}_m(\mathbb{R}^n)$. Since $\varphi$ is a smooth bi-Lipschitz diffeomorphism, the hypotheses of locally bounded first variation and positivity of the density are preserved under pushforward.

Therefore, by the Euclidean Rectifiability Theorem \cite{Simon}*{Theorem 5.5, p. 246}, it follows that
\begin{displaymath}
    \widetilde V \in \mathbf{RV}_m\bigl(\varphi(B_{d_g}(\xi,\rho))\bigr).
\end{displaymath}

Finally, since rectifiability is preserved under bi-Lipschitz changes of coordinates, we conclude that
\begin{displaymath}
    V \in \mathbf{RV}_m(B_{d_g}(\xi,\delta\rho)),
\end{displaymath}
which completes the proof.
\end{proof}
\section{Riemannian Intrinsic Allard's Interior \texorpdfstring{$\varepsilon$}{E}-Regularity Theorem}\label{part:allardreg}
In the subsequent sections, we will require a more extensive background in Riemannian geometry than was needed in \Cref{part:generalvarifolds}. For a comprehensive overview, we recommend the following references: \cite{Lee}*{Chapter 8}, \cite{Chavel93}*{Chapter IX}, and particularly \cite{Gray04}*{Chapters 2, 3, 8, and 9}.

\subsection{Preliminaries on submanifolds of Riemannian manifolds}\label{section:lemmas}
Before proceeding, let us establish some notation and review key aspects of the theory of submanifolds in abstract Riemannian manifolds.
\begin{defi}\label{defi2:subsection:campos:section:FermiCoordinates}
Let $\Sigma \subset M^n$ be an $m$-dimensional submanifold of $M^n$ of class $\C^2$, $x \in \Sigma$ and $(U,\{x_i\}_{i=1}^n)$ a Fermi coordinate chart at $x$. For $u_{\Sigma} > 0$ we define
\begin{equation}\label{equa1:defi2:subsection:campos:section:FermiCoordinates}
u_{\Sigma}^2 \defeq \sum_{i=m+1}^n x_i^2 \hspace{1cm} { and } \hspace{1cm} N \defeq \sum_{i=m+1}^n \dfrac{x_i}{u_{\Sigma}}\dfrac{\partial}{\partial x_i}.
\end{equation}
\end{defi}
The following lemma characterizes the function $u_{\Sigma}$ and the vector field $N$. A proof can be found in \cite{Gray04}*{Lemma 2.7, p. 23}.
\begin{lemm}\label{lemm1.1:subsection:campos:section:FermiCoordinates}
Let $\Sigma \subset M^n$ be an $m$-dimensional submanifold of $M^n$ of class $\C^2$ and $x \in M^n$. Suppose that there exists a unique geodesic $\gamma$ from $x$ to $\Sigma$ meeting $\Sigma$ orthogonally. Then
\begin{displaymath}
u_{\Sigma}(x) = \dist_g(x,\Sigma) \hspace{1cm} \text{ and } \hspace{1cm} N(\gamma(s)) = \gamma'(s).
\end{displaymath}
\end{lemm}
\begin{defi}\label{defi:tube:subsection:ShapeOperator:section:FermiCoordinates}
Let $\Sigma \subset M^n$ be an $m$-dimensional submanifold of $M^n$ of class $\C^2$ and $r \geq 0$. We define the tube of radius $r$ around $\Sigma$, denoted by $\mathrm{Tube}(\Sigma,r)$, as the set of points in $M^n$ such that there exists a geodesic parting through the points, meeting $\Sigma$ orthogonally and of length less than or equal to $r$.
\end{defi}
Note that we can visualize a tubular neighborhood around an embedded submanifold using the exponential map.
\begin{displaymath}
\mathrm{Tube}(\Sigma,r) = \bigcup\limits_{x \in \Sigma} \{ \exp_x(v) : v \in T_x\Sigma^{\perp} \text{ and } ||v||_{g_x} \leq r\}.
\end{displaymath}
\begin{defi}\label{defi:tubularhiper:subsection:ShapeOperator:section:FermiCoordinates}
Let $\Sigma \subset M^n$ be an $m$-dimensional submanifold of $M^n$ of class $\C^2$ and $0 < t \leq r$. We define the tubular hypersurface at a distance $t$ from $\Sigma$ by
\begin{displaymath}
\Sigma_t \defeq \{ x \in \mathrm{Tube}(\Sigma,r) : \dist_g(p,\Sigma) = t\}.
\end{displaymath}
\end{defi}
The tubular hypersurfaces $\{\Sigma_t\}_{0 <t \leq r}$ form a natural foliation of the set $\mathrm{Tube}(\Sigma,r) \setminus \Sigma$.
\begin{defi}\label{defi:shape:subsection:ShapeOperator:section:FermiCoordinates}
Let $\Sigma \subset M^n$ be an $m$-dimensional submanifold of $M^n$ of class $\C^2$ and $x \in \Sigma$. We define the shape operator of $\Sigma$ at $x$ as the map 
\begin{align*}
S_x : T_x\Sigma \times T_x\Sigma^{\perp} &\to T_x\Sigma \nonumber \\
(X,\nu) &\mapsto S_x(X,\nu) \defeq P_{T_x\Sigma}\left(\nabla^g_X \nu\right).
\end{align*}
\end{defi}\label{defi:secondfundamentalform:subsection:ShapeOperator:section:FermiCoordinates}
\begin{defi}
Let $\Sigma \subset M^n$ be an $m$-dimensional submanifold of $M^n$ of class $\C^2$, $x \in \Sigma$ and $\nu \in T_x\Sigma^{\perp}$. We define the second fundamental form of $\Sigma$ with respect to $\nu$ at $x$ as the map
\begin{align*}
\mathrm{II}_{\nu,x} : T_x\Sigma \times T_x\Sigma &\to \mathbb{R} \nonumber \\
(X,Y) &\mapsto \mathrm{II}_{\nu,x}(X,Y) \defeq \left\langle S_x(X,\nu),Y \right\rangle_{g_x}.
\end{align*}
\end{defi}

Note that the second fundamental form of $\Sigma$ with respect to $\nu$ at $x$ is symmetric. Moreover, for a fixed $\nu \in T_x\Sigma^{\perp}$, we denote by $S_{\nu,x}$ the map 
\begin{align*}
S_{\nu,x} : T_x\Sigma &\to T_x\Sigma \nonumber \\
X &\mapsto S_{\nu,x}(X) \defeq S_x(X,\nu).
\end{align*}

Therefore, $S_{\nu,x}$ is self-adjoint. Now, assume that $\nu \in T_x\Sigma^{\perp}$ has $||\nu||_g = 1$. The $m$ eigenvalues of $S_{\nu,x}$, all of which are real due to its self-adjointness, are called the principal curvatures of $\Sigma$ in the direction $\nu$. Correspondingly, the associated eigenvectors are referred to as the principal curvature vectors of $\Sigma$ in the direction $\nu$.

Consider the scenario where $\Sigma$ has co-dimension $1$, meaning $n = m + 1$. In this case, for each $x \in \Sigma$, there exist exactly two normal vectors $\nu \in T_x\Sigma^{\perp}$ satisfying $||\nu||_g = 1$. It's important to observe that the second fundamental form and the shape operator change sign when we switch the normal vector.

Furthermore, note that for $X \in T_x\Sigma$ arbitrary, we have
\begin{displaymath}
S_{\nu,x}(X) = \nabla^g_X \nu.
\end{displaymath}

Let $\Sigma \subset M^n$ be an $m$-dimensional submanifold of $M^n$ of class $\C^2$, $0 < t \leq \injec{r0}$, $\Sigma_t$ the tubular hypersurface at a distance $t$ from $\Sigma$.

Assume that every point at $\Sigma_t$ has only one nearest point projection onto $\Sigma$ and fix $x \in \Sigma_t$ arbitrarily. We denote by $\proj{\Sigma}(x) \in \Sigma$ the nearest point projection of $x$ onto $\Sigma$ and $\gamma:[0,\injec{r0}[ \to M^n$ the geodesic that minimizes the distance between $\proj{\Sigma}(x)$ and $x$, with the constraint that $\gamma(0) = \proj{\Sigma}(x)$ and $\gamma(t) = x$. By  \Cref{lemm1.1:subsection:campos:section:FermiCoordinates} we have that
\begin{displaymath}
u_{\Sigma}(\gamma(s)) = \dist_g^{\Sigma}\left(\gamma(s)\right) = s \hspace{1cm} \text{ and } \hspace{1cm} N(\gamma(l)) = \gamma'(l).
\end{displaymath}

for $0 \leq s \leq t$ and $0 < l \leq t$. Consider ${f_1(\proj{\Sigma}(x)),\ldots,f_m(\proj{\Sigma}(x))}$ the principal curvature vectors of $\Sigma$ at $\proj{\Sigma}(x)$ in the direction $\gamma'(0)$ and $\kappa_1(\proj{\Sigma}(x)),\ldots,\kappa_m(\proj{\Sigma}(x))$ their respective principal curvatures. Now, we extend these tangent vectors to unit vector fields $F_1(\proj{\Sigma}(x))(s), \ldots, F_m(\proj{\Sigma}(x))(s)$ along $\gamma$ in such a way that, for each $s \in ]0,t]$ and $j=1,\ldots,m$,
\begin{equation}
S(s)(F_j(\proj{\Sigma}(x))(s)) = \kappa_j(\proj{\Sigma}(x))(s)F_j(\proj{\Sigma}(x))(s),
\end{equation}

where $S(s)$ is given by \Cref{defi:shape:subsection:ShapeOperator:section:FermiCoordinates}. That is, for $j = 1, \ldots, m$, $F_j(\proj{\Sigma}(x))(s)$ is the principal curvature vector of $\Sigma_s$ at $\gamma(s)$ in the direction of $N(\gamma(s)) = \gamma'(s)$ and $\kappa_j(\proj{\Sigma}(x))(s)$ its respective principal curvature, for all $0 < s \leq t$. 

For all $0 < s \leq t$, let $F_{m+2}(\proj{\Sigma}(x))(s), \ldots, F_n(\proj{\Sigma}(x))(s)$ be the remaining principal curvature vectors of $\Sigma_s$ at $\gamma(s)$ in the direction $N(\gamma(s)) = \gamma'(s)$ with principal curvatures\\ $\kappa_{m+2}(\proj{\Sigma}(x))(s), \ldots, \kappa_n(\proj{\Sigma}(x))(s)$, respectively. That is, for each $s \in ]0,t]$ and $j=1,\ldots,m$,
\begin{displaymath}
S(s)(F_j(\proj{\Sigma}(x))(s)) = \kappa_j(s)F_j(\proj{\Sigma}(x))(s).
\end{displaymath}

Let $F_{m+1}(\proj{\Sigma}(x))(s) \defeq \gamma'(s)$, for all $0 < s \leq t$. Then $s \mapsto \{F_1(\proj{\Sigma}(x))(s), \ldots, F_n(\proj{\Sigma}(x))(s)\}$ is an orthonormal frame field along $\gamma$, for all $0 < s \leq t$. We will also denote, for all $0 < s \leq t$, by 
\begin{displaymath}
R^g(s)(X) \defeq R^{g}(X,N)N
\end{displaymath}

for all $X \in T_{\gamma(s)}\Sigma_s$, and 
\begin{displaymath}
\mathrm{II}(s)(X,Y) \defeq \mathrm{II}_{N,x}(X,Y)
\end{displaymath}

for all $X,Y \in T_{\gamma(s)}\Sigma_s$.

The following lemma is a standard consequence of \cite{Gray04}*{Lemma 2.8, p. 23}, Generalized Gauss Lemma \cite{Gray04}*{Lemma 2.11, p. 26}, and results from Elliptic Regularity Theory \cite{EvansPDE}*{Chapter 5, p. 251}. The main difficulty arises from the fact that $\Sigma$ is only of class $\C^2$, meaning that $u_{\Sigma}$ is of class $\C^1$, and $\grad{g}(u_{\Sigma})$ is of class $\C^0$. Consequently, $S(t)$ is also of class $\C^0$, requiring us to consider distributional derivatives. However, by Elliptic Regularity Theory, we can conclude \textit{a posteriori} that the derivative exists in the usual sense.
\begin{lemm}\label{lemm1:subsection:ShapeOperator:section:FermiCoordinates}
Let $\Sigma \subset M^n$ be an $m$-dimensional submanifold of $M^n$ of class $\C^2$,  $0 < t \leq \injec{r0}$, $\Sigma_t$ a tubular hypersurface at a distance $t$, and $x \in \Sigma_t$. Then, for $X,Y \in T_x\Sigma_t$, we have
\begin{enumerate}
    \item\label{item1statement:lemm1:subsection:ShapeOperator:section:FermiCoordinates} $\hess_g(u_{\Sigma})(X,X) = \mathrm{II}(t)(X,X)$.
    \item\label{item2statement:lemm1:subsection:ShapeOperator:section:FermiCoordinates} $S'(t)(X) = - S^2(t)(X) - R^g(t)(X)$.
    \item\label{item3statement:lemm1:subsection:ShapeOperator:section:FermiCoordinates} $\kappa_j(\proj{\Sigma}(x))'(t) = - \kappa_j(\proj{\Sigma}(x))^2(s) - \sec_g(F_j(t),N)$.
\end{enumerate}
\end{lemm}
\begin{defi}\label{defi:cothdifferent:subsection:ShapeOperator:section:FermiCoordinates}
Let $\kappa,K \in \mathbb{R}$ with $K>0$. Define, for $t \in \left[0, \frac{1}{\sqrt{K}}\left(\pi-\cot^{-1}\left(\frac{\kappa}{\sqrt{K}}\right)\right)\right[$: 
\begin{displaymath}
\cot_{K,\kappa}(t) \defeq \sqrt{K}\cot\left( \sqrt{K}t + \cot^{-1}\left(\dfrac{\kappa}{\sqrt{K}}\right) \right).
\end{displaymath}
\end{defi}
\begin{rema}\label{rema:cothdifferent:subsection:ShapeOperator:section:FermiCoordinates}
By straightforward calculation, we can deduce that
\begin{equation}\label{equa1:rema:cothdifferent:subsection:ShapeOperator:section:FermiCoordinates}
\begin{cases}
\cot_{K,\kappa}'(t) = - \cot_{K,\kappa}^2(t) - K \\
\cot_{K,\kappa}(0) = \kappa,
\end{cases}
\end{equation}

and
\begin{equation}\label{equa2:rema:cothdifferent:subsection:ShapeOperator:section:FermiCoordinates}
\lim\limits_{t \uparrow \infty} \cot^{-1}(t) = 0.
\end{equation}

Moreover, we have
\begin{displaymath}
\cot_{K,\kappa}(t) = \kappa  - \left(\kappa^2 + K\right)t + O\left(t^2\right).
\end{displaymath}

Therefore, there exists $t_0 \defeq t_0(K,\kappa) > 0$ such that, for all $0 < t < t_0$, we have
\begin{equation}\label{equa3:rema:cothdifferent:subsection:ShapeOperator:section:FermiCoordinates}
\kappa - 2(\kappa^2 + K)t \leq \cot_{K,\kappa}(t) \leq \kappa - \frac{\kappa^2+K}{2}t
\end{equation}

On the other hand, in the notation of \Cref{defi:cotb}, we have
\begin{displaymath}
t\cot_K(t) = 1 - \dfrac{K}{3}t^2 + O\left(t^4\right).
\end{displaymath}

Therefore, there exists $t_0 \defeq t_0(K) > 0$ such that, for all $0 < t < t_0$, we have
\begin{equation}\label{equa4:rema:cothdifferent:subsection:ShapeOperator:section:FermiCoordinates}
1 - \dfrac{2K}{3}t^2 \leq t\cot_K(t) \leq 1 - \dfrac{K}{6}t^2.
\end{equation}
\end{rema}
Let $\Sigma \subset M^n$ be an $m$-dimensional submanifold of $M^n$ of class $\C^2$, $0 < t \leq \injec{r0}$, $\Sigma_t$ the tubular hypersurface at a distance $t$ from $\Sigma$.

Assume that every point at $\Sigma_t$ has only one nearest point projection onto $\Sigma$ and fix $x \in \Sigma_t$ arbitrarily. 
\begin{itemize}
    \item $j = 1, \ldots, m$:

    By the item \Cref{item3statement:lemm1:subsection:ShapeOperator:section:FermiCoordinates} of \Cref{lemm1:subsection:ShapeOperator:section:FermiCoordinates}, we have that
    \begin{displaymath}
    \begin{cases}
    \kappa_j(\proj{\Sigma}(x))'(t) \geq -\kappa_j(\proj{\Sigma}(x))^2(t) - K,\\
    \kappa_j(\proj{\Sigma}(x))(0) = \kappa_j(\proj{\Sigma}(x)).
    \end{cases}
    \end{displaymath}

    Which implies that $\kappa_j(\proj{\Sigma}(x))(t)$ is a super solution for $x'(t) = - x^2(t) - K$. By \Cref{equa1:rema:cothdifferent:subsection:ShapeOperator:section:FermiCoordinates}, $\cot_{K,\kappa_j(\proj{\Sigma}(x))}(t)$ is a solution of $x'(t) = - x^2(t) - K$ and 
    \begin{displaymath}
    \cot_{K,\kappa_j(\proj{\Sigma}(x))}(0) = \kappa_j(\proj{\Sigma}(x)) = \cot_{K,\kappa_j(\proj{\Sigma}(x))}(0).
    \end{displaymath}

    By \cite{ODETeschl}*{Lemma 1.2, p. 24} we have that
    \begin{equation}\label{equa2:subsection:ShapeOperator:section:FermiCoordinates}
    \kappa_j(\proj{\Sigma}(x))(t) > \cot_{K,\kappa_j(\proj{\Sigma}(x))}(t),
    \end{equation}

    for $t \in \left]0,\min\left\{\frac{1}{\sqrt{K}}\left(\pi-\cot^{-1}\left(\frac{\kappa_j(\proj{\Sigma}(x))}{\sqrt{K}}\right)\right),\injec{r0}\right\}\right[$.   
    \item $j = m+2, \ldots, n$:

    By \Cref{item3statement:lemm1:subsection:ShapeOperator:section:FermiCoordinates} of \Cref{lemm1:subsection:ShapeOperator:section:FermiCoordinates}, we have that
    \begin{displaymath}
    \begin{cases}
    \kappa_j(\proj{\Sigma}(x))'(t)  \geq -\kappa_j(\proj{\Sigma}(x))^2(t) - K\\
    \kappa_j(\proj{\Sigma}(x))(t) \xrightarrow{t \downarrow 0} +\infty.
    \end{cases}
    \end{displaymath}

    Let $\varepsilon > 0$ arbitrary and define $\overline{\kappa_j(\proj{\Sigma}(x))}(t) \defeq \kappa_j(\proj{\Sigma}(x))(t+\varepsilon)$. Note that
    \begin{displaymath}
    \begin{cases}
    \overline{\kappa_j(\proj{\Sigma}(x))}'(t)  \geq -\overline{\kappa_j(\proj{\Sigma}(x))}^2(t) - K\\
    \overline{\kappa_j(\proj{\Sigma}(x))}(0) = \kappa_j(\proj{\Sigma}(x))(\varepsilon).
    \end{cases}
    \end{displaymath}

    Which implies that $\overline{\kappa_j(\proj{\Sigma}(x))}(t)$ is a super solution for $x'(t) = - x^2(t) - K$. By \Cref{equa1:rema:cothdifferent:subsection:ShapeOperator:section:FermiCoordinates}, $\cot_{K,\kappa_j(\proj{\Sigma}(x))(\varepsilon)}(t)$ is a solution of $x'(t) = - x^2(t) - K$ and
    \begin{displaymath}
    \cot_{K,\kappa_j(\proj{\Sigma}(x))(\varepsilon)}(0) = \kappa_j(\proj{\Sigma}(x))(\varepsilon) =  \overline{\kappa_j(\proj{\Sigma}(x))}(0).
    \end{displaymath}

    By \cite{ODETeschl}*{Lemma 1.2, p. 24} we have that
    \begin{displaymath}
    \overline{\kappa_j(\proj{\Sigma}(x))}(t) > \cot_{K,\kappa_j(\proj{\Sigma}(x))(\varepsilon)}(t),
    \end{displaymath}

    for $t \in \left]0,\min\left\{\frac{1}{\sqrt{K}}\left(\pi-\cot^{-1}\left(\frac{\kappa_j(\proj{\Sigma}(x))(\varepsilon)}{\sqrt{K}}\right)\right),\injec{r0}\right\}\right[$. Therefore, by \Cref{equa2:rema:cothdifferent:subsection:ShapeOperator:section:FermiCoordinates} and continuity of $\cot$ in $\left[0,\frac{1}{\sqrt{K}}\left(\pi-\cot^{-1}\left(\frac{\kappa_j(\proj{\Sigma}(x))(\varepsilon)}{\sqrt{K}}\right)\right)\right[$, we have that
    \begin{equation}\label{equa3:subsection:ShapeOperator:section:FermiCoordinates}
    \kappa_j(\proj{\Sigma}(x))(t) > \cot_K(t)
    \end{equation}

    for $t \in ]0,\injec{r0}[$.
\end{itemize}
For the remainder of this section, we fix $\injec{r1} \in \mathbb{R}$ satisfying
\begin{injectivityharmonic}\label{r1}
0 < \injec{r1} < r_H(M^n,g),
\end{injectivityharmonic}

where $r_H(M^n,g)$ denotes the harmonic radius of $(M^n, g)$ as defined in \cite{HebeyHerzlich97}*{Definition 7, p. 579}. Such a choice of $\injec{r1}$ is possible because we assume that $|\sec_g| \leq K$, which implies the existence of $K_1 > 0$ such that $|\mathrm{Ric}_g| \leq K_1$. By \cite{HebeyHerzlich97}*{Theorem 6, p. 577}, it follows that $r_H(M^n,g) > 0$. We now present some fundamental results, using the notation of this paper, that will be repeatedly employed throughout this section. For a detailed presentation, see \cite{Jost84}, particularly Lemma 2.8.4 on page 69, which provides estimates on the principal curvatures
.\begin{lemm}\label{lemm:harmoni:section:lemmas}
Let $\xi \in M^n$. Then, for each $y \in B_{d_g}(\xi,\injec{r1})$ and $S \in Gr(m,T_yM^n)$, there exists a unique $m$-dimensional submanifold of $M^n$ of class $\C^2$, denoted by $\Sigma \defeq \Sigma(y,S)$ and refereed as the \high{harmonic $m$-submanifold of $M^n$ at $(y,S)$}, such that
\begin{enumerate}
    \item\label{hyp1:lemm:harmoni:section:lemmas} $y \in \Sigma$,    
    \item\label{hyp2:lemm:harmoni:section:lemmas} $T_y\Sigma = S$,
    \item\label{hyp3:lemm:harmoni:section:lemmas} $\mathrm{cut}(\Sigma) \cap B_{d_g}(\xi,\injec{r1}) = \emptyset$, and 
    \item\label{hyp4:lemm:harmoni:section:lemmas} $\kappa_j(\proj{\Sigma}(x)) = o(1)$ for all $j=1,\ldots,m$ and all $x \in B_{d_g}(\xi,\injec{r1})$,
\end{enumerate}

where $\proj{\Sigma}(x)$ is the nearest point projection of $x$ onto $\Sigma$ and $\kappa_j\left(\proj{\Sigma}(x)\right)$ are the principal curvatures of $\Sigma$ at $\proj{\Sigma}(x)$.
\end{lemm}
\begin{lemm}\label{lemm:fermitransport:section:lemmas}
Let $\xi \in M^n$, $S \in Gr(m,T_{\xi}M^n)$, $\Sigma \subset M^n$ the harmonic $m$-submanifold of $M^n$ at $(\xi,S)$. Then, for each $y \in B_{d_g}(\xi,\injec{r1})$, there exists a unique $m$-dimensional submanifold of $M^n$ of class $\C^2$, denoted by $\Sigma_y$ and refereed as the \high{Fermi transport of $\Sigma$ from $\xi$ to $y$}, such that
\begin{enumerate}
    \item $y \in \Sigma_y$,
    \item $T_x\Sigma_y = \spn\{F_1(\proj{\Sigma}(x))(t), \ldots, F_m(\proj{\Sigma}(x))(t)\}$,
    \item $\dist_g^{\Sigma(y,T_y\Sigma_y)}(x) = o(1)$, for all $x \in \Sigma_y \cap B_{d_g}(\xi,\injec{r1})$, where $\Sigma(y,T_y\Sigma_y) \subset M^n$ is the harmonic $m$-submanifold of $M$ at $(y,T_y\Sigma_y)$,
    \item $\mathrm{cut}(\Sigma_y) \cap B_{d_g}(\xi,\injec{r1}) = \emptyset$,
    \item $\kappa_j(\proj{\Sigma_y}(x)) = o(1)$, for all $j=1,\ldots,m$ and all $x \in B_{d_g}(\xi,\injec{r1})$.
\end{enumerate}
\end{lemm}
The following lemmas are standard consequences of the fact that if $g$ is of class $\mathcal{C}^2$, then the distance induced by $g$ is bi-Lipschitz equivalent to the Euclidean distance. Moreover, the fact that the first term of the expansion of the Riemannian metric is the Euclidean metric, and the second term depends on the sectional curvature, which we assume to be bounded. With this assumption, we can then make the estimates below.
\begin{lemm}\label{lemm:riemmanianadapt1:section:lemmas}
Let $\xi \in M^n$, $S \in Gr(m,T_{\xi}M^n)$, $\Sigma \subset M^n$ the harmonic $m$-submanifold of $M^n$ at $(\xi,S)$. Then, for each $y \in B_{d_g}(\xi,\injec{r1})$, we have that
\begin{displaymath}
\dfrac{1}{2} \left|\left| P_{S^{\perp}}\left(u_y(\xi)\grad{g}(u_y)(\xi)\right)\right|\right|_{g_{\xi}}^2 \leq \left(\dist_g^{\Sigma}(y)\right)^2 \leq 2\left|\left| P_{S^{\perp}}\left(u_y(\xi)\grad{g}(u_y)(\xi)\right)\right|\right|_{g_{\xi}}^2,
\end{displaymath}

where $u_y(\xi) \defeq d_g(y,\xi)$.
\end{lemm}
\begin{lemm}\label{lemm:riemmanianadapt2:section:lemmas}
Let $\xi \in M^n$, $y_i \in B_{d_g}(\xi,\injec{r1})$, $S_i \in Gr(m,T_{y_i}M^n)$, $\Sigma_i \subset M^n$ the harmonic $m$-submanifold of $M^n$ at $(y_i,S_i)$, for $i = 1,2$. Then, for each $x \in B_{d_g}(\xi,\injec{r1})$, we have that
\begin{displaymath}
\dfrac{1}{2}\left|\left| P_{T_x{\Sigma_1}_x} - P_{T_x{\Sigma_2}_x} \right|\right|_{\mathrm{HS}_{g_x}}^2 \leq \left|\left| P_{T_{\xi}{\Sigma_1}_{\xi}} - P_{T_{\xi}{\Sigma_2}_{\xi}} \right|\right|_{\mathrm{HS}_{g_{\xi}}}^2 \leq 2 \left|\left| P_{T_x{\Sigma_1}_x} - P_{T_x{\Sigma_2}_x} \right|\right|_{\mathrm{HS}_{g_x}}^2.
\end{displaymath}
\end{lemm}
\begin{lemm}\label{lemm:riemmanianadapt3:section:lemmas}
Let $\xi \in M^n$, $S_i \in Gr(m,T_{\xi}M^n)$, $\Sigma_i \subset M^n$ the harmonic $m$-submanifold of $M$ at $(\xi,S_i)$, for $i=1,2$. Then, for each $y \in B_{d_g}(\xi,\injec{r1})$, we have that
\begin{align*}
\dfrac{1}{2} \left|\left| P_{S_1}\left(\exp_{\xi}^{-1}(y)\right) - P_{S_2}\left(\exp_{\xi}^{-1}(y)\right) \right|\right|_{g_{\xi}} &\leq d_g\left(\proj{\Sigma_1}(y),\proj{\Sigma_2}(y)\right) \nonumber \\
&\leq 2 \left|\left| P_{S_1}\left(\exp_{\xi}^{-1}(y)\right) - P_{S_2}\left(\exp_{\xi}^{-1}(y)\right) \right|\right|_{g_{\xi}},
\end{align*}

where $\proj{\Sigma_i}(y)$ is the nearest point projection of $y$ onto $\Sigma_i$, for $i=1,2$.
\end{lemm}
\begin{defi}\label{defi:excess}
Let $\xi \in M^n$, $S \in \Gr(m,T_{\xi}M^n)$, $0 < \sigma < \injec{r1}$, and $V \in \mathbf{RV}_m(M^n)$. We define the \high{tilt-excess of $V$ over the ball $B_{d_{g}}(\xi,\sigma)$ relative to the $m$-dimensional space $S$} by 
\begin{displaymath}
E\left(\xi,\sigma,S,V\right) \defeq \dfrac{1}{\sigma^m}\int_{B_{d_g}(\xi,\sigma)} \left|\left| P_{T_xV} - P_{T_x\Sigma(\xi,S)_x} \right|\right|_{\mathrm{HS}_{g_x}}^2 d||V||(x),
\end{displaymath}

where $\Sigma(\xi,S) \subset M^n$ is the harmonic $m$-submanifold of $M^n$ at $(\xi,S)$, and $\Sigma(\xi,S)_x$ the Fermi transport of $\Sigma(\xi,S)$ from $\xi$ to $x$, for all $x \in \supp(||V||)$.
\end{defi}
In the following Remark, we explain why the point $\xi$ in the tilt-excess definition may not be in the support of $||V||$ and why the tilt-excess is interesting (in our approach) only for points in the support of $||V||$.
\begin{rema}
Note that in the definition of the tilt-excess, $\xi$ is not necessarily in $\supp(||V||)$. Since one of our objectives is to find a small $\sigma$ such that the excess $E(\xi, \sigma, S, V)$ is also small, the only points of interest are those where $\xi \in \supp(||V||)$. That is because if $\xi \not\in \supp(||V||)$, we can simply choose $\sigma_0$ small enough to ensure that the ball $B_{d_g}(\xi, \sigma_0)$ does not intersect $\supp(||V||)$. In this case, it is immediate that $E(\xi, \sigma_0, S, V) = 0$.
\end{rema}
\subsection{Caccioppoli inequality}\label{section:Caccioppoli}
Note that because of the adaptations we had to make in the definition of the tilt-excess, we also obtained a different Caccioppoli inequality, which depends on the principal curvatures of the harmonic submanifold that we considered. Additionally, note that in the Euclidean sense, the harmonic submanifold that we can take is the linear subspace, which has all zero principal curvatures. Thus, our version of the Caccioppoli inequality extends the Euclidean one shown in \cite{Simon}*{Lemma 2.5, Chapter 5, Section 2, p. 131}.
\begin{lemm}[Intrinsic Riemannian Caccioppoli Inequality]\label{lemm:excess:section:DensityEstimatesApp}
Let $m < p < \infty$, $0 < \delta < 1$, $V \in \mathbf{RV}_m(M^n)$, and $\xi \in \supp(||V||)$. There exists $\rho_0 \defeq \rho_0(m,K,\injec{r1}) > 0$ such that if $V$ has $L^p_{\mathrm{loc}}$-bounded generalized mean curvature vector in $B_{d_g}(\xi,\rho)$, for some $\rho < \rho_0$, then, for all $y \in B_{d_g}(\xi,\delta\rho)$, for all $S \in Gr(m,T_yM^n)$, for all $\sigma \in ]0,(1-\delta)\rho]$, and for all $\alpha \in ]0,1[$, we have that 
\footnotesize
\begin{align}\label{equa3statement:lemm:excess:section:DensityEstimatesApp}
E\left(y,\alpha\sigma,S,V\right) &\leq \dfrac{\mathbf{C}}{\sigma^m} \int_{B_{d_g}(y,\sigma)} \left(\dfrac{\dist_g^{\Sigma}(x)}{\sigma}\right)^2  d||V||(x) + \dfrac{2\alpha^{-m}}{\sigma^{m-2}} \int_{B_{d_g}(y,\sigma)} \left|\left| \overrightarrow{H_g}(x)\right|\right|_{g_x}^2  d||V||(x) \nonumber \\
&+ \dfrac{4m\rho\alpha^{-m}}{\sigma^m} \int_{B_{d_g}(y,\sigma)} \dfrac{\dist_g^{\Sigma}(x)}{\sigma} \sum\limits_{j=1}^m \left|\kappa_j(\proj{\Sigma}(x))\right| d||V||(x) \nonumber\\
&+ \dfrac{8m\rho^2\alpha^{-m}}{\sigma^m} \int_{B_{d_g}(y,\sigma)} \left(\dfrac{\dist_g^{\Sigma}(x)}{\sigma}\right)^2 \sum\limits_{j=1}^m \kappa_j(\proj{\Sigma}(x))^2  d||V||(x).
\end{align}
\normalsize

where $\Sigma \subset M$ is the harmonic $m$-submanifold of $M^n$ at $(y,S)$ and $\mathbf{C} = \mathbf{C}(m,n,\alpha,K,\rho) > 0$.
\end{lemm}
\begin{rema}
The constant obtained in \Cref{lemm:excess:section:DensityEstimatesApp} is given explicitly by
\begin{displaymath}
\mathbf{C} = \mathbf{C}(m,n,\alpha,K,\rho) = \alpha^{-m}\left(2 + \frac{1}{(1-\alpha)^2} + 8K\left(\dfrac{m(n-m-1)}{3}+m^2\right)\rho^2\right).
\end{displaymath}

Observe that the part of the constant which does not depend on the sectional curvature bound $K$, namely
\begin{displaymath}
    \alpha^{-m}\left(2+\frac{1}{(1-\alpha)^2}\right),
\end{displaymath}
is exactly the Euclidean constant. In particular, if the ambient manifold has zero sectional curvature, then our formula recovers the Euclidean one. Indeed, in this case the $K$-dependent term vanishes, and moreover the last two lines in \Cref{equa3statement:lemm:excess:section:DensityEstimatesApp} are also equal to zero, since the principal curvatures $\kappa_j(\proj{\Sigma}(x))$ vanish for totally geodesic submanifolds. This is precisely the situation in the Euclidean ambient space, where $\Sigma$ is an affine plane.
\end{rema}
\begin{proof}[Proof of \Cref{lemm:excess:section:DensityEstimatesApp}]
Let us fix $\rho_0 > 0$ arbitrarily. Suppose that $V$ has $L^{p}_{\mathrm{loc}}$-bounded generalized mean curvature vector in $B_{d_g}(\xi,\rho)$, for some $0 < \rho < \rho_0$. Consider $y \in B_{d_g}(\xi,\delta\rho)$, $S \in Gr(m,T_yM^n)$, $\sigma \in ]0,(1-\delta)\rho]$, and $\alpha \in ]0,1[$ arbitrary.

Let $x \in \supp(||V||) \cap B_{d_g}(y,\sigma)$ fixed arbitrarily. Note that
\begin{align}\label{equa1proof:lemm:excess:section:DensityEstimatesApp}
\dfrac{1}{2}\left|\left| P_{T_xV} - P_{T_x\Sigma_x} \right|\right|_{\mathrm{HS}_{g_x}}^2 = m - \trace_{g_x}\left(P_{T_xV} \circ P_{T_x\Sigma_x}\right).
\end{align}

If $x \in \Sigma$, as $V \in \mathbf{RV}_m(M^n)$, there exists a countable $m$-rectifiable set $\Gamma \subset M^n$ and a Borel map $f:\Gamma \to ]0,+\infty[$ such that, for all $B \in \mathcal{B}(\Gr_m(M^n))$, we have that
\begin{equation}\label{equa2proof:lemm:excess:section:DensityEstimatesApp}
V(B) = V(\Gamma,f,g)(B) = \int_{\{x \in \Gamma : (x,T_x\Gamma) \in B\}} f(x) d\mathcal{H}_{d_g}^{m}(x).
\end{equation}

Thus, by \Cref{equa1proof:lemm:excess:section:DensityEstimatesApp,equa2proof:lemm:excess:section:DensityEstimatesApp}, we have that
\begin{align}\label{equa3proof:lemm:excess:section:DensityEstimatesApp}
\int_{B_{d_g}(y,\sigma) \cap \Sigma} \left|\left| P_{T_xV} - P_{T_x\Sigma_x} \right|\right|_{\mathrm{HS}_{g_x}}^2 d||V||(x) = 0.
\end{align}

Now, if $x \not\in  \Sigma$, let $\left(\mathrm{Tube}(\Sigma,\injec{r1}),\left\{x_i\right\}_{i=1}^n\right)$ be a system of Fermi coordinates for $\Sigma$ at $y$. We define
\begin{displaymath}
u_{\Sigma}^2 \defeq \sum\limits_{i = m+1}^n x_i^2,
\end{displaymath}

and, for $u_{\Sigma} > 0$, 
\begin{displaymath}
N \defeq \sum\limits_{i=m+1}^n \dfrac{x_i}{u_{\Sigma}} \dfrac{\partial}{\partial x_i}.
\end{displaymath}

By Generalized Gauss Lemma \cite{Gray04}*{Lemma 2.11, p. 26}, $N = \grad{g}(u_{\Sigma})$ on $\mathrm{Tube}(\Sigma,\injec{r1}) \setminus \Sigma$.

Let $\zeta \in \C^1_c\left(\mathrm{Tube}(\Sigma,\injec{r1})\right)$ be non-negative, $X \defeq \zeta^2 u_{\Sigma} \grad{g}\left(u_{\Sigma}\right)$, $\left\{e_1,\ldots,e_m\right\}$ an orthonormal basis for $T_{x}V$, $t_x \defeq \dist_{g}^{\Sigma}\left(x\right)$, and $\Sigma_{t_x}$ the tubular hypersurface at a distance $t_x$ around $\Sigma$. Since $T_{x}\Sigma_x \subset T_{x}\Sigma_{t_x}$, we have that $\grad{g}(u_{\Sigma}) \perp v$, for all $v \in T_{x}\Sigma_x$. Thus
\footnotesize
\begin{align}\label{equa4proof:lemm:excess:section:DensityEstimatesApp}
\diver_{T_{x}V}\left(X\right) = 2\zeta u_{\Sigma}\left\langle   \grad{g}(u_{\Sigma}) , \left(P_{T_{x}V} - P_{T_{x}\Sigma_x}\right)\left(\grad{g}\left(\zeta\right)\right) \right\rangle_g + \zeta^2\diver_{T_{x}V}\left(u_{\Sigma} \grad{g}(u_{\Sigma})\right). 
\end{align}
\normalsize

Since we are assuming that $V$ has $L^p_{\mathrm{loc}}$-bounded generalized mean curvature vector in $B_{d_g}(\xi,\rho)$, by \Cref{prop:generalizedmeancurvaturaLpbounded:section:firstvariation}, we have that
\begin{displaymath}
\int_{B_{d_g}(y,\sigma)} \diver_{T_{x}V}\left(X\right)(x) d||V||(x) = - \int_{B_{d_g}(y,\sigma)} \left\langle \overrightarrow{H_g}(x),X(x)\right\rangle_{g_x} d||V||(x).
\end{displaymath}

By \Cref{equa4proof:lemm:excess:section:DensityEstimatesApp}, we have that
\footnotesize
\begin{align}\label{equa5proof:lemm:excess:section:DensityEstimatesApp}
\int_{B_{d_g}(y,\sigma)} \zeta^2(x)\dfrac{1}{2} \left|\left| P_{T_xV} - P_{T_x\Sigma_x} \right|\right|_{\mathrm{HS}_{g_x}}^2 d||V||(x) &= - \int_{B_{d_g}(y,\sigma)} \left\langle \overrightarrow{H_g}(x),\zeta^2(x)u_{\Sigma}(x)\grad{g}(u_{\Sigma})(x)\right\rangle_{g_x} d||V||(x) \nonumber \\
&\hspace{-4cm}- \int_{B_{d_g}(y,\sigma)} 2\zeta(x) u_{\Sigma}(x)\left\langle   \grad{g}(u_{\Sigma})(x) , \left(P_{T_{x}V} - P_{T_{x}\Sigma_x}\right)\left(\grad{g}\left(\zeta\right)(x)\right) \right\rangle_{g_x} d||V||(x) \nonumber \\
&\hspace{-4cm}+ \int_{B_{d_g}(y,\sigma)} \zeta^2(x)\left(\dfrac{1}{2} \left|\left| P_{T_xV} - P_{T_x\Sigma_x} \right|\right|_{\mathrm{HS}_{g_x}}^2 - \diver_{T_{x}V}\left(u_{\Sigma} \grad{g}(u_{\Sigma})\right)(x)\right) d||V||(x)
\end{align}
\normalsize

Next, we will proceed to estimate the quantity $\frac{1}{2} \left|\left| P_{T_xV} - P_{T_x\Sigma_x} \right|\right|_{\mathrm{HS}_g}^2 - \diver_{T_{x}V}\left(u_{\Sigma} \grad{g}(u_{\Sigma})\right)$. Firstly, we note that
\begin{align}\label{equa6proof:lemm:excess:section:DensityEstimatesApp}
\diver_{T_{x}V}\left(u_{\Sigma} \grad{g}(u_{\Sigma})\right) = \sum\limits_{i=1}^{m} \left|\left\langle \grad{g}(u_{\Sigma}) ,  e_i \right\rangle_g\right|^2 + u_{\Sigma} \sum\limits_{i=1}^{m} \left\langle    \nabla^g_{e_i}\left(\grad{g}(u_{\Sigma})\right) , e_i \right\rangle_g.
\end{align}

Let $\gamma:[0,\sigma] \to M^n$ be the minimizing geodesic between $x$ and $\proj{\Sigma}(x)$ such that $\gamma(0) = \proj{\Sigma}(x)$ and $\gamma(t_x) = x$, and $t \mapsto \{F_1(\proj{\Sigma}(x))(t), \ldots, F_n(\proj{\Sigma}(x))(t)\}$ the orthonormal frame field along $\gamma$ of principal curvature vectors. Thus
\begin{displaymath}
e_i = \sum_{j=1}^n \left\langle e_i,F_j(\proj{\Sigma}(x))(t_x) \right\rangle_g F_j(\proj{\Sigma}(x))(t_x).
\end{displaymath}

By \Cref{equa6proof:lemm:excess:section:DensityEstimatesApp}, we have that 
\tiny
\begin{align}\label{equa7proof:lemm:excess:section:DensityEstimatesApp}
\diver_{T_{x}V}\left(u_{\Sigma} \grad{g}(u_{\Sigma})\right) = \sum\limits_{i=1}^{m} \left|\left\langle F_{m+1}(\proj{\Sigma}(x))(t_x) ,  e_i \right\rangle_g\right|^2 + u_{\Sigma} \sum\limits_{i=1}^{m} \sum\limits_{\substack{j=1 \\ j \neq m+1}}^n    \left|\left\langle e_i,F_j(\proj{\Sigma}(x))(t_x) \right\rangle_g\right|^2 \kappa_j(\proj{\Sigma}(x))(t_x).
\end{align}
\normalsize

By \Cref{lemm:fermitransport:section:lemmas}, we note that $\left\{F_1(\proj{\Sigma}(x))(t_x), \ldots, F_n(\proj{\Sigma}(x))(t_x)\right\}$ is an orthonormal basis for $T_xM$ such that $\left\{F_1(\proj{\Sigma}(x))(t_x),\ldots,F_m(\proj{\Sigma}(x))(t_x)\right\}$ is an orthonormal basis for $T_{x}\Sigma_{x}$. By using the basis\\ $\left\{F_1(\proj{\Sigma}(x))(t_x), \ldots, F_n(\proj{\Sigma}(x))(t_x)\right\}$ in \Cref{equa1proof:lemm:excess:section:DensityEstimatesApp}, we obtain
\tiny
\begin{align}\label{equa8proof:lemm:excess:section:DensityEstimatesApp}
\dfrac{1}{2}\left|\left| P_{T_xV} - P_{T_x\Sigma_x} \right|\right|_{\mathrm{HS}_{g_x}}^2 = \sum\limits_{i=1}^m \sum\limits_{j=m+1}^n \left|\left\langle e_i , F_j(\proj{\Sigma}(x))(t_x) \right\rangle_{g_x}\right|^2.
\end{align}
\normalsize

Note that $x \in B_{d_g}(y,\sigma)$, then $\proj{\Sigma}(x) \in B_{d_g}(y,\sigma)$. Therefore, since $0 < \delta < 1$ we have that $0 < (1-\delta) < 1$, thus
\begin{equation}\label{equa810proof:lemm:excess:section:DensityEstimatesApp}
u_{\Sigma}(x) = d_g\left(x,\proj{\Sigma}(x)\right) \leq d_g(x,y) + d_g\left(y,\proj{\Sigma}(x)\right) < 2\sigma < 2(1-\delta)\rho < 2\rho.
\end{equation}

Then there exists $\rho_0 \defeq \rho_0(\injec{r1}) > 0$ such that $u_{\Sigma}(x) < \injec{r1}$. By \Cref{equa3:subsection:ShapeOperator:section:FermiCoordinates,equa7proof:lemm:excess:section:DensityEstimatesApp,equa8proof:lemm:excess:section:DensityEstimatesApp} we have that
\footnotesize
\begin{align}\label{equa81proof:lemm:excess:section:DensityEstimatesApp}
\dfrac{1}{2}\left|\left| P_{T_xV} - P_{T_x\Sigma_x} \right|\right|_{\mathrm{HS}_g}^2 - \diver_{T_{x}V}\left(u_{\Sigma} \grad{g}(u_{\Sigma})\right) &\leq \left(1-u_{\Sigma}\cot_K(u_{\Sigma})\right)\sum\limits_{i=1}^{m} \sum\limits_{j=m+1}^n    \left|\left\langle e_i,F_j(\proj{\Sigma}(x))(t_x) \right\rangle_g\right|^2 \nonumber \\
&- u_{\Sigma} \sum\limits_{i=1}^{m} \sum\limits_{j=1}^m    \left|\left\langle e_i,F_j(\proj{\Sigma}(x))(t_x) \right\rangle_g\right|^2 \kappa_j(\proj{\Sigma}(x))(u_{\Sigma}(x)).
\end{align}
\normalsize

By \Cref{lemm:harmoni:section:lemmas} we have that $\kappa_j(\proj{\Sigma}(x)) = o(1)$ for all $j=1,\ldots,m$ and for all $x \in B_{d_g}(\xi,\injec{r1})$, thus, by \Cref{equa810proof:lemm:excess:section:DensityEstimatesApp} there exists $\rho_0 \defeq \rho_0(K,m,\injec{r1}) > 0$ such that
\begin{equation}\label{equa831proof:lemm:excess:section:DensityEstimatesApp}
u_{\Sigma}(x) < \min\left\{\min_{j=1,\ldots,m}\frac{1}{\sqrt{K}}\left(\pi-\cot^{-1}\left(\frac{\kappa_j(\proj{\Sigma}(x))}{\sqrt{K}}\right)\right),\injec{r1}\right\}
\end{equation}

Thus, by \Cref{equa3:rema:cothdifferent:subsection:ShapeOperator:section:FermiCoordinates,equa2:subsection:ShapeOperator:section:FermiCoordinates,equa831proof:lemm:excess:section:DensityEstimatesApp}, there exists $\rho_0 \defeq \rho_0(K,m,\injec{r1}) > 0$ such that
\begin{align}\label{equa82proof:lemm:excess:section:DensityEstimatesApp}
\kappa_j(\proj{\Sigma}(x))(u_{\Sigma}(x)) &\geq \cot_{K,\kappa_j(\proj{\Sigma}(x))}(u_{\Sigma}(x)) \nonumber \\
&\geq \kappa_j(\proj{\Sigma}(x))  - 2\left(\kappa_j(\proj{\Sigma}(x))^2 + K\right)u_{\Sigma}(x),
\end{align}

for all $j=1,\ldots,m$. On the other hand, by \Cref{equa4:rema:cothdifferent:subsection:ShapeOperator:section:FermiCoordinates,equa831proof:lemm:excess:section:DensityEstimatesApp} there exists $\rho_0 \defeq \rho_0(K,m,\injec{r1}) > 0$ such that
\begin{equation}\label{equa83proof:lemm:excess:section:DensityEstimatesApp}
1-u_{\Sigma}(x)\cot_K(u_{\Sigma}(x)) \leq \dfrac{2K}{3}u_{\Sigma}(x)^2.
\end{equation}

Thus, by \Cref{equa81proof:lemm:excess:section:DensityEstimatesApp,equa82proof:lemm:excess:section:DensityEstimatesApp,equa83proof:lemm:excess:section:DensityEstimatesApp}, we have that
\footnotesize
\begin{align}\label{equa9proof:lemm:excess:section:DensityEstimatesApp}
\dfrac{1}{2}\left|\left| P_{T_xV} - P_{T_x\Sigma_x} \right|\right|_{\mathrm{HS}_g}^2 - \diver_{T_{x}V}\left(u_{\Sigma} \grad{g}(u_{\Sigma})\right) \leq  \mathbf{C_1}u_{\Sigma}^2 + mu_{\Sigma} \sum\limits_{j=1}^m \left|\kappa_j(\proj{\Sigma}(x))\right| + 2mu_{\Sigma}^2 \sum\limits_{j=1}^m  \kappa_j(\proj{\Sigma}(x))^2.
\end{align}
\normalsize

where $\mathbf{C_1} = \mathbf{C_1}(m,n,K)>0$. By \Cref{equa5proof:lemm:excess:section:DensityEstimatesApp,equa9proof:lemm:excess:section:DensityEstimatesApp}, Cauchy-Schwarz inequality, and the fact that $\left|\left| \grad{g}(u_{\Sigma})(x) \right|\right|_{g_x} = 1$ for all $x \in \mathrm{Tube}(\Sigma,\injec{r1}) \setminus \Sigma$, we have 
\tiny
\begin{align}\label{equa10proof:lemm:excess:section:DensityEstimatesApp}
\int_{B_{d_g}(y,\sigma)} \zeta^2(x)\dfrac{1}{2} \left|\left| P_{T_xV} - P_{T_x\Sigma_x} \right|\right|_{\mathrm{HS}_{g_x}}^2 d||V||(x) &\leq  \int_{B_{d_g}(y,\sigma)} \zeta^2(x)u_{\Sigma}(x) \left|\left| \overrightarrow{H_g}(x)\right|\right|_{g_x}  d||V||(x) \nonumber \\
&+ \int_{B_{d_g}(\xi,\rho)} 2\zeta(x)u_{\Sigma}(x) \left|\left| \left(P_{T_{x}V} - P_{T_{x}\Sigma_x}\right)\left(\grad{g}\left(\zeta\right)(x)\right) \right|\right|_{g_x} d||V||(x) \nonumber \\
&+ \int_{B_{d_g}(y,\sigma)} \zeta^2(x)\mathbf{C_1} u_{\Sigma}^2(x) d||V||(x) \nonumber \\
&+ \int_{B_{d_g}(y,\sigma)} \zeta^2(x)mu_{\Sigma}(x) \sum\limits_{j=1}^m \left|\kappa_j(\proj{\Sigma}(x))\right| d||V||(x) \nonumber\\
&+ \int_{B_{d_g}(y,\sigma)} \zeta^2(x)2mu_{\Sigma}^2(x) \sum\limits_{j=1}^m  \kappa_j(\proj{\Sigma}(x))^2 d||V||(x).
\end{align}
\normalsize

By a standard calculation using Hölder's inequality, we obtain\\ $\left|\left| \left(P_{T_{x}V} - P_{T_{x}\Sigma_x}\right)\left(\grad{g}\left(\zeta\right)(x)\right) \right|\right|_{g_x} \leq \left|\left| P_{T_xV} - P_{T_x\Sigma_x} \right|\right|_{\mathrm{HS}_{g_x}} \left|\left| \grad{g}\left(\zeta\right)(x) \right|\right|_{g_x}$ which, together with \Cref{equa10proof:lemm:excess:section:DensityEstimatesApp} and some computations, implies that
\tiny
\begin{align}\label{equa12proof:lemm:excess:section:DensityEstimatesApp}
\int_{B_{d_g}(y,\sigma)} \zeta^2(x) \left|\left| P_{T_xV} - P_{T_x\Sigma_x} \right|\right|_{\mathrm{HS}_{g_x}}^2 d||V||(x) &\leq  \int_{B_{d_g}(y,\sigma)} 2\zeta^2(x)\dfrac{u_{\Sigma}^2(x)}{\sigma^2}  d||V||(x) +  \int_{B_{d_g}(y,\sigma)} 2\zeta^2(x) \sigma^2\left|\left| \overrightarrow{H_g}(x)\right|\right|_{g_x}^2  d||V||(x) \nonumber \\
&+ \int_{B_{d_g}(y,\sigma)} 16u_{\Sigma}^2(x)\left|\left| \grad{g}\left(\zeta\right)(x) \right|\right|_{g_x}^2  d||V||(x) + \int_{B_{d_g}(y,\sigma)} 4\zeta^2(x)\mathbf{C_1} u_{\Sigma}^2(x) d||V||(x) \nonumber \\
&+ \int_{B_{d_g}(y,\sigma)} 4\zeta^2(x)mu_{\Sigma}(x) \sum\limits_{j=1}^m \left|\kappa_j(\proj{\Sigma}(x))\right| d||V||(x) \nonumber\\
&+ \int_{B_{d_g}(y,\sigma)} 4\zeta^2(x)2mu_{\Sigma}^2(x) \sum\limits_{j=1}^m  \kappa_j(\proj{\Sigma}(x))^2  d||V||(x).
\end{align}
\normalsize

Let $\zeta \equiv 1$ on $B_{d_g}(y,\alpha\sigma)$, $\zeta \equiv 0$ outside of $B_{d_g}(y,\sigma)$ and $\left|\left| \grad{g}\left(\zeta\right)(x) \right|\right|_g \leq \frac{1}{\sqrt{16}(1-\alpha)\sigma}$ on $B_{d_g}(y,\sigma)$. Since $\sigma < (1-\delta)\rho < \rho$, we have that $1 < \frac{\rho}{\sigma}$ and $1 < \frac{\rho^2}{\sigma^2}$, which implies that $u_{\Sigma}(x) < \rho\frac{u_{\Sigma}(x)}{\sigma}$ and $u_{\Sigma}^2(x) < \rho^2 \left(\frac{u_{\Sigma}(x)}{\sigma}\right)^2$. Moreover, since $\alpha \in ]0,1[$ we have that $B_{d_g}(y,\alpha\sigma) \subset B_{d_g}(y,\sigma)$. Thus, also by \Cref{equa12proof:lemm:excess:section:DensityEstimatesApp}, we obtain
\tiny
\begin{align*}
\dfrac{1}{\left(\alpha\sigma\right)^m}\int_{B_{d_g}(y,\alpha\sigma)} \left|\left| P_{T_xV} - P_{T_x\Sigma_x} \right|\right|_{\mathrm{HS}_{g_x}}^2 d||V||(x) &\leq \dfrac{\alpha^{-m}\left(2 + \frac{1}{(1-\alpha)^2} + 4\mathbf{C_1} \rho^2\right)}{\sigma^m} \int_{B_{d_g}(y,\sigma)} \left(\dfrac{u_{\Sigma}(x)}{\sigma}\right)^2  d||V||(x) \nonumber \\
&+ \dfrac{2\alpha^{-m}}{\sigma^{m-2}} \int_{B_{d_g}(y,\sigma)} \left|\left| \overrightarrow{H_g}(x)\right|\right|_{g_x}^2  d||V||(x) \nonumber \\
&+ \dfrac{4m\rho\alpha^{-m}}{\sigma^m} \int_{B_{d_g}(y,\sigma)} \dfrac{u_{\Sigma}(x)}{\sigma} \sum\limits_{j=1}^m \left|\kappa_j(\proj{\Sigma}(x))\right| d||V||(x) \nonumber\\
&+ \dfrac{8m\rho^2\alpha^{-m}}{\sigma^m} \int_{B_{d_g}(y,\sigma)} \left(\dfrac{u_{\Sigma}(x)}{\sigma}\right)^2 \sum\limits_{j=1}^m \kappa_j(\proj{\Sigma}(x))^2  d||V||(x).
\end{align*}
\normalsize

To conclude the proof, we choose $\rho_0 \defeq \rho_0(m, K, \injec{r1}) > 0$ as the minimum value ensuring that all the above estimates hold. Then, \Cref{equa3statement:lemm:excess:section:DensityEstimatesApp} holds.
\end{proof}
\subsection{Density estimates}\label{section:DensityEstimatesApp}
\begin{defi}\label{defi:boundsratio:section:DensityEstimatesApp}
Let $m < p < \infty$; in case $m=1$, we require that $p \geq 2$, $\xi \in M^n$, $0 < \delta < 1$, $0 < \rho < \injec{r1}$, and $V \in \mathbf{V}_m(M^n)$. We say that $V$ satisfies the regularity condition in $B_{d_g}(\xi,\rho)$ with constants $\delta$ and $p$ when $\xi \in \supp(||V||)$,
\begin{equation}\label{statement1:defi:boundsratio:section:DensityEstimatesApp}
\Theta^m(||V||,x) \geq 1,
\end{equation}

for $||V||$-a.e. $x \in B_{d_g}(\xi,\rho)$,
\begin{equation}\label{statement2:defi:boundsratio:section:DensityEstimatesApp}
\dfrac{||V||(B_{d_g}(\xi,\rho))}{\omega_m\rho^m} \leq 1 + \delta.
\end{equation}

and
\begin{equation}\label{statement3:defi:boundsratio:section:DensityEstimatesApp}
\left|\delta_gV(X)\right| \leq \dfrac{\delta}{\rho^{1-\frac{m}{p}}} \left|\left| X \right|\right|_{L^{\frac{p}{p-1}}(M^n,||V||)},
\end{equation}

for all $X \in \mathfrak{X}^0_c(M^n)$ such that $\supp(X) \subset B_{d_g}(\xi,\rho)$. Furthermore, $\mathcal{AC}(\xi,\rho,\delta,p)$ denotes the set al all varifolds that satisfies the regularity condition in $B_{d_g}(\xi,\rho)$ with constants $\delta$ and $p$.
\end{defi}
The conditions presented in \Cref{maintheocomplete} also depend on a parameter $d$, as given by equations \Cref{equa1maintheo,equa2maintheo,equa3maintheo}. However, since we can always modify the density of the varifold without changing its support, we will restrict ourselves to varifolds in $\mathcal{AC}(\xi,\rho,\delta,p)$, that is, when $d = 1$.
\begin{rema}\label{rema:boundsratio:section:DensityEstimatesApp}
Let $m < p < \infty$, $\xi \in M^n$, $0 < \delta < 1$, and $0 < \rho < \injec{r1}$.

If $V \in \mathcal{AC}(\xi,\rho,\delta,p)$ then, for every open subset $W \subset \subset B_{d_g}(\xi,\rho)$, there exists a real constant $C = C(\delta,\rho,m) \defeq \frac{\delta}{\rho^{1-\frac{m}{p}}} > 0$ such that
\begin{displaymath}
\left|\delta_gV(X)\right| \leq C \left|\left| X \right|\right|_{L^{\frac{p}{p-1}}(W,||V||)}
\end{displaymath}
    
holds for every $X \in \mathfrak{X}^0_c(W)$. 

Therefore, by \Cref{prop:generalizedmeancurvaturaLpbounded:section:firstvariation}, we have that the total variation of $\delta_gV$ is a Radon measure. Furthermore, there exists a $||V||$-measurable function $\overrightarrow{H_g}: B_{d_g}(\xi,\rho) \to TB_{d_g}(\xi,\rho)$ such that
\begin{displaymath}
\delta_g V(X)=-\int_{B_{d_g}(\xi,\rho)}\left\langle \overrightarrow{H_g},X\right\rangle_gd||V||
\end{displaymath}
    
and $||\overrightarrow{H_g}||_{L^p(W,||V||)} \leq C$, for every open subset $W \subset \subset B_{d_g}(\xi,\rho)$ and $X\in\mathfrak{X}^0_c(W)$.
\end{rema}
\begin{rema}\label{rema:rectifiableAC}
Let $m < p < \infty$, $\xi \in M^n$, $0 < \delta < 1$, and $0 < \rho < \injec{r1}$.

If $V \in \mathcal{AC}(\xi,\rho,\delta,p)$ then $V \in \mathbf{RV}_m\left(B_{d_g}(\xi,\delta\rho)\right)$. In fact, let $W \subset \subset B_{d_g}(\xi,\rho)$ and $X \in \mathfrak{X}^0_c(W)$
\begin{align*}
|\delta_gV(X)| \leq \dfrac{\delta}{\rho^{1-\frac{m}{p}}}\left(\omega_m\rho^m\right)^{\frac{p}{p-1}}\left(1+\delta\right)^{\frac{p}{p-1}} ||X||_{L^{\infty}(W,||V||)}.
\end{align*}

Which implies that $V$ has locally bounded first variation in $B_{d_g}(\xi,\rho)$. Therefore, by \Cref{theo:rectifiable}, we conclude that $V \in \mathbf{RV}_m\left(B_{d_g}(\xi,\delta\rho)\right)$.
\end{rema}
\begin{lemm}\label{lemm:density:section:DensityEstimatesApp}
Let $m < p < \infty$ and  $\xi \in M^n$. There exist $\delta_0 \defeq \delta_0(m,p,K,\injec{r1}) > 0$ and $\rho_0 \defeq \rho_0(m,p,K,\injec{r1}) > 0$ such that if $V \in \mathcal{AC}(\xi,\rho,\delta,p)$, for some $0 < \delta < \delta_0$ and for some $0 < \rho < \rho_0$, then, for all $y \in \supp(||V||) \cap B_{d_g}(\xi,2\delta\rho)$ and for all $\sigma \in ]0,(1-2\delta)\rho]$, we have that 
\begin{equation}\label{equa3statement:lemm:density:section:DensityEstimatesApp}
\dfrac{1}{2} \leq 1 - \mathbf{C_1}\delta - \mathbf{C_2}\rho^2 \leq \dfrac{||V||\left(B_{d_g}(y,\sigma)\right)}{\omega_m\sigma^m} \leq 1 + \mathbf{C_1}\delta + \mathbf{C_2}\rho^2 \leq 2,
\end{equation}

where $\mathbf{C_1} = \mathbf{C_1}(m,p) > 0$ and $\mathbf{C_2} = \mathbf{C_2}(K,m) > 0$. Furthermore, if $\rho < \sqrt{\delta}$ then
\begin{equation}\label{equa4statement:lemm:density:section:DensityEstimatesApp}
1 - \mathbf{C_3}\delta \leq \dfrac{||V||\left(B_{d_g}(y,\sigma)\right)}{\omega_m\sigma^m} \leq 1 + \mathbf{C_3}\delta,
\end{equation}

where $\mathbf{C_3} = \mathbf{C_3}(m,p,K) > 0$.
\end{lemm}
\begin{proof}
Let us fix $\delta_0 > 0$ and $\rho_0 > 0$ arbitrarily. Suppose that $V \in \mathcal{AC}(\xi,\rho,\delta,p)$, for some $0 < \delta < \delta_0$ and for some $0 < \rho < \rho_0$. Consider  $y \in \supp(||V||) \cap B_{d_g}(\xi,2\delta\rho)$ and $\sigma \in ]0,(1-2\delta)\rho]$ arbitrary.

By \Cref{rema:boundsratio:section:DensityEstimatesApp}, we have, for all $t \in ]0,(1-2\delta)\rho]$, that
\begin{align*}
\left|\left| \overrightarrow{H_g} \right|\right|_{L^p\left(B_{d_g}(y,t), ||V||\right)} \leq \left(1-\dfrac{m}{p}\right)G,
\end{align*}

where 
\begin{displaymath}
G \defeq \dfrac{\delta}{\left(1-\frac{m}{p}\right)\rho^{1-\frac{m}{p}}}.
\end{displaymath}

By Hölder inequality, for all $t \in ]0,(1-2\delta)\rho]$, we have that
\begin{align}\label{equa2proof:lemm:density:section:DensityEstimatesApp}
\left|\left| \overrightarrow{H_g} \right|\right|_{L^1\left(B_{d_g}(y,t), ||V||\right)} \leq \left(1-\dfrac{m}{p}\right)G \left(||V||\left(B_{d_g}(y,t)\right)\right)^{-\frac{1}{p}}\left(||V||\left(B_{d_g}(y,t)\right)\right).
\end{align}

Also note that by \Cref{rema:boundsratio:section:DensityEstimatesApp} we have, in particular, that $V$ has $L^p_{\mathrm{loc}}$-bounded generalized mean curvature vector in $B_{d_g}(\xi,\rho)$. Let $0 < \delta_0 < \frac{1}{2}$, thus, by \Cref{theo:mono-Lpm:section:monotonicity}, there exists $\rho_0 \defeq \rho_0(\injec{r1},K) > 0$ such that, for all $0 < \gamma < t < (1-2\delta)\rho$, we have that
\footnotesize
\begin{align}\label{equa3proof:lemm:density:section:DensityEstimatesApp}
\left( \dfrac{||V||(B_{d_g}(y,t))}{t^m} \right)^{\frac{1}{p}} &\geq \left( \omega_m \right)^{\frac{1}{p}} - \mathbf{C_4} t^{1-\frac{m}{p}},
\end{align}
\normalsize

where $\mathbf{C_4} = \mathbf{C_4}(\delta,m,p,\rho,K) > 0$. By \Cref{equa4:rema:cothdifferent:subsection:ShapeOperator:section:FermiCoordinates,equa3proof:lemm:density:section:DensityEstimatesApp} and Taylor expansion of $\left(1+\delta\right)^{\frac{1}{p}}$ in $\delta$ around $0$ we have that there exist $\delta_0 \defeq \delta_0(m,K,\injec{r1},p) > 0$ and $\rho_0 \defeq \rho_0(m,K,\injec{r1},p) > 0$ such that 
\begin{align}\label{equa11proof:lemm:density:section:DensityEstimatesApp}
\left(||V||(B_{d_g}(y,t))\right)^{-\frac{1}{p}} &\leq \left(\dfrac{2}{\left( \omega_m \right)^{\frac{1}{p}}}\right)t^{-\frac{m}{p}}.
\end{align}

Finally, by \Cref{equa2proof:lemm:density:section:DensityEstimatesApp,equa11proof:lemm:density:section:DensityEstimatesApp}, we have that
\footnotesize
\begin{align}\label{equa12proof:lemm:density:section:DensityEstimatesApp}
\left|\left| \overrightarrow{H_g} \right|\right|_{L^1\left(B_{d_g}(y,t), ||V||\right)} \leq  \alpha \Lambda \left(\dfrac{t}{(1-\delta)\rho}\right)^{\alpha-1}||V||\left(B_{d_g}(y,t)\right),
\end{align}
\normalsize

where
\begin{displaymath}
\begin{cases}
\Lambda \defeq \dfrac{2\delta}{\left(1-\frac{m}{p}\right)\rho^{1-\frac{m}{p}}\left(\omega_m\right)^{\frac{1}{p}}((1-2\delta)\rho)^{\frac{m}{p}}}, \\
\alpha \defeq 1-\dfrac{m}{p}.
\end{cases}
\end{displaymath}

By \Cref{rema:boundsratio:section:DensityEstimatesApp} we have, in particular, that $V$ has $L^{p}_{\mathrm{loc}}$-bounded generalized mean curvature vector in $B_{d_g}(\xi,\rho)$ which with \Cref{equa12proof:lemm:density:section:DensityEstimatesApp} permits us to apply the \Cref{theo:densLpa:section:monotonicity}. Therefore, there exists $\rho_0 \defeq \rho_0(\injec{r1},K) > 0$, such that the function
\begin{align*}
f : ]0,(1-2\delta)\rho] &\to [0,+\infty[ \nonumber \\
t &\mapsto e^{\lambda(t)}\dfrac{||V||\left(B_{d_g}(y,t)\right)}{t^m}
\end{align*}

is non-decreasing, where
\footnotesize
\begin{align*}
\lambda : ]0,(1-2\delta)\rho] &\to ]0,+\infty[ \nonumber \\
t &\mapsto \left(\dfrac{2\delta}{\left(1-\frac{m}{p}\right)\rho^{1-\frac{m}{p}}\left(\omega_m\right)^{\frac{1}{p}}((1-2\delta)\rho)^{\frac{m}{p}}}\left(\dfrac{t}{(1-2\delta)\rho}\right)^{-\frac{m}{p}} - \dfrac{c((1-2\delta)\rho,K)-1}{(1-2\delta)\rho}m\right)t.
\end{align*}
\normalsize

Since $f$ is non-decreasing on $]0,(1-2\delta)\rho]$ we have that
\begin{align}\label{equa13proof:lemm:density:section:DensityEstimatesApp}
\dfrac{||V||\left(B_{d_g}(y,\sigma)\right)}{\omega_m\sigma^m} &\leq e^{\lambda((1-2\delta)\rho)-\lambda(\sigma)}\dfrac{1+\delta}{(1-2\delta)^m}.
\end{align}

Since $\lambda(\sigma) > 0$, \Cref{equa4:rema:cothdifferent:subsection:ShapeOperator:section:FermiCoordinates} and Taylor expansion of $\delta(1-2\delta)^{1-\frac{m}{p}}$ in $\delta$ around $0$ we have that there exist $\rho_0 \defeq \rho_0(K) > 0$ and $\delta_0 \defeq \delta_0(m,p) > 0$ such that
\footnotesize
\begin{align}\label{equa17proof:lemm:density:section:DensityEstimatesApp}
\lambda\left((1-2\delta)\rho\right) - \lambda\left(\sigma\right) \leq  \mathbf{C_5}\delta + \dfrac{2K}{3}m\rho^2,
\end{align}
\normalsize
where $\mathbf{C_5} = \mathbf{C_5}(m,p) > 0$. Therefore, by \Cref{equa13proof:lemm:density:section:DensityEstimatesApp,equa17proof:lemm:density:section:DensityEstimatesApp} and Taylor expansion of $e^{\mathbf{C_5}\delta}\frac{1+\delta}{(1-2\delta)^m}$ in $\delta$ around $0$ we have that there exists $\delta_0 \defeq \delta_0(m,p) > 0$ such that 
\begin{align}\label{equa20proof:lemm:density:section:DensityEstimatesApp}
\dfrac{||V||\left(B_{d_g}(y,\sigma)\right)}{\omega_m\sigma^m} \leq  e^{\frac{2K}{3}m\rho^2}\left(1 + 2\left(\dfrac{4}{\left(1-\frac{m}{p}\right)\left(\omega_m\right)^{\frac{1}{p}}} + 2m + 1\right)\delta\right).
\end{align}

Now, let $\sigma_1 < \sigma$ and note that 
\begin{align}\label{equa21proof:lemm:density:section:DensityEstimatesApp}
\dfrac{1}{e^{\lambda(\sigma)}} &\leq \dfrac{||V||\left(B_{d_g}(y,\sigma)\right)}{\omega_m\sigma^m}.
\end{align} 

Since $\sigma \leq (1-2\delta)\rho$ we have that
\begin{align}\label{equa22proof:lemm:density:section:DensityEstimatesApp}
\lambda(\sigma) \leq \dfrac{2}{\left(1-\frac{m}{p}\right)\left(\omega_m\right)^{\frac{1}{p}}}\delta(1-2\delta)^{1-\frac{m}{p}} - \left((1-2\delta)\rho\cot_K((1-2\delta)\rho)-1\right)m.
\end{align}

Since $\lambda(\sigma) > 0$, \Cref{equa4:rema:cothdifferent:subsection:ShapeOperator:section:FermiCoordinates,equa22proof:lemm:density:section:DensityEstimatesApp} and Taylor expansion of $\delta(1-2\delta)^{1-\frac{m}{p}}$ we have that there exist $\rho_0 \defeq \rho_0(K) > 0$ and $\delta_0 \defeq \delta_0(m,p) > 0$ such that
\begin{align}\label{equa23proof:lemm:density:section:DensityEstimatesApp}
\dfrac{1}{e^{\lambda(\sigma)}} &\geq \dfrac{1}{e^{\mathbf{C_5}\delta + \frac{2K}{3}m\rho^2}}.
\end{align}

Therefore, by \Cref{equa20proof:lemm:density:section:DensityEstimatesApp,equa21proof:lemm:density:section:DensityEstimatesApp,equa23proof:lemm:density:section:DensityEstimatesApp} and Taylor expansion of $e^{-\mathbf{C_5}\delta}$ in $\delta$ at $0$, we obtain that there exists $\delta_0 \defeq \delta_0(m,p) > 0$ such that 
\begin{equation}\label{equa27proof:lemm:density:section:DensityEstimatesApp}
e^{-\frac{2K}{3}m\rho^2}\left(1-\mathbf{C_1}\delta\right) \leq \dfrac{||V||\left(B_{d_g}(y,\sigma)\right)}{\omega_m\sigma^m} \leq e^{\frac{2K}{3}m\rho^2}\left(1+\mathbf{C_1}\delta\right),
\end{equation}

By \Cref{equa27proof:lemm:density:section:DensityEstimatesApp}, and Taylor expansion of $e^{\frac{2K}{3}m\rho^2}$ and $e^{-\frac{2K}{3}m\rho^2}$ in $\rho$ around $0$ we have that there exist $\delta_0 \defeq \delta_0(m,p) > 0$ and $\rho_0 \defeq \rho_0(K,m)>0$ such that
\begin{equation}\label{equa32proof:lemm:density:section:DensityEstimatesApp}
1 - \mathbf{C_1}\delta - \mathbf{C_2}\rho^2 \leq \dfrac{||V||\left(B_{d_g}(y,\sigma)\right)}{\omega_m\sigma^m} \leq 1 + \mathbf{C_1}\delta + \mathbf{C_2}\rho^2.
\end{equation}

To conclude the proof, we choose $\delta_0 \defeq \delta_0(m,K,\injec{r1},p) > 0$ and  $\rho_0 \defeq \rho_0(m,K,\injec{r1},p)>0$ as the minimum value ensuring that all the above estimates hold. Then, we have that \Cref{equa3statement:lemm:density:section:DensityEstimatesApp} holds. Now, supposing that $\rho < \sqrt{\delta}$. By \Cref{equa32proof:lemm:density:section:DensityEstimatesApp} we conclude that \Cref{equa4statement:lemm:density:section:DensityEstimatesApp} holds.
\end{proof}
\begin{coro}\label{coro:poslemm:section:DensityEstimatesApp}
Let $m < p < \infty$ and $\xi \in M^n$. There exist $\delta_0 \defeq \delta_0(m,p,K,\injec{r1}) > 0$ and $\rho_0 \defeq \rho_0(m,p,K,\injec{r1}) > 0$ such that if $V \in \mathcal{AC}(\xi,\rho,\delta,p)$, for some $0 < \delta < \delta_0$ and for some $0 < \rho < \rho_0$, then, for all $y \in \supp(||V||) \cap B_{d_g}(\xi,2\delta\rho)$, we have that 
\begin{equation}\label{equa2statement:coro:poslemm:section:DensityEstimatesApp}
\int_{B_{d_g}(y,(1-2\delta)\rho)} \dfrac{\left(\dist_g^{\Sigma(x,T_xV)}(y)\right)^2}{u_y(x)^{m+2}} d||V||(x) \leq 4\left(\mathbf{C_1}\delta + \mathbf{C_2}\rho^2\right)\omega_m,
\end{equation}

where $\Sigma(x,T_xV)$ is the harmonic $m$-submanifold of $M^n$ at $(x,T_xV)$, for all $x \in \supp(||V||)$ where it exists and $u_y(x) \defeq d_g(x,y)$ for all $x \in M^n$, $\mathbf{C_1} = \mathbf{C_1}(m,p) > 0$, and $\mathbf{C_2} = \mathbf{C_2}(K,m) > 0$. Furthermore, if $\rho < \sqrt{\delta}$ then 
\begin{equation}\label{equa3statement:coro:poslemm:section:DensityEstimatesApp}
\int_{B_{d_g}(y,(1-2\delta)\rho)} \dfrac{\left(\dist_g^{\Sigma(x,T_xV)}(y)\right)^2}{u_y(x)^{m+2}} d||V||(x) \leq 4\omega_m\mathbf{C_3}\delta,
\end{equation}
where $\mathbf{C_3} = \mathbf{C_3}(m,p,K) > 0$.
\end{coro}
\begin{proof}
Let us fix $\delta_0 \defeq \delta_0(m,p,K,\injec{r1}) > 0$ and $\rho_0 \defeq \rho_0(m,p,K,\injec{r1}) > 0$ given by \Cref{lemm:density:section:DensityEstimatesApp}. Suppose that $V \in \mathcal{AC}(\xi,\rho,\delta,p)$, for some $0 < \delta < \delta_0$ and for some $0 < \rho < \rho_0$. Consider $y \in \supp(||V||) \cap B_{d_g}(\xi,2\delta\rho)$ arbitrary.

Letting $\sigma \downarrow 0$ in \Cref{equa3statement:lemm:density:section:DensityEstimatesApp}, we have that
\begin{equation}\label{equa1proof:coro:poslemm:section:DensityEstimatesApp}
1 - \mathbf{C_4}\delta - \mathbf{C_5}\rho^2 \leq \Theta^m(||V||,y) \leq 1 + \mathbf{C_4}\delta + \mathbf{C_5}\rho^2 ,
\end{equation}
where $\mathbf{C_4} = \mathbf{C_4}(m,p) > 0$ and $\mathbf{C_5} = \mathbf{C_5}(K,m) > 0$. By \Cref{theo:densLpa:section:monotonicity} there exists $\rho_0 \defeq \rho_0(\injec{r1},K) > 0$ such that, for $0 < \sigma < (1-2\delta)\rho$, we have that
\begin{align}\label{equa2proof:coro:poslemm:section:DensityEstimatesApp}
&e^{\lambda((1-2\delta)\rho)}\dfrac{||V||(B_{d_g}(y,(1-2\delta)\rho))}{(1-2\delta)^m\rho^m} - e^{\lambda(\sigma)}\dfrac{||V||(B_{d_g}(y,\sigma))}{\sigma^m} \geq \nonumber \\
&\geq \int_{\Gr_m\left(B_{d_g}(y,(1-2\delta)\rho) \setminus B_{d_g}(y,\sigma)\right)} \dfrac{\left|\left| P_{S(x)^{\perp}}(\grad{g}(u_y)(x)) \right|\right|_{g_x}^2}{u_y(x)^m}  dV(x,S(x)),
\end{align}

Taking $\sigma \downarrow 0$ in \Cref{equa2proof:coro:poslemm:section:DensityEstimatesApp} and using the same approach adopted in the proof of \Cref{lemm:density:section:DensityEstimatesApp}, we obtain that
\footnotesize
\begin{align}\label{equa3proof:coro:poslemm:section:DensityEstimatesApp}
\int_{\Gr_m(B_{d_g}(y,(1-2\delta)\rho))} \dfrac{\left|\left| P_{S(x)^{\perp}}(\grad{g}(u_y)(x)) \right|\right|^2_{g_x}}{u_y(x)^m} dV(x,S) \leq 2\left(\mathbf{C_4}\delta + \mathbf{C_5}\rho^2\right)\omega_m
\end{align}
\normalsize

By \Cref{rema:rectifiableAC} we have that $V \in \mathbf{RV}_m(B_{d_g}(\xi,\delta\rho))$, which implies that $T_xV$ exists for $||V||$-a.e. $x \in \supp(||V||) \cap B_{d_g}(\xi,\delta\rho)$. Thus, by \Cref{equa3proof:coro:poslemm:section:DensityEstimatesApp} and Disintegration Lemma, we have that
\begin{equation}\label{equa4proof:coro:poslemm:section:DensityEstimatesApp}
\int_{B_{d_g}(y,(1-2\delta)\rho)} \dfrac{\left|\left| P_{(T_xV)^{\perp}}(\grad{g}(u_y)(x)) \right|\right|^2_{g_x}}{u_y(x)^m} d||V||(x) \leq 2\left(\mathbf{C_4}\delta + \mathbf{C_5}\rho^2\right)\omega_m.
\end{equation}

Let $\Sigma(x,T_xV) \subset M^n$ be the harmonic $m$-submanifold of $M^n$ at $(x,T_xV)$ for all $x \in \supp(||V||)$ where $T_xV$ exists, by \Cref{lemm:riemmanianadapt1:section:lemmas} there exists $\rho_0 \defeq \rho_0(K,\injec{r1}) > 0$ such that 
\footnotesize
\begin{equation}\label{equa5proof:coro:poslemm:section:DensityEstimatesApp}
\dfrac{1}{2}\left|\left| P_{(T_xV)^{\perp}}(u_y(x)\grad{g}(u_y)(x)) \right|\right|^2_{g_x} \leq \left(\dist_g^{\Sigma(x,T_xV)}(y)\right)^2 \leq 2\left|\left| P_{(T_xV)^{\perp}}(u_y(x)\grad{g}(u_y)(x)) \right|\right|^2_{g_x},
\end{equation}
\normalsize

for all $x \in B_{d_g}(y,(1-2\delta)\rho)$. 

Therefore, by \Cref{equa4proof:coro:poslemm:section:DensityEstimatesApp,equa5proof:coro:poslemm:section:DensityEstimatesApp}, we obtain that \Cref{equa2statement:coro:poslemm:section:DensityEstimatesApp} holds. To conclude the proof, we choose $\delta_0 \defeq \delta_0(m,K,\injec{r1},p) > 0$ and  $\rho_0 \defeq \rho_0(m,K,\injec{r1},p)>0$ as the minimum value ensuring that all the above estimates hold. Now suppose that $\rho < \sqrt{\delta}$. By \Cref{equa4proof:coro:poslemm:section:DensityEstimatesApp,equa5proof:coro:poslemm:section:DensityEstimatesApp} we conclude that \Cref{equa3statement:coro:poslemm:section:DensityEstimatesApp} holds.
\end{proof}
\subsection{Affine approximation}\label{section:affineapp}
\begin{rema}\label{rema:adjustment}
Let $m < p < \infty$ and $\xi \in M^n$. There exist $\delta_0 \defeq \delta_0(m,p,K,\injec{r1}) > 0$ and $\rho_0 \defeq \rho_0(m,p,K,\injec{r1}) > 0$ such that if $V \in \mathcal{AC}(\xi,\rho,\delta,p)$, for some $0 < \delta < \delta_0$ and for some $0 < \rho < \min\left\{\rho_0,\sqrt{\delta}\right\}$, then, for all $y \in \supp(||V||) \cap B_{d_g}(\xi,\delta\rho)$ and for all $\sigma \in ]0,4\delta\rho]$, we have that
\begin{equation}\label{eq1rm}
\dfrac{1}{2} \leq \dfrac{||V||\left(B_{d_g}(y,\sigma)\right)}{\omega_m\sigma^m} \leq 2,
\end{equation}

and
\begin{equation}\label{eq2rm}
\int_{B_{d_g}(y,2\sigma)} \left(\dist_g^{\Sigma(x,T_xV)}(y)\right)^{2} d||V||(x) \leq 2^{m+2}\sigma^{m+2}4\omega_m\mathbf{C}\delta,
\end{equation}

where $\mathbf{C} = \mathbf{C}(m,p,K) > 0$. In fact, let us fix $\delta_0 \defeq \delta_0(m,p,K,\injec{r1}) > 0$ and $\rho_0 \defeq \rho_0(m,p,K,\injec{r1}) > 0$ given by \Cref{lemm:density:section:DensityEstimatesApp}. Suppose that $V \in \mathcal{AC}(\xi,\rho,\delta,p)$ for some $0 < \delta < \delta_0$ and for some $0 < \rho < \min\left\{\rho_0,\sqrt{\delta}\right\}$. Consider $y \in \supp(||V||) \cap B_{d_g}(\xi,\delta\rho)$ and $\sigma \in ]0,4\delta\rho]$ arbitrary.

Taking $0 < \delta_0 \frac{1}{6}$ we have that $4\delta\rho \leq (1-2\delta)\rho$. Thus, by \Cref{lemm:density:section:DensityEstimatesApp}, we have that \Cref{eq1rm} holds.

On the other hand, note that for every $x \in B_{d_g}(y,2\sigma)$ we have
\begin{align*}
\left( \dist_g^{\Sigma(x,T_xV)}(y) \right)^{2} &\leq \dfrac{\left( \dist_g^{\Sigma(x,T_xV)}(y) \right)^{2}}{u_y(x)^{m+2}} 2^{m+1}\sigma^{m+2}.
\end{align*}
By integrating the above inequality over $B_{d_g}(y,2\sigma)$ with respect to $||V||$, we obtain
\begin{equation}\label{eqre1}
\int_{B_{d_g}(y,2\sigma)} \left( \dist_g^{\Sigma(x,T_xV)}(y) \right)^{2} d||V||(x) \leq \int_{B_{d_g}(y,2\sigma)} \dfrac{\left( \dist_g^{\Sigma(x,T_xV)}(y) \right)^{2}}{u_y(x)^{m+2}} 2^{m+1}\sigma^{m+2} d||V||(x).
\end{equation}

Taking  $0 < \delta_0 < \frac{1}{10}$ we have that $2(4\delta\rho) < (1-2\delta)\rho$. By \Cref{eqre1} and \Cref{coro:poslemm:section:DensityEstimatesApp} there exist $\delta_0 \defeq \delta_0(m,p,K,\injec{r1}) > 0$ and $\rho_0 \defeq \rho_0(m,p,K,\injec{r1}) > 0$ such that \Cref{eq2rm} holds.

To conclude the proof, we choose $\delta_0 \defeq \delta_0(m,K,\injec{r1},p) > 0$ and  $\rho_0 \defeq \rho_0(m,K,\injec{r1},p)>0$ as the minimum value ensuring that all the above estimates hold.
\end{rema}
\begin{lemm}\label{lemm:affine:section:DensityEstimatesApp}
Let $m < p < \infty$ and $\xi \in M^n$. There exist $\delta_0 \defeq \delta_0(m,p,K,\injec{r1}) > 0$ and $\rho_0 \defeq \rho_0(m,p,K,\injec{r1}) > 0$ such that if $V \in \mathcal{AC}(\xi,\rho,\delta,p)$, for some $0 < \delta < \delta_0$ and for some $0 < \rho < \min\left\{\rho_0,\sqrt{\delta}\right\}$, then, for all $y \in \supp(||V||) \cap B_{d_g}(\xi,\delta\rho)$ and for all $\sigma \in ]0,4\delta\rho]$, we have that 
\begin{equation}\label{equa3statement:lemm:affine:section:DensityEstimatesApp}
\sup\limits_{x \in \supp(||V||) \cap B_{d_g}(y,\sigma)} \dist_{g}^{\Sigma(y,S(y))}(x) \leq \mathbf{C}\delta^{\frac{1}{2(m+1)}}\sigma,
\end{equation}
for some $S(y) \in \Gr(m,T_yM^n)$, where $\Sigma(y,S(y)) \subset M^n$ is the harmonic $m$-submanifold of $M^n$ at $(y,S(y))$ and $\mathbf{C} = \mathbf{C}(m,p,K) > 0$.
\end{lemm}
\begin{proof}
Let us fix $\delta_0 \defeq \delta_0(m,p,K,\injec{r1}) > 0$ and $\rho_0 \defeq \rho_0(m,p,K,\injec{r1}) > 0$ given by \Cref{rema:adjustment}. Suppose that $V \in \mathcal{AC}(\xi,\rho,\delta,p)$ for some $0 < \delta < \delta_0$ and for some $0 < \rho < \min\left\{\rho_0,\sqrt{\delta}\right\}$. Consider $y \in \supp(||V||) \cap B_{d_g}(\xi,\delta\rho)$ and $\sigma \in ]0,4\delta\rho]$ arbitrary.

By \Cref{rema:adjustment} we have that
\begin{equation}\label{equa1.1proof:lemm:affine:section:DensityEstimatesApp}
\dfrac{1}{2} \leq \dfrac{||V||\left(B_{d_g}(y,\sigma)\right)}{\omega_m\sigma^m}  \leq 2.
\end{equation}

Recall the general principle that if $W \subset M^n$ is compact and $\eta > 0$, then any maximal pairwise disjoint collection of closed balls $\left\{\overline{B_{d_g}\left(w_i,\frac{\eta}{2}\right)}\right\}_{i=1,\ldots,N}$ with $w_i \in W$, for all $i=1,\ldots,N$, will automatically have the property that $W \subset \cup_{i=1}^{N}\overline{B_{d_g}(w_i,\eta)}$.

Let $\alpha \in ]0,1[$ arbitrary. Using the general principle with $\eta = \delta^{\alpha}\sigma$ we have pairwise disjoint closed balls $\left\{\overline{B_{d_g}\left(y_i.\frac{\delta^{\alpha}\sigma}{2}\right)}\right\}_{i=1,\ldots,N}$ with $y_i \in \supp(||V||) \cap \overline{B_{d_g}\left(y,\frac{\sigma}{4}\right)}$, for all $i = 1, \ldots, N$, for some $N \in \mathbb{N}$, such that
\begin{equation}\label{equa1.2proof:lemm:affine:section:DensityEstimatesApp}
\supp(||V||) \cap \overline{B_{d_g}\left(y,\dfrac{\sigma}{4}\right)} \subset \bigcup\limits_{i=1}^{N} \overline{B_{d_g}\left(y_i,\delta^{\alpha}\sigma\right)}.
\end{equation}

Let $i=1,\ldots, N$ be arbitrary. We claim that $y_i \in \supp(||V||) \cap B_{d_g}\left(\xi,2\delta\rho\right)$ and $\frac{\delta^{\alpha}\sigma}{2} \in ]0,(1-2\delta)\rho]$.

In fact, since $y_i \in \supp(||V||) \cap \overline{B_{d_g}\left(y,\frac{\sigma}{4}\right)}$, $\sigma \in ]0,4\delta\rho]$, and $y \in \supp(||V||) \cap B_{d_g}(\xi,\delta\rho)$ we have that
\begin{align*}
d_g(y_i,\xi) \leq d_g(y_i,y) + d_g(y,\xi) \leq \dfrac{\sigma}{4} + \delta\rho < \dfrac{4\delta\rho}{4} + \delta\rho = 2\delta\rho.
\end{align*}

Furthermore, since $\sigma \in ]0,4\delta\rho]$, there exists $\delta_0 > 0$ such that $\frac{\delta^{\alpha}\sigma}{2} < (1-2\delta)\rho$. Thus, by \Cref{lemm:density:section:DensityEstimatesApp}, we have that
\begin{align}\label{equa2proof:lemm:affine:section:DensityEstimatesApp}
\dfrac{\omega_m\delta^{m\alpha}\sigma^m}{2^{m+1}} \leq ||V||\left(B_{d_g}\left(y_i,\dfrac{\delta^{\alpha}\sigma}{2}\right)\right) \leq ||V||\left(\overline{B_{d_g}\left(y_i,\dfrac{\delta^{\alpha}\sigma}{2}\right)}\right).
\end{align}

On the other hand, since $\overline{B_{d_g}\left(y_i,\frac{\delta^{\alpha}\sigma}{2}\right)} \cap \overline{B_{d_g}\left(y_j,\frac{\delta^{\alpha}\sigma}{2}\right)} = \emptyset$ for all $i,j = 1,\ldots, N$ with $i \neq j$, we have, by \Cref{equa2proof:lemm:affine:section:DensityEstimatesApp}, that
\begin{equation}\label{equa3proof:lemm:affine:section:DensityEstimatesApp}
\dfrac{N\omega_m\delta^{m\alpha}\sigma^m}{2^{m+1}} \leq \sum\limits_{i=1}^{N}||V||\left(\overline{B_{d_g}\left(y_i,\dfrac{\delta^{\alpha}\sigma}{2}\right)}\right) = ||V||\left(\bigcup\limits_{i=1}^N\overline{B_{d_g}\left(y_i,\dfrac{\delta^{\alpha}\sigma}{2}\right)}\right).
\end{equation}

Let $i = 1, \ldots, N$ be arbitrary. We claim that $\overline{B_{d_g}\left(y_i,\frac{\delta^{\alpha}\sigma}{2}\right)} \subset B_{d_g}(y,\sigma)$. Which implies that $\cup_{i=1}^N\overline{B_{d_g}\left(y_i,\frac{\delta^{\alpha}\sigma}{2}\right)} \subset B_{d_g}(y,\sigma)$.

In fact, let $x \in \overline{B_{d_g}\left(y_i,\frac{\delta^{\alpha}\sigma}{2}\right)}$. Since $y_i \in \supp(||V||) \cap \overline{B_{d_g}\left(y,\frac{\sigma}{4}\right)}$ we have that
\begin{align*}
d_g\left(x,y\right) \leq d_g(x,y_i) + d_g(y_i,y) \leq \dfrac{\delta^{\alpha}\sigma}{2} + \dfrac{\sigma}{4} = \sigma\left(\dfrac{\delta^{\alpha}}{2} + \dfrac{1}{4}\right) < \sigma.
\end{align*}

Therefore, by \Cref{equa1.1proof:lemm:affine:section:DensityEstimatesApp,equa3proof:lemm:affine:section:DensityEstimatesApp}, we have that 
\begin{align}\label{equa4proof:lemm:affine:section:DensityEstimatesApp}
N &\leq \dfrac{2^{m+2}}{\delta^{m\alpha}}.
\end{align}

Since $y_i \in \supp(||V||) \cap B_{d_g}\left(\xi,2\delta\rho\right)$ and $\sigma \in ]0,4\delta\rho]$ by \Cref{rema:adjustment} (applied to $y_i$ and $\sigma$) we have that
\begin{equation}\label{equa5proof:lemm:affine:section:DensityEstimatesApp}
\int_{B_{d_g}(y_i,2\sigma)} \left(\dist_g^{\Sigma(x,T_xV)}(y_i)\right)^2 d||V||(x) \leq 2^{m+2}\sigma^{m+2} 4\omega_m\mathbf{C_1}\delta
\end{equation}
holds for all $i = 1,\ldots, N$, where $\mathbf{C_1} = \mathbf{C_1}(m,p,K) > 0$. We claim that $B_{d_g}(y,\sigma) \subset B_{d_g}(y_i,2\sigma)$ for all $i = 1,\ldots, N$.

In fact, let $i = 1,\ldots, N$ arbitrary, and $x \in B_{d_g}(y,\sigma)$, since $y_i \in \supp(||V||) \cap \overline{B_{d_g}\left(y,\frac{\sigma}{4}\right)}$, we have that
\begin{displaymath}
d_g(x,y_i) \leq d_g(x,y) + d_g(y,y_i) \leq \sigma + \dfrac{\sigma}{4} < 2\sigma.
\end{displaymath}

Therefore, by \Cref{equa4proof:lemm:affine:section:DensityEstimatesApp,equa5proof:lemm:affine:section:DensityEstimatesApp}, we have that
\footnotesize
\begin{align*}
\int_{B_{d_g}(y,\sigma)} \sum\limits_{i=1}^N \left(\dist_g^{\Sigma(x,T_xV)}(y_i)\right)^2 d||V||(x) &\leq 2^{2m+6}\sigma^{m+2}\omega_m\dfrac{\mathbf{C_1}\delta}{\delta^{m\alpha}}.
\end{align*}
\normalsize

Thus, for any given $\lambda \geq 1$, we have that
\begin{equation}\label{equa6proof:lemm:affine:section:DensityEstimatesApp}
\sum\limits_{i=1}^N \left(\dist_g^{\Sigma(x,T_xV)}(y_i)\right)^2 \leq 2^{2m+6}\sigma^{m+2}\omega_m\dfrac{\mathbf{C_1}\delta}{\delta^{m\alpha}} \lambda
\end{equation}

holds for $||V||$-a.e. $y \in \supp(||V||) \cap B_{d_g}(y,\sigma) \cap A$ such that $||V||(A) \geq \frac{1}{\lambda}$.

Since $\delta^{\alpha}\sigma < \sigma$ we have by \Cref{equa1.1proof:lemm:affine:section:DensityEstimatesApp} that  $||V||\left(B_{d_g}\left(y,\delta^{\alpha}\sigma\right)\right) \geq \frac{\omega_m\delta^{m\alpha}\sigma^m}{2}$, and we can select $\lambda = \frac{2}{\omega_m\delta^{m\alpha}\sigma^m}$ in such a way that \Cref{equa6proof:lemm:affine:section:DensityEstimatesApp} holds for some $x_0 \in \supp(||V||) \cap B_{d_g}\left(y,\delta^{\alpha}\sigma\right)$.

So we have shown that there is $x_0 \in \supp(||V||) \cap B_{d_g}\left(y,\delta^{\alpha}\sigma\right)$ such that 
\begin{align*}
\dist_g^{\Sigma(x_0,T_{x_0}V)}(y_i) &\leq 2^{m + \frac{7}{2}}\sigma \sqrt{\mathbf{C_1}} \delta^{\frac{1}{2}-m\alpha}
\end{align*}

holds for all $i = 1,\ldots, N$. Now let $\Sigma(x_0,T_{x_0}V)_{y}$ be the Fermi transport of $\Sigma(x_0,T_{x_0}V)$ from $x_0$ to $y$ and note that for all $i=1, \ldots, N$, we have that
\begin{align}\label{equa7proof:lemm:affine:section:DensityEstimatesApp}
\dist_g^{\Sigma(x_0,T_{x_0}V)_{y}}(y_i) \leq \dist_g^{\Sigma(x_0,T_{x_0}V)}(y_i) + d_g(x_0,y).
\end{align}

Define $S(y) \defeq T_y\Sigma(x_0,T_{x_0}V)_y$ and $\Sigma(y,S)$ the harmonic $m$-submanifold of $M$ at $(y,S(y))$. By \Cref{lemm:fermitransport:section:lemmas}, for all $i = 1,\ldots, N$, we have
\begin{align}\label{equa8proof:lemm:affine:section:DensityEstimatesApp}
\dist_g^{\Sigma(y,S(y))}(y_i) \leq \dist_g^{\Sigma(x_0,T_{x_0}V)_{y}}(y_i) + o(1).
\end{align}

Therefore, by \Cref{equa7proof:lemm:affine:section:DensityEstimatesApp,equa8proof:lemm:affine:section:DensityEstimatesApp}, we have for all $i = 1, \ldots, N$
\begin{align*}
\dist_g^{\Sigma(y,S(y))}(y_i) \leq 2^{m + \frac{7}{2}}\sigma \sqrt{\mathbf{C_1}} \delta^{\frac{1}{2}-m\alpha} + \delta^{\alpha}\sigma + \delta^{\alpha}\sigma
\end{align*}

Furthermore, for all $i=1,\ldots, N$ and all $x \in B_{d_g}(y_i,\delta^{\alpha}\sigma)$ we have
\begin{align*}
\dist_g^{\Sigma(y,S(y))}(x) \leq \delta^{\alpha}\sigma + 2^{m + \frac{7}{2}}\sigma \sqrt{\mathbf{C_1}} \delta^{\frac{1}{2}-m\alpha} + 2\delta^{\alpha}\sigma.
\end{align*}

Then, selecting $\alpha$ such that $\frac{1}{2} - m\alpha = \alpha$, that is, $\alpha = \frac{1}{2(m+1)}$, we have, for all $i=1,\ldots, N$ and all $x \in B_{d_g}(y_i,\delta^{\alpha}\sigma)$, that
\begin{align*}
\dist_g^{\Sigma(y,S(y))}(x) \leq  \left(3+2^{m + \frac{7}{2}} \sqrt{\mathbf{C_1}}\right)\delta^{\frac{1}{2(m+1)}}\sigma
\end{align*}

Therefore, by \Cref{equa1.2proof:lemm:affine:section:DensityEstimatesApp} and the fact that $\supp(||V||) \cap B_{d_g}\left(y,\sigma\right) \subset \supp(||V||) \cap \overline{B_{d_g}\left(y,\sigma\right)}$, we have that \Cref{equa3statement:lemm:affine:section:DensityEstimatesApp} holds. To conclude the proof, we choose $\delta_0 \defeq \delta_0(m,K,\injec{r1},p) > 0$ and  $\rho_0 \defeq \rho_0(m,K,\injec{r1},p)>0$ as the minimum value ensuring that all the above estimates hold.
\end{proof}
\begin{coro}\label{coro:Remark212:section:DensityEstimatesApp}
Let $m < p < \infty$ and $\xi \in M^n$. There exist $\delta_0 \defeq \delta_0(m,p,K,\injec{r1})>0$  and $\rho_0 \defeq \rho_0(m,p,K,\injec{r1}) >0$ such that if $V \in \mathcal{AC}(\xi,\rho,\delta,p)$, for some $0 < \delta < \delta_0$ and for some $0 < \rho < \min \left\{\rho_0,\sqrt{\delta}\right\}$, then, for all $y \in \supp(||V||) \cap B_{d_g}(\xi,\delta\rho)$, for all $\sigma \in ]0,4\delta\rho]$, and for all $\alpha \in ]0,1[$, we have that 
\begin{equation}\label{equa3statement:coro:Remark212:section:DensityEstimatesApp}
\sup\limits_{x \in \supp(||V||) \cap B_{d_g}(y,\sigma)} \dist_{g}^{\Sigma(y,S(y))}(x) \leq \mathbf{C_1}\delta^{\frac{1}{2(m+1)}}\sigma,
\end{equation}
where $\mathbf{C_1} = \mathbf{C_1}(m,p,K) > 0$ and
\begin{equation}\label{equa4statement:coro:Remark212:section:DensityEstimatesApp}
E(y,\alpha\sigma,S(y),V) \leq \mathbf{C_2}\delta^{\frac{1}{m+1}}
\end{equation}

for some $S(y) \in \Gr(m,T_yM^n)$, where $\Sigma(y,S(y)) \subset M^n$ is the harmonic $m$-submanifold of $M^n$ at $(y,S(y))$ and $\mathbf{C_2} = \mathbf{C_2}(m,n,p,\alpha,\rho,K) > 0$.
\end{coro}
\begin{proof}
Let us fix $\delta_0 \defeq \delta_0(m,p,K,\injec{r1})>0$  and $\rho_0 \defeq \rho_0(m,p,K,\injec{r1}) >0$ given by \Cref{lemm:affine:section:DensityEstimatesApp}. Suppose that $V \in \mathcal{AC}(\xi,\rho,\delta,p)$, for some $0 < \delta < \delta_0$ and for some $0 < \rho < \min \left\{\rho_0,\sqrt{\delta}\right\}$. Consider $y \in \supp(||V||) \cap B_{d_g}(\xi,\delta\rho)$,  $\sigma \in ]0,4\delta\rho]$, and $\alpha \in ]0,1[$ arbitrary.

By \Cref{lemm:affine:section:DensityEstimatesApp}, we have that
\begin{equation}\label{equa1proof:coro:Remark212:section:DensityEstimatesApp}
\sup\limits_{x \in \supp(||V||) \cap B_{d_g}(y,\sigma)} \dist_{g}^{\Sigma(y,S(y))}(x) \leq \mathbf{C_3}\delta^{\frac{1}{2(m+1)}}\sigma,
\end{equation}

for some $S(y) \in \Gr(m,T_yM^n)$, where $\Sigma(y,S(y)) \subset M^n$ is the harmonic $m$-submanifold of $M^n$ at $(y,S(y))$ and $\mathbf{C_3} = \mathbf{C_3}(m,p,K) > 0$.

By \Cref{rema:boundsratio:section:DensityEstimatesApp} $V$ has $L^p_{\mathrm{loc}}$-bounded generalized mean curvature vector in $B_{d_g}(\xi,\rho)$. Furthermore, taking $0 < \delta_0 < \frac{1}{5}$ we have that $4\delta\rho < (1-\delta)\rho$. Thus, by \Cref{lemm:excess:section:DensityEstimatesApp}, we have that
\footnotesize
\begin{align}\label{equa2proof:coro:Remark212:section:DensityEstimatesApp}
E\left(y,\alpha\sigma,S(y),V\right) &\leq \dfrac{\mathbf{C_4}}{\sigma^m} \int_{B_{d_g}(y,\sigma)} \left(\dfrac{\dist_g^{\Sigma(y,S(y))}(x)}{\sigma}\right)^2  d||V||(x) + \dfrac{2\alpha^{-m}}{\sigma^{m-2}} \int_{B_{d_g}(y,\sigma)} \left|\left| \overrightarrow{H_g}(x)\right|\right|_{g_x}^2  d||V||(x) \nonumber \\
&+ \dfrac{4m\rho\alpha^{-m}}{\sigma^m} \int_{B_{d_g}(y,\sigma)} \dfrac{\dist_g^{\Sigma(y,S(y))}(x)}{\sigma} \sum\limits_{j=1}^m \left|\kappa_j(\proj{\Sigma}(x))\right| d||V||(x) \nonumber\\
&+ \dfrac{8m\rho^2\alpha^{-m}}{\sigma^m} \int_{B_{d_g}(y,\sigma)} \left(\dfrac{\dist_g^{\Sigma(y,S(y))}(x)}{\sigma}\right)^2 \sum\limits_{j=1}^m \kappa_j(\proj{\Sigma}(x))^2  d||V||(x),
\end{align}
\normalsize
where $\mathbf{C_4} = \mathbf{C_4}(m,n,\alpha,K,\rho) > 0$. Therefore, by \Cref{equa1proof:coro:Remark212:section:DensityEstimatesApp,equa2proof:coro:Remark212:section:DensityEstimatesApp} and \Cref{lemm:harmoni:section:lemmas}, we conclude that
\tiny
\begin{align}\label{equa3proof:coro:Remark212:section:DensityEstimatesApp}
E\left(y,\alpha\sigma,S(y),V\right) &\leq \left(\mathbf{C_4}\mathbf{C_3}^2 + 4m\rho\alpha^{-m}\mathbf{C_3} + 8m^2\rho^2\alpha^{-m}\mathbf{C_3}^2\right) \delta^{\frac{1}{m+1}} \dfrac{||V||(B_{d_g}(y,\sigma))}{\sigma^m} \nonumber \\
&+ \dfrac{2\alpha^{-m}}{\sigma^{m-2}} \int_{B_{d_g}(y,\sigma)} \left|\left| \overrightarrow{H_g}(x)\right|\right|_{g_x}^2  d||V||(x).
\end{align}
\normalsize

Since $\sigma < 4\delta\rho$ taking $0 < \delta_0 < \frac{1}{6}$ we have that $\sigma < (1-2\delta)\rho$. On the other hand, $y \in B_{d_g}(\xi,\delta\rho) \subset B_{d_g}(\xi,2\delta\rho)$. Thus, by \Cref{lemm:density:section:DensityEstimatesApp}, there exists $\delta_0 \defeq \delta_0(m,p,K,\injec{r1}) > 0$ and $\rho_0 \defeq \rho_0(m,p,K,\injec{r1}) > 0$ such that  
\begin{equation}\label{equa4proof:coro:Remark212:section:DensityEstimatesApp}
\dfrac{||V||\left(B_{d_g}(y,\sigma)\right)}{\sigma^m} \leq 2\omega_m.
\end{equation}

Denote by $h(x) \defeq \left|\left| \overrightarrow{H_g}(x)\right|\right|_{g_x}^2$. By Hölder inequality, the fact that $B_{d_g}(y,\sigma) \subset B_{d_g}(\xi, \rho)$ and $1 \leq m < p$, we have that 
\begin{equation}\label{equa5proof:coro:Remark212:section:DensityEstimatesApp}
\dfrac{2\alpha^{-m}}{\sigma^{m-2}} \int_{B_{d_g}(y,\sigma)} \left|\left| \overrightarrow{H_g}(x)\right|\right|_{g_x}^2  d||V||(x) \leq 2^{\frac{6p-4m-2}{p}}\alpha^{-m}(\omega_m)^{\frac{p-2}{p}}\delta^{\frac{1}{m+1}}.
\end{equation}

Therefore, by \Cref{equa3proof:coro:Remark212:section:DensityEstimatesApp,equa4proof:coro:Remark212:section:DensityEstimatesApp,equa5proof:coro:Remark212:section:DensityEstimatesApp}, we have that \Cref{equa4statement:coro:Remark212:section:DensityEstimatesApp} holds. To conclude the proof, we choose $\delta_0 \defeq \delta_0(m,K,\injec{r1},p) > 0$ and  $\rho_0 \defeq \rho_0(m,K,\injec{r1},p)>0$ as the minimum value ensuring that all the above estimates hold.
\end{proof}
\subsection{Lipschitz approximation}\label{section:lipschitzapp}
\begin{defi}\label{defi:normalgraph:section:DensityEstimatesApp}
Let $\Sigma \subset M^n$ be an $m$-dimensional of class $\C^2$ submanifold of $M^n$, $\xi \in \Sigma$,\\ $\left(\mathrm{Tube}(\Sigma,\injec{r1}),\{x_i\}_{i=1}^n\right)$ a Fermi coordinate chart for $\Sigma$ at $\xi$, and 
\begin{align*}
f : \Sigma &\to \mathbb{R}^{n-m} \nonumber \\
x &\mapsto \left(f^{m+1}(x),\ldots,f^n(x)\right),
\end{align*}

where $f^j:\Sigma \to \mathbb{R}$ for all $j=m+1,\ldots,n$. We define the normal graph of $f$ over $\Sigma$ by
\begin{displaymath}
\normalgraph{\Sigma}{f} \defeq \left\{\exp_{x}\left(\sum_{j=m+1}^n f^j(x)\dfrac{\partial}{\partial x_j}(x)\right) : x \in \Sigma\right\}.
\end{displaymath}
\end{defi}
\begin{lemm}\label{lemm:lipschitz:section:DensityEstimatesApp}
Let $m < p < \infty$, $\xi \in M^n$ and $0 < L \leq 1$. There exists $\delta_0 \defeq \delta_0(m,n,p,K,\injec{r1})>0$ and $\rho_0 \defeq \rho_0(m,K,\injec{r1},p) > 0$ such that if $V \in \mathcal{AC}(\xi,\rho,\delta,p)$, for some $0 < \delta < (\delta_0L)^{2(m+1)}$ and for some $0 < \rho < \min \left\{ \rho_0 , \sqrt{\delta} \right\}$, then, for all $y \in \supp(||V||) \cap B_{d_g}\left(\xi,\frac{\delta\rho}{2}\right)$, for all $\sigma \in ]0,2\delta\rho]$ and for all $S(y) \in \Gr(m,T_yM^n)$ such that \Cref{equa3statement:coro:Remark212:section:DensityEstimatesApp,equa4statement:coro:Remark212:section:DensityEstimatesApp} hold, we have that there exists a Lipschitz function $\mathbf{f}_{S(y)} : B_{d_g}(y,\sigma) \cap \Sigma(y,S(y)) \to \mathbb{R}^{n-m}$ such that
\begin{equation}\label{equa3statement:lemm:lipschitz:section:DensityEstimatesApp}
\mathrm{Lip}(\mathbf{f}_{S(y)}) \leq L,
\end{equation}
\begin{equation}\label{equa4statement:lemm:lipschitz:section:DensityEstimatesApp}
\sup_{x \in \Sigma(y,S(y)) \cap B_{d_g}(y,2\delta\sigma)} \left|\left|\mathbf{f}_{S(y)}(x)\right|\right|_{\euclid} \leq \mathbf{C_1} \delta^{\frac{1}{2(m+1)}}\sigma,
\end{equation}
where $\mathbf{C_1} = \mathbf{C_1}(m,p,K) > 0$,
\footnotesize
\begin{equation}\label{equa5statement:lemm:lipschitz:section:DensityEstimatesApp}
\dfrac{||V||\mres B_{d_g}(y,\sigma)\left(\supp(||V||) \setminus \normalgraph{\Sigma(y,S(y))}{\mathbf{f}_{S(y)}}\right)}{\sigma^m} \leq \dfrac{\mathbf{C_2}}{L^2}E(y,\sigma,S(y),V) \leq \dfrac{\mathbf{C_3}}{L^2}\delta^{\frac{1}{m+1}},
\end{equation}
\normalsize
where $\mathbf{C_2} = \mathbf{C_2}(m,p,K,\injec{r1}) > 0$ and $\mathbf{C_3} = \mathbf{C_3}(m,n,p,K,\injec{r1}) > 0$, and
\footnotesize
\begin{equation}\label{equa6statement:lemm:lipschitz:section:DensityEstimatesApp}
\dfrac{\mathcal{H}^m_{d_g}\mres B_{d_g}(y,\sigma)\left(\normalgraph{\Sigma(y,S(y))}{\mathbf{f}_{S(y)}} \setminus \supp(||V||)\right)}{\sigma^m}  \leq \dfrac{\mathbf{C_4}}{L^2}E(y,\sigma,S(y),V) \leq \dfrac{\mathbf{C_5}}{L^2}\delta^{\frac{1}{m+1}},
\end{equation}
\normalsize

where where $\mathbf{C_4} = \mathbf{C_4}(m,n,p,K,\injec{r1}) > 0$ and $\mathbf{C_5} = \mathbf{C_5}(m,n,p,K,\injec{r1}) > 0$.
\end{lemm}
\begin{proof}
Let $\delta_0 \defeq \delta_0(m,p,K,\injec{r1}) > 0$ and $\rho_0 \defeq \rho_0(m,p,K,\injec{r1}) > 0$ given by \Cref{lemm:density:section:DensityEstimatesApp}. Suppose that $V \in \mathcal{AC}(\xi,\rho,\delta,p)$, for some $0 < \delta < (\delta_0L)^{2(m+1)}$ and for some $0 < \rho < \min \left\{ \rho_0 , \sqrt{\delta} \right\}$. Consider $y \in \supp(||V||) \cap B_{d_g}\left(\xi,\frac{\delta\rho}{2}\right)$ and $\sigma \in ]0,2\delta\rho]$ arbitrary, and $S(y) \in \Gr(m,T_yM^n)$ such that \Cref{equa3statement:coro:Remark212:section:DensityEstimatesApp,equa4statement:coro:Remark212:section:DensityEstimatesApp} hold. Define
\tiny
\begin{equation}\label{equa2proof:lemm:lipschitz:section:DensityEstimatesApp}
G \defeq \left\{ z \in \supp(||V||) \cap B_{d_g}(y,2\delta\sigma) : \sup\limits_{t \in ]0,(1 - 2\delta)\sigma]} \dfrac{1}{t^m}\int_{B_{d_g}(z,t)} \left|\left| P_{T_xV} - P_{T_x\Sigma(y,S(y))_x} \right|\right|_{\mathrm{HS}_{g_x}}^2 d||V||(x) \leq \delta_0^2L^2\right\}.
\end{equation}
\normalsize

Thus, if $z \in \supp(||V||) \cap B_{d_g}(y,2\delta\sigma) \setminus G$ then there exists $t \in ]0,(1 - 2\delta)\sigma]$ such that
\begin{equation}\label{equa3proof:lemm:lipschitz:section:DensityEstimatesApp}
\delta_0^2L^2 < \dfrac{1}{t^m}\int_{B_{d_g}(z,t)} \left|\left| P_{T_xV} - P_{T_x\Sigma(y,S(y))_x} \right|\right|_{\mathrm{HS}_{g_x}}^2 d||V||(x).
\end{equation}

By the five-times covering lemma we can pick pairwise disjoint balls $B_{d_g}(z_j,t_j)$ such that \Cref{equa3proof:lemm:lipschitz:section:DensityEstimatesApp} holds for $t = t_j \in ]0,(1-2\delta)\sigma]$ and $z = z_j \in \supp(||V||) \cap B_{d_g}(y,2\delta\sigma) \setminus G$, and such that 
\begin{equation}\label{equa4proof:lemm:lipschitz:section:DensityEstimatesApp}
\supp(||V||) \cap B_{d_g}(y,2\delta\sigma) \setminus G \subset \bigcup_{j}B_{d_g}(z_j,5t_j).
\end{equation}

By \Cref{equa3proof:lemm:lipschitz:section:DensityEstimatesApp,equa4proof:lemm:lipschitz:section:DensityEstimatesApp} and since $\delta_0 \defeq \delta_0(m,p,K,\injec{r1}) > 0$ and $\rho_0 \defeq \rho_0(m,p,K,\injec{r1}) > 0$ were given by \Cref{lemm:density:section:DensityEstimatesApp}, we have that
\footnotesize
\begin{equation}\label{equa5proof:lemm:lipschitz:section:DensityEstimatesApp}
||V||\left(\supp(||V||) \cap B_{d_g}(y,2\delta\sigma) \setminus G\right) < \dfrac{2\omega_m5^m}{\delta_0^2L^2} \int_{B_{d_g}(y,\sigma)} \left|\left| P_{T_xV} - P_{T_x\Sigma(y,S(y))_x} \right|\right|_{\mathrm{HS}_{g_x}}^2 d||V||(x).
\end{equation}
\normalsize

We now claim that $G$ is the normal graph Lipschitz function over $\Sigma(y,S(y))$. To check this, let $z_1,z_2$ be distinct points of $G$. Let $\gamma : [0,\injec{r1}[ \to M^n$ be the unique minimizing geodesic between $z_1$ and $z_2$ such that $\gamma(0) = z_1$ and $\gamma(t) = z_2$, where $t \defeq d_g(z_1,z_2)$.

Taking $0 < \delta_0 < \frac{1}{6}$ we have that $4\delta\sigma < (1-2\delta)\sigma$.  Since $z_1,z_2 \in G \subset B_{d_g}(y,2\delta\sigma)$, we obtain that
\begin{equation}\label{equa8proof:lemm:lipschitz:section:DensityEstimatesApp}
t = d_g(z_1,z_2) \leq d_g(z_1,y) + d_g(y,z_2) < 2\delta\sigma + 2\delta\sigma = 4\delta\sigma < (1-2\delta)\sigma.
\end{equation}

Since $z_1 \in G$ we have that 
\begin{equation}\label{equa9proof:lemm:lipschitz:section:DensityEstimatesApp}
\dfrac{1}{t^m}\int_{B_{d_g}\left(z_1,\frac{t}{2}\right)} \left|\left| P_{T_xV} - P_{T_x\Sigma(y,S)_x} \right|\right|_{\mathrm{HS}_{g_x}}^2 d||V||(x) \leq \dfrac{1}{2^m}\delta_0^2L^2.
\end{equation}

Since $\sigma \in ]0,2\delta\rho]$ there exists $\delta_0 > 0$ such that $\sigma < 2\delta\rho < \frac{\rho}{5}$. On the other hand, since $z_1 \in B_{d_g}(y,2\delta\rho)$ and $y \in B_{d_g}\left(\xi,\frac{\delta\rho}{2}\right)$ we have that
\begin{displaymath}
d_g(z_1,\xi) \leq d_g(z_1,y) + d_g(y,\xi) < 2\delta\sigma + \frac{\delta\rho}{2} < \frac{2\delta\rho}{5} +  \frac{\delta\rho}{2} = \dfrac{9}{10}\delta\rho < \delta\rho.
\end{displaymath}

Therefore, $z_1 \in B_{d_g}(\xi,\delta\rho)$, and $t < 4\delta\sigma < 4\delta\rho$. By \Cref{coro:Remark212:section:DensityEstimatesApp} there exists $\delta_0 \defeq \delta_0(m,p,K,\injec{r1}) > 0$ and $\rho_0 \defeq \rho_0(m,p,K,\injec{r1}) > 0$ such that 
\begin{align}\label{equa10proof:lemm:lipschitz:section:DensityEstimatesApp}
\dfrac{1}{t^m}\int_{B_{d_g}\left(z_1,\frac{t}{2}\right)}\left|\left| P_{T_xV} - P_{T_x\Sigma(z_1,S(z_1))_x} \right|\right|_{\mathrm{HS}_{g_x}}^2 d||V||(x) \leq \dfrac{1}{2^m} \mathbf{C_6} \delta_0^2L^2,
\end{align}
for some $S(z_1) \in \Gr(m,T_{z_1}M^n)$, where $\mathbf{C_6} = \mathbf{C_6}(m,n,p,K,\injec{r1}) > 0$. On the other hand, fix $s \in [0,t]$ arbitrarily. By \Cref{lemm:riemmanianadapt2:section:lemmas} we have, for all $x \in \supp(||V||) \cap B_{d_g}\left(z_1,\frac{t}{2}\right)$, that
\footnotesize
\begin{align}\label{equa11proof:lemm:lipschitz:section:DensityEstimatesApp}
\dfrac{1}{2}\left|\left| P_{T_{\gamma(s)}\Sigma(y,S(y))_{\gamma(s)}} - P_{T_{\gamma(s)}\Sigma(z_1,S(z_1))_{\gamma(s)}} \right|\right|_{\mathrm{HS}_{g_{\gamma(s)}}}^2 &\leq \left|\left| P_{T_x\Sigma(y,S(y))_x} - P_{T_x\Sigma(z_1,S(z_1))_x}\right|\right|_{\mathrm{HS}_{g_x}}^2 \nonumber \\
&\leq 2 \left|\left| P_{T_{\gamma(s)}\Sigma(y,S(y))_{\gamma(s)}} - P_{T_{\gamma(s)}\Sigma(z_1,S(z_1))_{\gamma(s)}} \right|\right|_{\mathrm{HS}_{g_{\gamma(s)}}}^2.
\end{align}
\normalsize

Then, by \Cref{equa11proof:lemm:lipschitz:section:DensityEstimatesApp} and \Cref{lemm:density:section:DensityEstimatesApp} there exists $\delta_0 \defeq \delta_0(m,p,K,\injec{r1}) > 0$ and $\rho_0 \defeq \rho_0(m,p,K,\injec{r1}) > 0$ such that
\begin{align}\label{equa12proof:lemm:lipschitz:section:DensityEstimatesApp}
&\dfrac{1}{t^m}\int_{B_{d_g}\left(z_1,\frac{t}{2}\right)}\left|\left| P_{T_x\Sigma(y,S(y))_x} - P_{T_x\Sigma(z_1,S(z_1))_x}\right|\right|_{\mathrm{HS}_{g_x}}^2 d||V||(x) \nonumber \\
&\geq  \dfrac{\omega_m}{2^{m+2}}\left|\left| P_{T_{\gamma(s)}\Sigma(y,S(y))_{\gamma(s)}} - P_{T_{\gamma(s)}\Sigma(z_1,S(z_1))_{\gamma(s)}} \right|\right|_{\mathrm{HS}_{g_{\gamma(s)}}}^2.
\end{align}

For all $x \in B_{d_g}(y,\sigma)$ where $T_xV$ exists we have, by the squared triangle inequality, that 
\begin{align}\label{equa13proof:lemm:lipschitz:section:DensityEstimatesApp}
\left|\left| P_{T_x\Sigma(y,S(y))_x} - P_{T_x\Sigma(z_1,S(z_1))_x}\right|\right|_{\mathrm{HS}_{g_x}}^2 &\leq 2\left|\left| P_{T_x\Sigma(y,S(y))_x} - P_{T_xV}\right|\right|_{\mathrm{HS}_{g_x}}^2 \nonumber \\
&+ 2\left|\left| P_{T_xV} - P_{T_x\Sigma(z_1,S(z_1))_x}\right|\right|_{\mathrm{HS}_{g_x}}^2.
\end{align}

By \Cref{equa8proof:lemm:lipschitz:section:DensityEstimatesApp} we have that  $\frac{t}{2} < t < (1-2\delta)\sigma$, which implies, since $z_1 \in G$, with \Cref{equa2proof:lemm:lipschitz:section:DensityEstimatesApp,equa12proof:lemm:lipschitz:section:DensityEstimatesApp,equa13proof:lemm:lipschitz:section:DensityEstimatesApp}, that
\tiny
\begin{align}\label{equa15proof:lemm:lipschitz:section:DensityEstimatesApp}
\left|\left|P_{\left(T_{\gamma(s)}\Sigma(y,S(y))_{\gamma(s)}\right)^{\perp}}\left(\gamma'(s)\right)\right|\right|_{g_{\gamma(s)}} \leq \left|\left|P_{\left(T_{\gamma(s)}\Sigma(z_1,S(z_1))_{\gamma(s)}\right)^{\perp}}\left(\gamma'(s)\right)\right|\right|_{g_{\gamma(s)}} +\sqrt{\dfrac{8\left(\mathbf{C_7}+ 1\right)}{\omega_m}} \delta_0L \left|\left|\gamma'(s) \right|\right|_{g_{\gamma(s)}},
\end{align}
\normalsize
where $\mathbf{C_7} = \mathbf{C_7}(m,n,p,K,\injec{r1}) > 0$. Integrating \Cref{equa15proof:lemm:lipschitz:section:DensityEstimatesApp} in $[0,t]$ with respect to $ds$ we obtain that
\begin{align}\label{equa16proof:lemm:lipschitz:section:DensityEstimatesApp}
\int_0^t \left|\left|P_{\left(T_{\gamma(s)}\Sigma(y,S(y))_{\gamma(s)}\right)^{\perp}}\left(\gamma'(s)\right)\right|\right|_{g_{\gamma(s)}} ds \leq \left(\mathbf{C_8} +\sqrt{\dfrac{8\left(\mathbf{C_7}+ 1\right)}{\omega_m}} \right)\delta_0 L d_g(z_1,z_2),
\end{align}
where $\mathbf{C_8} = \mathbf{C_8}(m,p,K) > 0$. On the other hand
\begin{align}\label{equa17proof:lemm:lipschitz:section:DensityEstimatesApp}
d_g(z_1,z_2) \leq \int_0^t \left|\left| P_{\left(T_{\gamma(s)}\Sigma(y,S(y))_{\gamma(s)}\right)^{\perp}}\left(\gamma'(s)\right)\right|\right|_{g_{\gamma(s)}} ds + d_g\left(\proj{\Sigma(y,S(y))}(z_1),\proj{\Sigma(y,S(y))}(z_2)\right).
\end{align}

By \Cref{equa16proof:lemm:lipschitz:section:DensityEstimatesApp,equa17proof:lemm:lipschitz:section:DensityEstimatesApp}, we have that
\begin{align}\label{equa18proof:lemm:lipschitz:section:DensityEstimatesApp}
&\int_0^t \left|\left|P_{\left(T_{\gamma(s)}\Sigma(y,S(y))_{\gamma(s)}\right)^{\perp}}\left(\gamma'(s)\right)\right|\right|_{g_{\gamma(s)}} ds \nonumber \\
&\leq \left(\mathbf{C_8} +\sqrt{\dfrac{8\left(\mathbf{C_7}+ 1\right)}{\omega_m}} \right)\delta_0 L \int_0^t \left|\left| P_{\left(T_{\gamma(s)}\Sigma(y,S(y))_{\gamma(s)}\right)^{\perp}}\left(\gamma'(s)\right)\right|\right|_{g_{\gamma(s)}} ds \nonumber \\
&+ \left(\mathbf{C_8} +\sqrt{\dfrac{8\left(\mathbf{C_7}+ 1\right)}{\omega_m}} \right)\delta_0 L d_g\left(\proj{\Sigma(y,S(y))}(z_1),\proj{\Sigma(y,S(y))}(z_2)\right).
\end{align}

Note that there exists $\delta_0 \defeq \delta_0(m,n,p,K,\injec{r1}) > 0$ such that
\begin{equation}\label{equa19proof:lemm:lipschitz:section:DensityEstimatesApp}
\left(\mathbf{C_8} +\sqrt{\dfrac{8\left(\mathbf{C_7}+ 1\right)}{\omega_m}} \right)\delta_0 \leq \dfrac{1}{2}.
\end{equation}

Since $0 < L \leq 1$, by \Cref{equa18proof:lemm:lipschitz:section:DensityEstimatesApp,equa19proof:lemm:lipschitz:section:DensityEstimatesApp}, we conclude that
\begin{align}\label{equa20proof:lemm:lipschitz:section:DensityEstimatesApp}
&\int_0^t \left|\left|P_{\left(T_{\gamma(s)}\Sigma(y,S(y))_{\gamma(s)}\right)^{\perp}}\left(\gamma'(s)\right)\right|\right|_{g_{\gamma(s)}} ds \nonumber \\
&\leq 2\left(\mathbf{C_8} +\sqrt{\dfrac{8\left(\mathbf{C_7}+ 1\right)}{\omega_m}} \right)\delta_0 L d_g\left(\proj{\Sigma(y,S(y))}(z_1),\proj{\Sigma(y,S(y))}(z_2)\right).
\end{align}

By \Cref{equa17proof:lemm:lipschitz:section:DensityEstimatesApp,equa19proof:lemm:lipschitz:section:DensityEstimatesApp,equa20proof:lemm:lipschitz:section:DensityEstimatesApp}, we conclude that
\begin{equation}\label{equa21proof:lemm:lipschitz:section:DensityEstimatesApp}
d_g\left(\proj{\Sigma(y,S(y))}(z_1),\proj{\Sigma(y,S(y))}(z_2)\right) \geq \dfrac{d_g(z_1,z_2)}{1+L} > 0
\end{equation}

Therefore, $\proj{\Sigma(y,S(y))} : G \to \Sigma(y,S)$ is an injective map. Thus,
\begin{displaymath}
\proj{\Sigma(y,S)} : G \to \proj{\Sigma(y,S)}\left(G\right)
\end{displaymath}

is bijective. Let $\left(\mathrm{Tube}(\Sigma(y,S(y)),\injec{r1}),\{x^i\}_{i=1}^n\right)$ be a Fermi coordinate chart for $\Sigma(y,S(y))$ at $y$ and define
\begin{align*}
f_{S(y)} : \proj{\Sigma(y,S(y))}\left(G\right) \subset \Sigma(y,S(y)) &\to \mathbb{R}^{n-m} \nonumber \\
\proj{\Sigma(y,S(y))}(z) &\mapsto \left(f_{S(y)}^{m+1}\left(\proj{\Sigma(y,S(y))}(z)\right) , \ldots, f_{S(y)}^n\left(\proj{\Sigma(y,S)}(z)\right) \right),
\end{align*}

where $f_{S(y)}^i\left(\proj{\Sigma(y,S)}(z)\right) \defeq x^i(z)$, for all $i = m+1, \ldots, n$. By \Cref{lemm1.1:subsection:campos:section:FermiCoordinates}, for all $z \in G$, we have that
\begin{align}\label{equa22proof:lemm:lipschitz:section:DensityEstimatesApp}
\left|\left| f_{S(y)}\left(\proj{\Sigma(y,S(y))}(z)\right) \right|\right|_{\euclid} = \dist_g^{\Sigma(y,S(y))}(z),
\end{align}

We also have $\normalgraph{\Sigma(y,S(y))}{f_{S(y)}} = G$. Therefore, by \Cref{equa3statement:coro:Remark212:section:DensityEstimatesApp,equa22proof:lemm:lipschitz:section:DensityEstimatesApp}, we have that
\begin{align}\label{equa24proof:lemm:lipschitz:section:DensityEstimatesApp}
\sup\limits_{z \in G} \left|\left|f_{S(y)}\left(\proj{\Sigma(y,S(y))}(z)\right)\right|\right|_{\euclid}  \leq \mathbf{C_6}\delta^{\frac{1}{2(m+1)}}2\delta\sigma,
\end{align}
where $\mathbf{C_6} = \mathbf{C_6}(m,p,K) > 0$. By \Cref{equa20proof:lemm:lipschitz:section:DensityEstimatesApp}, for all $z_1,z_2 \in G$, we have 
\begin{align}\label{equa25proof:lemm:lipschitz:section:DensityEstimatesApp}
&\left|\left|f_{S(y)}\left(\proj{\Sigma(y,S)}(z_1)\right) - f_{S(y)}\left(\proj{\Sigma(y,S)}(z_2)\right)\right|\right|_{\euclid} \nonumber \\
&\leq 2\left(\mathbf{C_8} +\sqrt{\dfrac{8\left(\mathbf{C_7}+ 1\right)}{\omega_m}} \right)\delta_0 L d_g\left(\proj{\Sigma(y,S)}(z_1),\proj{\Sigma(y,S)}(z_2)\right).
\end{align}

Then, by Lipschitz Extension Theorem, there exists $\mathbf{f}_{S(y)}:\Sigma(y,S) \cap B_{d_g}(y,2\delta\sigma) \to \mathbb{R}^{n-m}$ such that, by \Cref{equa19proof:lemm:lipschitz:section:DensityEstimatesApp,equa25proof:lemm:lipschitz:section:DensityEstimatesApp}, we have that
\begin{align}\label{equa251proof:lemm:lipschitz:section:DensityEstimatesApp}
\mathrm{Lip}\left(\mathbf{f}_{S(y)}\right) = 2\left(\mathbf{C_8} +\sqrt{\dfrac{8\left(\mathbf{C_7}+ 1\right)}{\omega_m}} \right)\delta_0 L \leq L,
\end{align}

which implies \Cref{equa3statement:lemm:lipschitz:section:DensityEstimatesApp}. Furthermore, 
\begin{displaymath}
{\mathbf{f}_{S(y)}}_{\restriction_{\proj{\Sigma(y,S(y))}(G) \cap B_{d_g}(y,2\delta\sigma)}} \equiv f_{S(y)}.
\end{displaymath}

Note that there exists $\delta_0 > 0$ such that $2\delta < 1$, thus, by \Cref{equa24proof:lemm:lipschitz:section:DensityEstimatesApp}, we also have 
\begin{align*}
\sup\limits_{x \in \Sigma(y,S(y)) \cap B_{d_g}(y,2\delta\sigma)} \left|\left|\mathbf{f}_{S(y)}(x)\right|\right|_{\euclid} \leq \mathbf{C_6}\delta^{\frac{1}{2(m+1)}}\sigma,
\end{align*}

which implies \Cref{equa4statement:lemm:lipschitz:section:DensityEstimatesApp}.

Since $G \subset \normalgraph{\Sigma(y,S(y))}{\mathbf{f}_{S(y)}}$, by \Cref{equa5proof:lemm:lipschitz:section:DensityEstimatesApp}, we have that
\begin{align*}
\dfrac{||V||\mres B_{d_g}(y,2\delta\sigma)\left(\supp(||V||) \setminus \normalgraph{\Sigma(y,S(y))}{\mathbf{f}_{S(y)}}\right)}{\sigma^m} \leq \dfrac{\mathbf{C_3}}{L^2}  \delta^{\frac{1}{m+1}},
\end{align*}

which implies \Cref{equa5statement:lemm:lipschitz:section:DensityEstimatesApp}. It thus remains only to prove \Cref{equa6statement:lemm:lipschitz:section:DensityEstimatesApp}.

To check this, take $z_0 \in  B_{d_g}(y,2\delta\sigma) \cap \normalgraph{\Sigma(y,S)}{\mathbf{f}_{S(y)}} \setminus \supp(||V||)$ and let
\begin{equation}\label{equa26proof:lemm:lipschitz:section:DensityEstimatesApp}
t \defeq \dfrac{3}{2}\dist_g^{\supp(||V||)}\left(z_0\right) \leq \dfrac{3}{2}d_g(z_0,y) \leq \dfrac{3}{2}2\delta\sigma = 3\delta\sigma,
\end{equation}

because $y \in \supp(||V||)$. Note that
\begin{equation}\label{equa27proof:lemm:lipschitz:section:DensityEstimatesApp}
B_{d_g}\left(z_0,\frac{2}{3}t\right) \cap \supp(||V||) = \emptyset,
\end{equation}

and
\begin{displaymath}
\overline{B_{d_g}\left(z_0,\frac{2}{3}t\right)} \cap \supp(||V||) \neq \emptyset.
\end{displaymath}

By an analogous argument to that presented in \Cref{lemm:density:section:DensityEstimatesApp}, we obtain that
\begin{equation}\label{equa55proof:lemm:lipschitz:section:DensityEstimatesApp}
t^m \leq \dfrac{3^me^{\ell(\rho)}}{2^{m-3}\omega_m} \left(||V||\left(B_{d_g}(z_0,t) \setminus \normalgraph{\Sigma(y,S(y))}{\mathbf{f}_{S(y)}}\right) + \sigma^{m} E(y,\sigma,S(y),V)\right).
\end{equation}

Now we observe that the collection of such balls $B_{d_g}(z_0,t)$ by definition cover all of $B_{d_g}(y,2\delta\sigma) \cap \normalgraph{\Sigma(y,S)}{\mathbf{f}_{S(y)}} \setminus \supp(||V||)$, so by the $5$-times covering lemma  and \Cref{equa55proof:lemm:lipschitz:section:DensityEstimatesApp} we can find a pairwise disjoint collection $\left\{B_{d_g}\left(z_j,t_j\right)\right\}_{j \in \mathbb{N}}$ of such balls with
\begin{equation}\label{equa56proof:lemm:lipschitz:section:DensityEstimatesApp}
t_j^m \leq \dfrac{3^me^{\ell(\rho)}}{2^{m-3}\omega_m} \left(||V||\left(B_{d_g}(z_j,t_j) \setminus \normalgraph{\Sigma(y,S(y))}{\mathbf{f}_{S(y)}}\right) +  \sigma^{m} E(y,\sigma,S(y),V)\right)
\end{equation}

for each $j \in \mathbb{N}$ and $B_{d_g}(y,2\delta\sigma) \cap \normalgraph{\Sigma(y,S)}{\mathbf{f}_{S(y)}} \setminus \supp(||V||) \subset \cup_{j \in \mathbb{N}} B_{d_g}\left(z_j,5t_j\right)$.

By \cite{Gray04}*{Theorem 9.12, p. 196} and \Cref{equa26proof:lemm:lipschitz:section:DensityEstimatesApp} we have that $t_j \leq 3\delta\sigma$ and then there exists $\rho_0 \defeq \rho_0(m,K,\injec{r1})>0$ such that
\begin{equation}\label{equa61proof:lemm:lipschitz:section:DensityEstimatesApp}
\mathcal{H}_{d_g}^m\left(B_{d_g}\left(z_j,5t_j\right) \cap \normalgraph{\Sigma(y,S(y))}{\mathbf{f}_{S(y)}}\right) \leq  5^m\omega_m\left(2 + \dfrac{n^2 K 5^2}{6(m+2)}\right)t_i^m.
\end{equation}

Since the collection $\left\{B_{d_g}\left(z_j,t_j\right)\right\}_{j \in \mathbb{N}}$ is disjoint, $B_{d_g}(y,2\delta\sigma) \cap \normalgraph{\Sigma(y,S)}{\mathbf{f}_{S(y)}} \setminus \supp(||V||) \subset \cup_{j \in \mathbb{N}} B_{d_g}\left(z_j,5t_j\right)$, by \Cref{equa5statement:lemm:lipschitz:section:DensityEstimatesApp,equa56proof:lemm:lipschitz:section:DensityEstimatesApp,equa61proof:lemm:lipschitz:section:DensityEstimatesApp}, we have that
\footnotesize
\begin{align}\label{equa62proof:lemm:lipschitz:section:DensityEstimatesApp}
&\dfrac{\mathcal{H}^m_{d_g}\mres B_{d_g}(y,2\delta\sigma)\left(\normalgraph{\Sigma(y,S(y))}{\mathbf{f}_{S(y)}} \setminus \supp(||V||)\right)}{\sigma^m} \leq \nonumber \\
&\leq \dfrac{5^m3^me^{\ell(\rho)}}{2^{m-3}}\left(2 + \dfrac{n^2 K 5^2}{6(m+2)}\right) \left(\dfrac{\mathbf{C_2}}{L^2}E(y,\sigma,S,V)  +  E(y,\sigma,S,V)\right).
\end{align}
\normalsize

Since $0 < L \leq 1$ and we can choose $\delta_0 \defeq \delta_0(m) > 0$ such that $0 < \delta_0 \leq \sqrt{2\omega_m5^m}$, then $1 \leq \frac{\mathbf{C_2}}{L^2}$. By \Cref{equa62proof:lemm:lipschitz:section:DensityEstimatesApp} we have that
\begin{align*}
\dfrac{\mathcal{H}^m_{d_g}\mres B_{d_g}(y,2\delta\sigma)\left(\normalgraph{\Sigma(y,S(y))}{\mathbf{f}_{S(y)}} \setminus \supp(||V||)\right)}{\sigma^m} \leq \dfrac{\mathbf{C_5}}{L^2}  \delta^{\frac{1}{m+1}},
\end{align*}

which implies \Cref{equa6statement:lemm:lipschitz:section:DensityEstimatesApp}. To conclude the proof, we choose $\delta_0 \defeq \delta_0(m,n,p,K,\injec{r1}) > 0$ and  $\rho_0 \defeq \rho_0(m,K,\injec{r1},p)>0$ as the minimum value ensuring that all the above estimates hold.
\end{proof}
\begin{coro}\label{coro:lipschitz:section:DensityEstimatesApp}
Let $m < p < \infty$, $\xi \in M^n$ and $0 < L \leq 1$. There exists $\delta_0 \defeq \delta_0(m,n,p,K,\injec{r1}) > 0$ and  $\rho_0 \defeq \rho_0(m,K,\injec{r1},p)>0$ such that if $V \in \mathcal{AC}(\xi,\rho,\delta,p)$ for some $0 < \delta < (\delta_0L)^{2(m+1)}$ and for some $0 < \rho < \min\left\{\rho_0,\sqrt{\delta}\right\}$, then, for all $y \in \supp(||V||) \cap B_{d_g}\left(\xi,\frac{\delta\rho}{2}\right)$, for all $\sigma \in ]0,2\delta\rho]$ and for all $S(y) \in \Gr(m,T_yM^n)$ such that \Cref{equa3statement:coro:Remark212:section:DensityEstimatesApp}, \Cref{equa4statement:coro:Remark212:section:DensityEstimatesApp}, and 
\begin{equation}\label{equa1statement:coro:lipschitz:section:DensityEstimatesApp}
\sup\limits_{t \in ]0,(1-2\delta)\sigma]} \dfrac{1}{t^m}\int_{B_{d_g}(z,t)} \left|\left| P_{T_xV} - P_{T_x\Sigma(y,S(y))_x} \right|\right|_{\mathrm{HS}_{g_x}}^2 d||V||(x) \leq \delta_0^2L^2
\end{equation}

hold for all $z \in \supp(||V||) \cap B_{d_g}(y,2\delta\sigma)$, then there exists a Lipschitz function $\mathbf{f}_{S(y)} : B_{d_g}(y,\sigma) \cap \Sigma(y,S(y)) \to \mathbb{R}^{n-m}$ such that \Cref{equa3statement:lemm:lipschitz:section:DensityEstimatesApp,equa4statement:lemm:lipschitz:section:DensityEstimatesApp} hold and
\begin{equation}\label{equa4statement:coro:lipschitz:section:DensityEstimatesApp}
\supp(||V||) \cap B_{d_g}(y,\sigma) = \normalgraph{\Sigma(y,S(y))}{\mathbf{f}_{S(y)}} \cap B_{d_g}(y,\sigma).
\end{equation}
\end{coro}
\begin{proof}
Let $\delta_0 \defeq \delta_0(m,n,p,K,\injec{r1}) > 0$ and $\rho_0 \defeq \rho_0(m,p,K,\injec{r1}) > 0$ given by \Cref{lemm:lipschitz:section:DensityEstimatesApp}. Suppose that $V \in \mathcal{AC}(\xi,\rho,\delta,p)$, for some $0 < \delta < (\delta_0L)^{2(m+1)}$ and for some $0 < \rho < \min \left\{ \rho_0 , \sqrt{\delta} \right\}$. Consider $y \in \supp(||V||) \cap B_{d_g}\left(\xi,\frac{\delta\rho}{2}\right)$ and $\sigma \in ]0,2\delta\rho]$ arbitrary, and $S(y) \in \Gr(m,T_yM^n)$ such that \Cref{equa3statement:coro:Remark212:section:DensityEstimatesApp,equa4statement:coro:Remark212:section:DensityEstimatesApp,equa1statement:coro:lipschitz:section:DensityEstimatesApp} hold for all $z \in \supp(||V||) \cap B_{d_g}(y,2\delta\sigma)$. 

By \Cref{lemm:lipschitz:section:DensityEstimatesApp}, there a Lipschitz function $\mathbf{f}_{S(y)} : B_{d_g}(y,\sigma) \cap \Sigma(y,S(y)) \to \mathbb{R}^{n-m}$ such that \Cref{equa3statement:lemm:lipschitz:section:DensityEstimatesApp,equa4statement:lemm:lipschitz:section:DensityEstimatesApp} hold. 

Thus we have that the set $G$, defined in \Cref{equa2proof:lemm:lipschitz:section:DensityEstimatesApp}, contains the set $\supp(||V||) \cap B_{d_g}(y,2\delta\sigma)$, Which implies that
\begin{equation}\label{equa1proof:coro:lipschitz:section:DensityEstimatesApp}
\supp(||V||) \cap B_{d_g}(y,2\delta\sigma) \subset G \subset \normalgraph{\Sigma(y,S(y))}{\mathbf{f}_{S(y)}}.
\end{equation}

Suppose, by contradiction, that there exists $z_0 \in B_{d_g}(y,2\delta\sigma) \cap \normalgraph{\Sigma(y,S)}{\mathbf{F}(y,S)} \setminus \supp(||V||)$ and let $t \defeq \frac{3}{2}\dist_{g}^{\supp(||V||)}(z_0) > 0$. By \Cref{equa1proof:coro:lipschitz:section:DensityEstimatesApp} we have that
\begin{equation}\label{equa2proof:coro:lipschitz:section:DensityEstimatesApp}
B_{d_g}(z_0,t) \setminus \normalgraph{\Sigma(y,S)}{\mathbf{F}(y,S)} \subset B_{d_g}(z_0,t) \setminus \left(\supp(||V||) \cap B_{d_g}(y,2\delta\sigma)\right)
\end{equation}

By \Cref{equa55proof:lemm:lipschitz:section:DensityEstimatesApp,equa2proof:coro:lipschitz:section:DensityEstimatesApp}, we have that 
\begin{equation}\label{equa3proof:coro:lipschitz:section:DensityEstimatesApp}
t^m \leq \dfrac{3^me^{\ell(\rho)}}{2^{m-3}\omega_m}\sigma^{m} E(y,\sigma,S(y),V).
\end{equation}

By \Cref{equa3proof:coro:lipschitz:section:DensityEstimatesApp,equa4statement:coro:Remark212:section:DensityEstimatesApp}, we have that
\begin{align*}
t^m \leq \left(\dfrac{3^me^{\ell(\rho)}\mathbf{C}\delta_0^2L^2}{2^{m-3}\omega_m}\right)\sigma^{m}, 
\end{align*}
where $\mathbf{C} = \mathbf{C}(m,n,p,K,\injec{r1}) > 0$. Since $\sigma \in ]0,2\delta\rho]$ is arbitrary, we conclude that for all $\sigma^m < \frac{2^{m-3}\omega_m}{3^me^{\ell(\rho)}\mathbf{C}\delta_0^2L^2}\frac{t^m}{2}$ we have a contradiction, because it implies that $t^m < \frac{t^m}{2}$ and $t > 0$. 

Therefore, there exists $\sigma \in ]0,2\delta\rho]$ such that $B_{d_g}(y,2\delta\sigma) \cap \normalgraph{\Sigma(y,S)}{\mathbf{F}(y,S)} \setminus \supp(||V||) = \emptyset$ which implies that
\begin{equation}\label{equa4proof:coro:lipschitz:section:DensityEstimatesApp}
B_{d_g}(y,2\delta\sigma) \cap \normalgraph{\Sigma(y,S)}{\mathbf{F}(y,S)} \subset \supp(||V||).
\end{equation}

Thus \Cref{equa1proof:coro:lipschitz:section:DensityEstimatesApp,equa4proof:coro:lipschitz:section:DensityEstimatesApp} implies \Cref{equa4statement:coro:lipschitz:section:DensityEstimatesApp}.
\end{proof}
\subsection{Harmonic Approximation, Tilt-excess decay and Main regularity theorem}\label{sec:HarmTiltMain}
The proof of Allard's Interior $\varepsilon$-Regularity Theorem 
(\Cref{maintheocomplete}) rests on three main ingredients: the Lipschitz approximation, the monotonicity formula, and the Caccioppoli inequality, all of which were established in the previous sections. Having these in hand, the proof follows the structure of \cite{Simon}*{Theorem 5.2, p. 146}. 

The previous sections required more significant adaptations from the Euclidean proofs due to the curvature of the ambient space. Here, this is no longer the case, and we only point out the main differences with respect to \cite{Simon}*{Theorem 5.2, p. 146}.

In the Harmonic Approximation, the main difference is that the Mean-Value property used in the Euclidean setting must be replaced by the Schauder Interior Estimates for elliptic operators.

In the Tilt-Excess Decay Theorem, the proof proceeds analogously to the Euclidean one. The main difference appears in the application of the Lipschitz approximation: since we work with normal graphs rather than Euclidean graphs, the necessary adjustments follow the same analysis carried out in 
\Cref{section:lipschitzapp}.

The proof of \Cref{maintheocomplete} then follows closely \cite{Simon}*{Theorem 5.2, p. 146}, with the same differences already identified in the Tilt-Excess Decay Theorem. We close this section with \Cref{maintheocomplete} and some remarks on the proof.
\begin{theo}\label{maintheocomplete}
Let $m,n \in \mathbb{N}$ and $p \in \mathbb{R}$ such that $1 \leq m < n$ and $m < p < \infty$; in case $m=1$, we require that $p \geq 2$, $(M^n,g)$ an $n$-dimensional complete connected Riemannian manifold with metric $g$ of class $\C^2$, $\xi \in M^n$, and $\varepsilon \in ]0,1[$. There exist $\delta_0 \defeq \delta_0(m,n,p,K,\injec{r1},\varepsilon) > 0$ and $\rho_0 \defeq \rho_0(m,n,p,K,\injec{r1}) > 0$ such that, if $V \in \mathbf{V}_m(M^n)$ satisfies that
\begin{equation}\label{equa1maintheo}
\Theta^m(||V||,x) \geq d > 0,
\end{equation}

for $||V||$-a.e. $x \in B_{d_g}(\xi,\rho)$,
\begin{equation}\label{equa2maintheo}
\dfrac{||V||(B_{d_g}(\xi,\rho))}{\omega_m\rho^m} \leq d(1 + \delta).
\end{equation}

and
\begin{equation}\label{equa3maintheo}
\left|\delta_gV(X)\right| \leq \dfrac{d^{\frac{1}{p}}\delta}{\rho^{1-\frac{m}{p}}} \left|\left| X \right|\right|_{L^{\frac{p}{p-1}}(M^n,||V||)},
\end{equation}
for all $X \in \mathfrak{X}^1_c(M^n)$ such that $\supp(X) \subset B_{d_g}(\xi,\rho)$, for some $0 < \delta < \delta_0$, for some $0 < d < \infty$, and for some $0 < \rho < \min\{\rho_0,\sqrt{\delta}\}$, then there exist $T(\xi) \in \Gr(m,T_{\xi}M^n)$ and a function $\mathbf{f}_{T(\xi)}:B_{d_g}(\xi,\varepsilon\rho) \cap \Sigma(\xi,T(\xi)) \to \mathbb{R}^{n-m}$ such that $\mathbf{f}_{T(\xi)}$ is of class $\C^{1,1-\frac{m}{p}}$, $\mathbf{f}_{T(\xi)}(\xi) = 0$, ${d\mathbf{f}_{T(\xi)}}_{\xi} \equiv 0$, 
\begin{equation*}
\supp(||V||) \cap B_{d_g}(\xi,\varepsilon\rho) = \normalgraph{\Sigma(\xi,T(\xi))}{\mathbf{f}_{T(\xi)}} \cap B_{d_g}(\xi,\varepsilon\rho),
\end{equation*}

and the estimate of the scaling invariant $C^{1,1-\frac{m}{p}}$-norm of $\mathbf{f}_{T(\xi)}$ depend only on $m,n,p,K,\injec{r1}$, and $\varepsilon$.
\end{theo}
Two remarks on the proof of \Cref{maintheocomplete} are in order. First, the conditions in \Cref{maintheocomplete} involve a parameter $d$, as in \Cref{equa1maintheo,equa2maintheo,equa3maintheo}. However, since the density of a varifold can always be modified without changing its support, it suffices to consider varifolds in $\mathcal{AC}(\xi,\rho,\delta,p)$, i.e., the case $d = 1$. Second, for a fixed $\xi \in M^n$, the continuity of $\sec_g$ and $r_H(M^n,g)$ --- which follows from $g$ being of class $\mathcal{C}^2$ --- ensures the existence of $\rho_0 > 0$ such that
\begin{displaymath}
|\sec_g(x)| \leq K \quad \text{and} \quad r_H(M^n,g)(x) > 0
\end{displaymath}
for all $x \in B_{d_g}(\xi, \rho_0)$ and some $K \geq 0$. With these 
observations in place, the proof follows from the theory developed in the previous sections.

\section{Allard's Interior \texorpdfstring{$\varepsilon$}{E}-Regularity Theorem in Alexandrov spaces}\label{sec:metric}
We recall some basic definitions and properties of Alexandrov spaces. For synthetic notions of curvature and definitions of $\mathrm{CAT}$ and $\mathrm{CBB}$ spaces, we refer to \cites{alexander2019invitation,kapovitch2020cd} and the references therein.
\begin{defi}
Let $(X,d)$ be a complete locally compact metric space, and let $\overline{\mathcal{K}}(X,d), \underline{\mathcal{K}}(X,d) \in \mathbb{R}$ satisfy $\underline{\mathcal{K}}(X,d) \leq \overline{\mathcal{K}}(X,d)$. We say that:
\begin{enumerate}
    \item $(X,d)$ is an Alexandrov space with curvature bounded above by $\overline{\mathcal{K}}(X,d)$ if it is $\mathrm{CAT}(\overline{\mathcal{K}}(X,d))$, and with curvature bounded below by $\underline{\mathcal{K}}(X,d)$ if it is $\mathrm{CBB}(\underline{\mathcal{K}}(X,d))$;

    \item $(X,d,\mathcal{H}^n_d)$ is non-collapsed if there exists a constant $V_0(X,d)>0$ such that $\inf_{x\in X}\mathcal{H}^n_d(B_d(x,1)) \geq V_0(X,d)$.
\end{enumerate}
\end{defi}
\begin{rema}\label{rmk:Alex}
We recall several properties of Alexandrov spaces $(X,d)$ with curvature bounded both from above and below that will be used repeatedly throughout this section:
\begin{enumerate}
    \item\label{rmk:AlexC1alpha}
    $(X,d)$ admits a $C^{3,\alpha}$ differentiable structure together with a $C^{1,\alpha}$ Riemannian metric $g$, for every $\alpha\in(0,1)$, whose induced distance coincides with $d$; see \cite{Nikolaev-Berestovskij}*{Theorem 14.1}.

    \item\label{rmk:Alexapproximants}
    The $C^{1,\alpha}$ structure can be approximated by smooth Riemannian metrics; see \cite{Nikolaev89}*{Theorem 3.1}. More precisely, there exists a sequence of smooth metrics $g_i\in C^\infty$ on $X$ converging to $g$ in the $C^{1,\alpha}$ topology. If $d_i$ denotes the distance induced by $g_i$, then each $(X,d_i)$ is $\mathrm{CAT}(\overline{\mathcal{K}}(X,d_i))$ and $\mathrm{CBB}(\underline{\mathcal{K}}(X,d_i))$, with
    \begin{displaymath}
        \underline{\mathcal{K}}(X,d) \leq \liminf_{i\to\infty}\underline{\mathcal{K}}(X,d_i) \leq \limsup_{i\to\infty} \overline{\mathcal{K}}(X,d_i) \leq \overline{\mathcal{K}}(X,d).
    \end{displaymath}

    \item\label{rmk:AlexInjHarm}
   By \cite{CheegerGromovTaylor82}*{Theorem 4.7, p.\,47}, $(X,d)$ admits a positive lower bound on its injectivity radius $\inj(X,d)$, depending only on $n$, $\underline{\mathcal{K}}(X,d)$, $\overline{\mathcal{K}}(X,d)$, and $V_0(X,d)$. Consequently, by \cite{Anderson}*{Main Lemma 2.2, p.\,433}, it also admits a positive lower bound on its harmonic radius $r_H(X,d)$, depending only on $n$, $\underline{\mathcal{K}}(X,d)$, $\overline{\mathcal{K}}(X,d)$, and $V_0(X,d)$.
\end{enumerate}
\end{rema}
We close this section by stating a complete version of \Cref{secondthm}. Unlike the statement given in the introduction, we make explicit here that $\supp(||V||)$ is not merely a submanifold on a neighborhood of $\xi$, but is in fact the normal graph of a $\mathcal{C}^{1,1-\frac{m}{p}}$ function --- the precise analogue of the classical Euclidean result. We also state the precise dependence of all constants on the geometric data.
\begin{theo}\label{secondthmcomplete}
Let $m,n \in \mathbb{N}$ and $p \in \mathbb{R}$ such that $1 \leq m < n$ and $m < p < \infty$; in case $m=1$, we require that $p \geq 2$, $(X,g)$ be an $n$-dimensional non-collapsed Alexandrov space with curvature bounded both from above and below by $\overline{\mathcal{K}}(X,g),\underline{\mathcal{K}}(X,g)$, respectively, $\xi \in X$, and $\varepsilon \in ]0,1[$. Then there exist $\delta_0 \defeq \delta_0(m,n,p,\underline{\mathcal{K}}(X,g),\overline{\mathcal{K}}(X,g),V_0(X,g),\varepsilon) > 0$ and $\rho_0 \defeq \rho_0(m,n,p,\underline{\mathcal{K}}(X,g),\overline{\mathcal{K}}(X,g),V_0(X,g)) > 0$ such that, if $V \in \mathbf{V}_m(X)$ satisfies 
\begin{equation}
\Theta^m(||V||,x) \geq d > 0,
\end{equation}

for $||V||$-a.e. $x \in B_{d_g}(\xi,\rho)$,
\begin{equation}
\dfrac{||V||(B_{d_g}(\xi,\rho))}{\omega_m\rho^m} \leq d(1 + \delta),
\end{equation}

and
\begin{equation}
\left|\delta_gV(Y)\right| \leq \dfrac{d^{\frac{1}{p}}\delta}{\rho^{1-\frac{m}{p}}} \left|\left| Y \right|\right|_{L^{\frac{p}{p-1}}(X,||V||)},
\end{equation}

for all $Y \in \mathfrak{X}^1_c(X)$ such that $\supp(Y) \subset B_{d_g}(\xi,\rho)$, for some $0 < \delta < \delta_0$, for some $0 < d < \infty$, and for some $0 < \rho < \min\{\rho_0,\sqrt{\delta}\}$, then there exist $T(\xi) \in \Gr(m,T_{\xi}X)$ and a function $\mathbf{f}_{T(\xi)}:B_{d_g}(\xi,\varepsilon\rho) \cap \Sigma(\xi,T(\xi)) \to \mathbb{R}^{n-m}$ such that $\mathbf{f}_{T(\xi)}$ is of class $\C^{1,1-\frac{m}{p}}$, $\mathbf{f}_{T(\xi)}(\xi) = 0$, ${d\mathbf{f}_{T(\xi)}}_{\xi} \equiv 0$, 
\begin{equation*}
\supp(||V||) \cap B_{d_g}(\xi,\varepsilon\rho) = \normalgraph{\Sigma(\xi,T(\xi))}{\mathbf{f}_{T(\xi)}} \cap B_{d_g}(\xi,\varepsilon\rho),
\end{equation*}
and the estimate of the scaling invariant $C^{1,1-\frac{m}{p}}$-norm of $\mathbf{f}_{T(\xi)}$ depend only on $m,n,p,\underline{\mathcal{K}}(X,d),\overline{\mathcal{K}}(X,d),V_0(X,d)$ and $\varepsilon$.
\end{theo}
\begin{proof}[Proof of \Cref{secondthmcomplete}]
Let $(X,d)$ be a metric space with bounded curvature and $(X,g_t)$ a sequence of infinitely differentiable Riemannian manifolds given by \Cref{rmk:Alexapproximants} of \Cref{rmk:Alex}, such that $(X,g_t)$ converge to $(X,g)$ with respect to the $\C^{1,\alpha}$ topology.

By \Cref{maintheocomplete} we have that there exist $\delta_0^t \defeq \delta_0^t(m,n,p,\overline{\mathcal{K}}(X,d_t),\underline{\mathcal{K}}(X,d_t),r_H(X,d_t),\varepsilon) > 0$\\ and $\rho_0^t \defeq \rho_0^t(m,p,\overline{\mathcal{K}}(X,d_t),\underline{\mathcal{K}}(X,d_t),r_H(X,d_t)) > 0$. By the uniformly dependence of the parameter by the proofs obtained here and \Cref{rmk:Alexapproximants} of \Cref{rmk:Alex}, we can define
\begin{displaymath}
\delta_0 \defeq \delta_0(m,n,p,\overline{\mathcal{K}}(X,d),\underline{\mathcal{K}}(X,d),r_H(X,d),\varepsilon) \defeq \inf\limits_{t \leq t_0} \delta_0^t > 0
\end{displaymath}

and
\begin{displaymath}
\rho_0 \defeq \rho_0(m,n,p,\overline{\mathcal{K}}(X,d),\underline{\mathcal{K}}(X,d),r_H(X,d)) \defeq \inf\limits_{t \leq t_0} \rho_0^t > 0,
\end{displaymath}

where $t_0$ is chosen so that $r_H(X,d_t)$, $\overline{\mathcal{K}}(X,d_t)$, and $\underline{\mathcal{K}}(X,d_t)$ are uniformly bounded over $t\leq t_0$, ensuring the infimum above are finite. Suppose that $V \in \mathbf{V}_m(X)$ satisfies
\begin{equation}\label{equa1:secondthm}
\Theta^m(||V||,x) = \lim\limits_{r \downarrow 0}\dfrac{||V||(B_{d_g}(x,r))}{\omega_mr^m} \geq 1,
\end{equation}

for $||V||$-a.e. $x \in B_{d_g}(\xi,\rho)$,
\begin{equation}\label{equa2:secondthm}
\dfrac{||V||(B_{d_g}(\xi,\rho))}{\omega_m\rho^m} \leq 1 + \delta,
\end{equation}

and
\begin{equation}\label{equa3:secondthm}
\left|\delta_gV(Y)\right| \leq \dfrac{\delta}{\rho^{1-\frac{m}{p}}} \left|\left| Y \right|\right|_{L^{\frac{p}{p-1}}(X,||V||)},
\end{equation}

for all $Y \in \mathfrak{X}^0_c(X)$ such that $\supp(Y) \subset B_{d_g}(\xi,\rho)$, for some $0 < \delta < \delta_0$ and $0 < \rho < \min\{\rho_0,\sqrt{\delta}\}$.

Now note that since $V$ satisfies \Cref{equa1:secondthm,equa2:secondthm,equa3:secondthm}, by \Cref{rema:rectifiableAC}, we have that $V \in \mathbf{RV}_m(B_{d_g}(\xi,\delta\rho))$, that is, there exist a countable $m$-rectifiable set $\Gamma \subset M^n$ and $\theta \in L^1_{loc}(\Gamma,]0,\infty[)$ such that $V = V(\Gamma,\theta,g)$.

Secondly, note that since $|| g_t - g ||_{\C^0} \to 0$, when $t \downarrow 0$, we have, for all $t$ sufficiently small, for all $y \in X$, and for all $r > 0$, that
\begin{equation}\label{equa7:secondthm}
B_{d_{g_t}}\left(y,\frac{r}{1+O(t)}\right) \subset B_{d_g}(y,r) \subset B_{d_{g_t}}(y,(1+O(t))r).
\end{equation}

Therefore, by \Cref{equa7:secondthm}, we obtain, for $||V||$-a.e. $x \in B_{d_g}(\xi,\rho)$ and for all $t$ sufficiently small, that
\begin{align}\label{equa8:secondthm}
\Theta^m(||V||,x) \left(1 - O(t)\right) &\leq \liminf\limits_{r \downarrow 0} \frac{||V||(B_{d_{g_t}}(x,r))}{\omega_mr^m} \leq \nonumber\\
&\leq \limsup\limits_{r \downarrow 0} \frac{||V||(B_{d_{g_t}}(x,r))}{\omega_mr^m} \leq \Theta^m(||V||,x)\left(1 + O(t)\right).
\end{align}

By \Cref{equa8:secondthm} we have, for $||V||$-a.e. $x \in B_{d_g}(\xi,\rho)$ and for all $t$ sufficiently small, that
\begin{displaymath}
\left|  \limsup\limits_{r \downarrow 0} \frac{||V||(B_{d_{g_t}}(x,r))}{\omega_mr^m} -  \liminf\limits_{r \downarrow 0} \frac{||V||(B_{d_{g_t}}(x,r))}{\omega_mr^m} \right| < O(t),
\end{displaymath}
which implies that $\Theta^m_{t}(||V||,x) \defeq \lim_{r \downarrow 0}\frac{||V||(B_{d_{g_t}}(x,r))}{\omega_mr^m}$ exists for $||V||$-a.e. $x \in B_{d_g}(\xi,\rho)$ and for all $t$ sufficiently small. Furthermore, by \Cref{equa8:secondthm,equa1:secondthm}, we also have that
\begin{displaymath}
\Theta^m_t(||V||,x) \geq \Theta^m(||V||,x) \left(1 - O(t)\right) \geq 1 - O(t),
\end{displaymath}

for $||V||$-a.e. $x \in B_{d_g}(\xi,\rho)$ and for all $t$ sufficiently small. Now define 
\begin{displaymath}
\widetilde{V_t} \defeq V\left(\Gamma,\frac{\theta}{1-O(t)},g\right)
\end{displaymath}

and note that for $\tilde{\rho} < \rho$ we have
\begin{equation}\label{equa9:secondthm}
\Theta^m_t(||\widetilde{V_t}||,x) = \frac{\Theta^m_t(||V||,x)}{1-O(t)} \geq 1,
\end{equation}

for $||\widetilde{V}||$-a.e. $x \in B_{d_{g_t}}(\xi,\tilde{\rho})$ and for all $t$ sufficiently small. Take $\tilde{\rho} < \frac{\rho}{1+O(t)}$ for all $t$ sufficiently small and $|\rho - \tilde{\rho}|$ arbitrarily small. By \Cref{equa2:secondthm} we have that there exists $1 > \tilde{\delta} > \delta$ such that $|\tilde{\delta} - \delta|$ is arbitrarily small and
\begin{equation}\label{equa10:secondthm}
\frac{||\widetilde{V_t}||(B_{d_{g_t}}(\xi,\tilde{\rho}))}{\omega_m\tilde{\rho}^m} \leq \frac{1}{1+O(t)} \frac{||V||(B_{d_{g_t}}(\xi,\tilde{\rho}))}{\omega_m\tilde{\rho}^m} \leq \frac{1}{1-O(t)} \frac{\rho^m}{\tilde{\rho}^m} (1 + \delta) \leq 1 + \tilde{\delta}.
\end{equation}

Note that since $|| g_t - g ||_{\C^0} \to 0$ and $|| \partial g_t - \partial g ||_{\C^0} \to 0$, when $t \downarrow 0$, we have, for all $t$ sufficiently small, for all $y \in X$, for all $S \in \Gr(m,T_yX)$ and all $Y \in \mathfrak{X}^0_c(X)$, that
\tiny
\begin{align}\label{equa11:secondthm}
\left| \diver^{g_t}_S(Y)(y) - \diver^g_S(Y)(y) \right| &= \left|\sum\limits_{i=1}^m \sum_{j=1}^m g_t^{ij}(y)\left\langle \nabla^{g_t}_{\tau_i}Y (y) ,\tau_j \right\rangle_{g_{t_y}} - g^{ij}\left\langle \nabla^{g}_{\tau_i}Y (y) ,\tau_j \right\rangle_{g_{y}}\right| \nonumber \\
&= \left|\sum\limits_{i=1}^m \sum_{j=1}^m \sum_{k=1}^n g_t^{ij}(y)g_{t,kj}(y)\left(\frac{\partial X^k}{\partial \tau_i}(y) +  \sum_{l=1}^n X^l(y) \Gamma_{t,il}^k(y)\right) - g^{ij}g_{kj}(y)\left(\frac{\partial X^k}{\partial \tau_i}(y) +  \sum_{l=1}^n X^l(y) \Gamma_{il}^k(y)\right)\right| \nonumber \\
&= \left|\sum_{i=1}^m \sum_{l=1}^n X^l(y) \left(\Gamma_{t,il}^i(y) - \Gamma_{il}^i(y) \right) \right| \nonumber \\
&\leq \max_{i,l} \left| \Gamma_{t,il}^i(y) - \Gamma_{il}^i(y) \right| m\sqrt{n} ||X(y)||_{\euclid} \nonumber \\
&= O(t) ||X(y)||_{g_y},
\end{align}
\normalsize

where $\{\tau_1, \ldots, \tau_m\}$ is a basis for $S$ and $\{\tau_1,\ldots,\tau_n\}$ a basis for $T_yX$.

For all $Y \in \mathfrak{X}^0_c(X)$ such that $\supp(Y) \subset B_{d_g}(\xi,\rho)$ and for all $t$ sufficiently small, by Hölder inequality, \Cref{equa2:secondthm,equa3:secondthm,equa11:secondthm} we have that 
\begin{align}\label{equa12:secondthm}
|\delta_{g_t}\widetilde{V_t}(Y)| &\leq \frac{1}{1-O(t)}|\delta_gV(Y)| + \left|\int_{B_{d_g}(\xi,\rho)} O(t)||Y(y)||_{g_y} d||V_t||(y)\right| \nonumber \\
&\leq \frac{\delta}{\rho^{1-\frac{m}{p}}} ||Y||_{L^{\frac{p}{p-1}}(X,||V_t||)} + O(t)||Y||_{L^{\frac{p}{p-1}}(X,||V_t||)} \nonumber \\
&\leq \frac{\tilde{\delta}}{\tilde{\rho}^{1-\frac{m}{p}}} ||Y||_{L^{\frac{p}{p-1}}(X,||V_t||)}.
\end{align}

Therefore, we have proven, in \Cref{equa9:secondthm,equa10:secondthm,equa12:secondthm}, that $\widetilde{V_t} \in \mathcal{AC}(\xi,\tilde{\rho},\tilde{\delta},p)$, for some $0 < \tilde{\delta} < \delta_0^t$ and $0 < \tilde{\rho} < \rho_0^t$, for $t$ arbitrarily small. Thus, by \Cref{maintheocomplete} there exist $T(\xi) \in \Gr(m,T_{\xi}X)$ and a function $\mathbf{f}_{T(\xi)}:B_{d_{g_t}}(\xi,\varepsilon\Tilde{\rho}) \cap \Sigma(\xi,T(\xi)) \to \mathbb{R}^{n-m}$ such that $\mathbf{f}_{T(\xi)}$ is of class $\C^{1,1-\frac{m}{p}}$, $\mathbf{f}_{T(\xi)}(\xi) = 0$, ${d\mathbf{f}_{T(\xi)}}_{\xi} \equiv 0$, 
\begin{equation}\label{equa13:secondthm}
\supp(||\widetilde{V}_t||) \cap B_{d_{g_t}}(\xi,\varepsilon\Tilde{\rho}) = \normalgraph{\Sigma(\xi,T(\xi))}{\mathbf{f}_{T(\xi)}} \cap B_{d_{g_t}}(\xi,\varepsilon\Tilde{\rho}),
\end{equation}

and the estimate of the scaling invariant $C^{1,1-\frac{m}{p}}$-norm of $\mathbf{f}_{T(\xi)}$ depend only on $m,n,p,\overline{\mathcal{K}}(X,d_t),\underline{\mathcal{K}}(X,d_t)$, $r_H(X,d_t)$ and $\varepsilon$

Since $\supp(||\widetilde{V}_t||) = \supp(||V||)$ by construction, taking limits as $t \downarrow 0$ in the scaling invariant $\mathcal{C}^{1,1-\frac{m}{p}}$-norm of $\mathbf{f}_{T(\xi)}$ --- which is possible by the uniform dependence of the constants on the approximation parameter --- yields the desired conclusion, up to the dependence on $r_H(X,d)$ in place of $V_0(X,d)$. The latter is 
addressed by \Cref{rmk:AlexInjHarm} of \Cref{rmk:Alex}, which allows one to replace the lower bound on $r_H(X,d)$ by the non-collapsing constant $V_0(X,d)$, yielding precisely the dependence stated in \Cref{secondthmcomplete}.
\end{proof}
\bibliography{Bibliography}

@article {Gromov17,
    AUTHOR = {Gromov, Misha},
     TITLE = {Geometric, algebraic, and analytic descendants of {N}ash
              isometric embedding theorems},
   JOURNAL = {Bull. Amer. Math. Soc. (N.S.)},
  FJOURNAL = {American Mathematical Society. Bulletin. New Series},
    VOLUME = {54},
      YEAR = {2017},
    NUMBER = {2},
     PAGES = {173--245},
      ISSN = {0273-0979,1088-9485},
   MRCLASS = {58D10 (58-02 58Jxx)},
  MRNUMBER = {3619725},
MRREVIEWER = {Fr\'ed\'eric\ Robert},
       DOI = {10.1090/bull/1551},
       URL = {https://doi.org/10.1090/bull/1551},
}

@article {Anderson,
    AUTHOR = {Anderson, Michael T.},
     TITLE = {Convergence and rigidity of manifolds under {R}icci curvature
              bounds},
   JOURNAL = {Invent. Math.},
  FJOURNAL = {Inventiones Mathematicae},
    VOLUME = {102},
      YEAR = {1990},
    NUMBER = {2},
     PAGES = {429--445},
      ISSN = {0020-9910,1432-1297},
   MRCLASS = {53C23 (53C21 58D27)},
  MRNUMBER = {1074481},
MRREVIEWER = {Gudlaugur\ Thorbergsson},
       DOI = {10.1007/BF01233434},
       URL = {https://doi.org/10.1007/BF01233434},
}

@article {Jacobowitz72,
    AUTHOR = {Jacobowitz, Howard},
     TITLE = {Implicit function theorems and isometric embeddings},
   JOURNAL = {Ann. of Math. (2)},
  FJOURNAL = {Annals of Mathematics. Second Series},
    VOLUME = {95},
      YEAR = {1972},
     PAGES = {191--225},
      ISSN = {0003-486X},
   MRCLASS = {53C40 (58C15)},
  MRNUMBER = {307127},
MRREVIEWER = {C.\ S.\ Houh},
       DOI = {10.2307/1970796},
       URL = {https://doi.org/10.2307/1970796},
}

@incollection {Nikolaev89,
    AUTHOR = {Nikolaev, I. G.},
     TITLE = {The closure of the set of classical {R}iemannian spaces},
 BOOKTITLE = {Problems in geometry, {V}ol. \ 21 ({R}ussian)},
    SERIES = {Itogi Nauki i Tekhniki},
     PAGES = {43--46, 216},
      NOTE = {Translated in J. Soviet Math.\ {\bf 55} (1991), no.\ 6,
              2100--2115},
 PUBLISHER = {Akad. Nauk SSSR, Vsesoyuz. Inst. Nauchn. i Tekhn. Inform.,
              Moscow},
      YEAR = {1989},
   MRCLASS = {53C20 (53C23 54E45)},
  MRNUMBER = {1027853},
MRREVIEWER = {Yi\ Bing\ Shen},
}

@incollection {Nikolaev-Berestovskij,
    AUTHOR = {Berestovskij, V. N. and Nikolaev, I. G.},
     TITLE = {Multidimensional generalized {R}iemannian spaces},
 BOOKTITLE = {Geometry, {IV}},
    SERIES = {Encyclopaedia Math. Sci.},
    VOLUME = {70},
     PAGES = {165--243, 245--250},
 PUBLISHER = {Springer, Berlin},
      YEAR = {1993},
      ISBN = {3-540-54701-0},
   MRCLASS = {53C20},
  MRNUMBER = {1263965},
       DOI = {10.1007/978-3-662-02897-1\_2},
       URL = {https://doi.org/10.1007/978-3-662-02897-1_2},
}

@article {Allard,
    AUTHOR = {Allard, William K.},
     TITLE = {On the first variation of a varifold},
   JOURNAL = {Ann. of Math. (2)},
  FJOURNAL = {Annals of Mathematics. Second Series},
    VOLUME = {95},
      YEAR = {1972},
     PAGES = {417--491},
      ISSN = {0003-486X},
   MRCLASS = {49F20},
  MRNUMBER = {0307015},
MRREVIEWER = {M. Klingmann},
       DOI = {10.2307/1970868},
       URL = {http://dx.doi.org/10.2307/1970868},
}

@article{Simon,
  title={Introduction to geometric measure theory},
  author={Simon, Leon},
  journal={NTU Lectures},
  year={2018}
}

@book {Folland,
    AUTHOR = {Folland, Gerald B.},
     TITLE = {Real analysis},
    SERIES = {Pure and Applied Mathematics (New York)},
   EDITION = {Second},
      NOTE = {Modern techniques and their applications,
              A Wiley-Interscience Publication},
 PUBLISHER = {John Wiley \& Sons, Inc., New York},
      YEAR = {1999},
     PAGES = {xvi+386},
      ISBN = {0-471-31716-0},
   MRCLASS = {00A05 (26-01 28-01 46-01)},
  MRNUMBER = {1681462},
}

@book {FedererGMT,
    AUTHOR = {Federer, Herbert},
     TITLE = {Geometric measure theory},
    SERIES = {Die Grundlehren der mathematischen Wissenschaften, Band 153},
 PUBLISHER = {Springer-Verlag New York, Inc., New York},
      YEAR = {1969},
     PAGES = {xiv+676},
   MRCLASS = {28.80 (26.00)},
  MRNUMBER = {0257325},
MRREVIEWER = {J. E. Brothers},
}

@book {EvansPDE,
    AUTHOR = {Evans, Lawrence C.},
     TITLE = {Partial differential equations},
    SERIES = {Graduate Studies in Mathematics},
    VOLUME = {19},
   EDITION = {Second},
 PUBLISHER = {American Mathematical Society, Providence, RI},
      YEAR = {2010},
     PAGES = {xxii+749},
      ISBN = {978-0-8218-4974-3},
   MRCLASS = {35-01},
  MRNUMBER = {2597943},
MRREVIEWER = {Diego\ M.\ Maldonado},
       DOI = {10.1090/gsm/019},
       URL = {https://doi.org/10.1090/gsm/019},
}

@book {Lee,
    AUTHOR = {Lee, John M.},
     TITLE = {Introduction to {R}iemannian manifolds},
    SERIES = {Graduate Texts in Mathematics},
    VOLUME = {176},
      NOTE = {Second edition of [ MR1468735]},
 PUBLISHER = {Springer, Cham},
      YEAR = {2018},
     PAGES = {xiii+437},
      ISBN = {978-3-319-91754-2; 978-3-319-91755-9},
   MRCLASS = {53-01 (53B20 53B30 53C20 53C21)},
  MRNUMBER = {3887684},
MRREVIEWER = {Robert J. Low},
}

@book {RudinAnalysis,
    AUTHOR = {Rudin, Walter},
     TITLE = {Real and complex analysis},
   EDITION = {Third},
 PUBLISHER = {McGraw-Hill Book Co., New York},
      YEAR = {1987},
     PAGES = {xiv+416},
      ISBN = {0-07-054234-1},
   MRCLASS = {00A05 (26-01 30-01 46-01)},
  MRNUMBER = {924157},
}

@book {ODETeschl,
    AUTHOR = {Teschl, Gerald},
     TITLE = {Ordinary differential equations and dynamical systems},
    SERIES = {Graduate Studies in Mathematics},
    VOLUME = {140},
 PUBLISHER = {American Mathematical Society, Providence, RI},
      YEAR = {2012},
     PAGES = {xii+356},
      ISBN = {978-0-8218-8328-0},
   MRCLASS = {34-01 (37-01 39-01)},
  MRNUMBER = {2961944},
MRREVIEWER = {Eleonora\ Catsigeras},
       DOI = {10.1090/gsm/140},
       URL = {https://doi.org/10.1090/gsm/140},
}

@book {Gray04,
    AUTHOR = {Gray, Alfred},
     TITLE = {Tubes},
    SERIES = {Progress in Mathematics},
    VOLUME = {221},
   EDITION = {Second},
      NOTE = {With a preface by Vicente Miquel},
 PUBLISHER = {Birkh\"{a}user Verlag, Basel},
      YEAR = {2004},
     PAGES = {xiv+280},
      ISBN = {3-7643-6907-8},
   MRCLASS = {53-02 (53-01 53B25 53C40)},
  MRNUMBER = {2024928},
       DOI = {10.1007/978-3-0348-7966-8},
       URL = {https://doi.org/10.1007/978-3-0348-7966-8},
}

@article {NarCalcVar,
    AUTHOR = {Nardulli, Stefano},
     TITLE = {The isoperimetric profile of a noncompact {R}iemannian
              manifold for small volumes},
   JOURNAL = {Calc. Var. Partial Differential Equations},
  FJOURNAL = {Calculus of Variations and Partial Differential Equations},
    VOLUME = {49},
      YEAR = {2014},
    NUMBER = {1-2},
     PAGES = {173--195},
      ISSN = {0944-2669},
   MRCLASS = {49Q20 (53A10 53C42 58A25 58E99)},
  MRNUMBER = {3148111},
MRREVIEWER = {Andrew Bucki},
       DOI = {10.1007/s00526-012-0577-1},
       URL = {https://doi.org/10.1007/s00526-012-0577-1},
}

@book{Almgren64,
  title={The Theory of Varifolds: A Variational Calculus in the Large for the K-dimensional Area Integrand},
  author={Almgren, F.J.},
  url={https://albert.ias.edu/20.500.12111/8025},
  year={1964},
  publisher={Institute for Advanced Study}
}

@book {Almgren66,
    AUTHOR = {Almgren, Jr., Frederick J.},
     TITLE = {Plateau's problem: {A}n invitation to varifold geometry},
 PUBLISHER = {W. A. Benjamin, Inc., New York-Amsterdam},
      YEAR = {1966},
     PAGES = {xii+74},
   MRCLASS = {53.04 (49.00)},
  MRNUMBER = {190856},
MRREVIEWER = {L.\ C.\ Young},
}

@misc{Fukuoka2006,
      title={Mollifier Smoothing of tensor fields on differentiable manifolds and applications to Riemannian Geometry}, 
      author={Ryuichi Fukuoka},
      year={2006},
      eprint={math/0608230},
      archivePrefix={arXiv},
      primaryClass={math.DG},
      url={https://arxiv.org/abs/math/0608230}, 
}

@book {Chavel93,
    AUTHOR = {Chavel, Isaac},
     TITLE = {Riemannian geometry---a modern introduction},
    SERIES = {Cambridge Tracts in Mathematics},
    VOLUME = {108},
 PUBLISHER = {Cambridge University Press, Cambridge},
      YEAR = {1993},
     PAGES = {xii+386},
      ISBN = {0-521-43201-4; 0-521-48578-9},
   MRCLASS = {53-02 (53Cxx)},
  MRNUMBER = {1271141},
MRREVIEWER = {Carolyn\ Gordon},
}

@book {Jost84,
    AUTHOR = {Jost, J\"urgen},
     TITLE = {Harmonic mappings between {R}iemannian manifolds},
    SERIES = {Proceedings of the Centre for Mathematical Analysis,
              Australian National University},
    VOLUME = {4},
 PUBLISHER = {Australian National University, Centre for Mathematical
              Analysis, Canberra},
      YEAR = {1984},
     PAGES = {iv+177},
      ISBN = {0-86784-403-5},
   MRCLASS = {58E20 (35J99 53C20)},
  MRNUMBER = {756629},
MRREVIEWER = {H.\ C. J. Sealey},
}

@article {HebeyHerzlich97,
    AUTHOR = {Hebey, E. and Herzlich, M.},
     TITLE = {Harmonic coordinates, harmonic radius and convergence of
              {R}iemannian manifolds},
   JOURNAL = {Rend. Mat. Appl. (7)},
  FJOURNAL = {Rendiconti di Matematica e delle sue Applicazioni. Serie VII},
    VOLUME = {17},
      YEAR = {1997},
    NUMBER = {4},
     PAGES = {569--605},
      ISSN = {1120-7183,2532-3350},
   MRCLASS = {53C21 (53C20)},
  MRNUMBER = {1620864},
MRREVIEWER = {Andrea\ Sambusetti},
}

@book{alexander2019invitation,
  title={An invitation to Alexandrov geometry: CAT (0) spaces},
  author={Alexander, Stephanie and Kapovitch, Vitali and Petrunin, Anton},
  year={2019},
  publisher={Springer}
}

@article{kapovitch2020cd,
  title={Cd meets cat},
  author={Kapovitch, Vitali and Ketterer, Christian},
  journal={Journal f{\"u}r die reine und angewandte Mathematik (Crelles Journal)},
  volume={2020},
  number={766},
  pages={1--44},
  year={2020},
  publisher={De Gruyter}
}

@ARTICLE{CheegerGromovTaylor82,
  title     = "Finite propagation speed, kernel estimates for functions of the
               Laplace operator, and the geometry of complete Riemannian
               manifolds",
  author    = "Cheeger, Jeff and Gromov, Mikhail and Taylor, Michael",
  journal   = "J. Differential Geom.",
  publisher = "International Press of Boston",
  volume    =  17,
  number    =  1,
  pages     = "15--53",
  month     =  jan,
  year      =  1982
}
\bibliographystyle{alpha}
\noindent
\begin{minipage}{\textwidth}
\begin{minipage}{.45\textwidth}
\footnotesize{Marcos Agnoletto,\\
\emph{CMCC, Federal University of ABC, Santo André, SP, Brazil,}\\
\emph{Email address:} marcos.forte@ufabc.edu.br,\\
Corresponding author.\\

Julio C. Correa Hoyos,\\
\emph{IME, Rio de Janeiro State University, Rio de Janeiro, RJ, Brazil,}\\
\emph{Email address:} julio.correa@ime.uerj.br.}
\end{minipage}
\hspace{0.05\textwidth}
\begin{minipage}{.45\textwidth}
\footnotesize{Márcio Fabiano da Silva,\\
\emph{CMCC, Federal University of ABC, Santo André, SP, Brazil,}\\
\emph{Email address:} marcio.silva@ufabc.edu.br.\\

Stefano Nardulli,\\
\emph{CMCC, Federal University of ABC, Santo André, SP, Brazil,}\\
\emph{Email address:} stefano.nardulli@ufabc.edu.br.}
\end{minipage}
\end{minipage}
\end{document}